%% file: main.tex
\documentclass[11pt]{report}
\usepackage[utf8]{inputenc}
\usepackage{graphicx}
\graphicspath{ {images/} }
\usepackage{caption}
\usepackage{subcaption}
\usepackage{float}
\usepackage[utf8]{inputenc}
\usepackage{amsfonts}
\usepackage{tikz-cd}
\usepackage[colorlinks=true]{hyperref}
\usepackage{amsmath}
\usepackage{xcolor}
\usepackage{amsthm}
\usepackage{pdflscape}
\usepackage{pgfplots}
\usepackage{mathrsfs}
\usepackage{amssymb}
\usepackage{nomencl}
\usepackage[width=150mm,top=35mm,bottom=25mm,bindingoffset=6mm]{geometry}
\usepackage{fancyhdr}
\usepackage{setspace}
\renewcommand{\headrulewidth}{0pt}
\newtheorem{thm}{Theorem}[section]
\newtheorem{cor}[thm]{Corollary}
\newtheorem{prop}[thm]{Proposition}
\newtheorem{lem}[thm]{Lemma}

\newtheorem{quest}[thm]{Question}

\newtheorem{defx}{Definition}

\theoremstyle{definition}
\newtheorem{defn}[thm]{Definition}

\newtheorem{exmp}[thm]{Example}

\newtheorem{remark}[thm]{Remark}
\tikzset{
  trim node/.default=1cm,
  trim node/.style={
    overlay,
    append after command={
      ([xshift={+#1}]\tikzlastnode.north west)
      ([xshift={+-#1}]\tikzlastnode.south east)}},
  down and trim/.default=1cm,
  down and trim/.style={
    yshift=-(\pgfmatrixcurrentcolumn-1)*1.5\baselineskip,
    trim node={#1}},
  downup and trim/.default=1cm,
  downup and trim/.style={
    yshift=iseven(\pgfmatrixcurrentcolumn) ? -1.5\baselineskip : 0pt,
    trim node={#1}},
  -|/.style={to path={-|(\tikztotarget)\tikztonodes}},
  |-/.style={to path={|-(\tikztotarget)\tikztonodes}},
  -| sl/.style={-|, xslant=-1},
  |- sl/.style={|-, xslant= 1},
  center picture/.style={
    trim left=(current bounding box.center),
    trim right=(current bounding box.center)}}

\theoremstyle{remark}

\newcommand{\R}{\mathbb{R}}                           
    \newcommand{\C}{\mathbb{C}}
         
       \newcommand{\Z}{\mathbb{Z}}
          \newcommand{\N}{\mathbb{N}}
          \newcommand{\U}{\mathcal{U}}
          \newcommand{\g}{\mathfrak{g}}
          
                \newcommand{\D}{\mathcal{D}}
                   \newcommand{\K}{\mathcal{K}}
                \newcommand{\G}{\mathcal{G}}
                
                \newcommand{\RKK}{\mathcal{R}KK}
                \newcommand{\hotimes}{\hat{\otimes}}
\usepackage[style= alphabetic]{biblatex}
\title{Index theory for Heisenberg elliptic and transversally Heisenberg elliptic operators from $KK$-theoretic viewpoint}
\author{Minjie Tian}
\date{\today}

\begin{document}
\pagenumbering{roman} 

\input{titlepage}

\input{abstract}

\chapter*{Dedication}
This thesis is dedicated to my parents, whose unwavering support and encouragement have been my guiding lights throughout this academic journey. To my friends, your camaraderie and understanding made the challenges more manageable and the successes more joyful.

\chapter*{Declaration}
I hereby declare that this thesis is entirely my own work. I affirm the following:
\begin{itemize}
    \item[(a)] This work has not been submitted for any other degree or academic qualification.
      \item[(b)] All sources of information and ideas used in this thesis are properly acknowledged.
        \item[(c)] Contributions made by others in the form of ideas, guidance, or any other support are duly recognized and appropriately cited.
          \item[(d)]  The research presented in this thesis adheres to the ethical standards and guidelines of academic integrity.
\end{itemize}

\chapter*{Acknowledgements}

I would like to express my sincere gratitude to my supervisor, Professor Kato Tsuyoshi, for his invaluable ideas and guidance for this research. Professor Gennadi Kasparov's patient leadership and Professor Nigel Higson’s valuable advice were instrumental in shaping this research. 

Special thanks to the Graduate Education Support Organization Program / Science and Technology Innovation Creation Fellowship for their generous support, enhancing the depth and scope of this work.

\tableofcontents

\doublespacing
\chapter*{Introduction}
\pagenumbering{arabic} 
\input{chapters/introduction}
\chapter*{Notations}
\input{chapters/notations}
\chapter{Preliminaries for H-elliptic theory}
\input{chapters/chapter01}

\chapter{Elliptic operators and transversally elliptic operators}
\input{chapters/chapter02}

\chapter{H-elliptic index theory}
\input{chapters/chapter03}

\chapter{Construction for transversally H-elliptic operators}
\input{chapters/chapter04}
\chapter{Index theorem for transversally H-elliptic operators}
\input{chapters/chapter05}

\chapter{Verification of the condition and basic examples}
\input{chapters/chapter06}

\appendix
\chapter{Appendix}

\input{chapters/appendix}

\printbibliography

\end{document}

%% file: titlepage.tex
\begin{titlepage}
    \begin{center}
        \vspace*{1cm}
        
        \Huge
        \textbf{Index theory for Heisenberg elliptic and transversally Heisenberg elliptic operators from $KK$-theoretic viewpoint}
        
        \vspace{0.5cm}
        \LARGE
        
        by\\
    
        \textbf{Minjie Tian}
        
        \vfill

        Geometry, Mathematics\\

        \vspace{1.8cm}

        \Large
        Kyoto University\\
        2024\\
        \vspace{1.0cm}

    \end{center}
    
\end{titlepage}

%% file: abstract.tex
\fancyhf{} 
\fancyhead[RO,R]{\thepage} 
\renewcommand{\headrulewidth}{0pt}

\begin{center}
    \Large
    \textbf{Index theory for Heisenberg elliptic and transversally Heisenberg elliptic operators from $KK$-theoretic viewpoint}
    
    \vspace{0.4cm}
    \large

    \textbf{Minjie Tian}
    
    \vspace{0.9cm}
    \textbf{Abstract}
\end{center}
This research comprehensively describes the basic theory of transversally Heisenberg elliptic operators, and investigates the index theory of Heisenberg elliptic and transversally Heisenberg elliptic operators from the perspective of $KK$-theory, applying Kasparov's methodology. Moreover, the analysis methodically examines specific conditions, with a focus on the Fourier transform of the nilpotent group $C^{\ast}$-algebra. We demonstrate enhanced methods for analyzing the hypoellipticity of operators, presenting a robust framework for defining and understanding transversal Heisenberg ellipticity in a $KK$-theoretic context. This work provides a solid foundation for future research into the properties of hypoelliptic differential operators in complicated manifolds.
\\\\
\noindent \textbf{Keywords.} Heisenberg elliptic, transversally Heisenberg elliptic, index theory, $KK$-theory\\\\

\noindent \textbf{e-mail:} \ \ \ \href{mailto:tian.minjie.23z@st.kyoto-u.ac.jp}{tian.minjie.23z@st.kyoto-u.ac.jp}\\

%% file: chapters/introduction.tex
Let $M$ be a smooth and closed manifold, $E, F \to M$  be vector bundles, and $P: C^{\infty}(M; E) \to C^{\infty}(M; F)$ be a differential operator over $M$. For elliptic operators, the spaces $\text{Ker}(P)$ and $\text{Ker}(P^{\ast})$ are both finite-dimensional, and the analytical index is defined as $\text{Ind}_a(P) = \dim \text{Ker}(P) - \dim \text{Coker}(P)\in \Z$. When a compact Lie group $G$ acts on $M$ and the operator $P$ is elliptic, the spaces $\text{Ker}(P)$ and $\text{Coker}(P)$ can be viewed as elements of $R(G)$, with the analytical index defined as $\text{Ind}_a(P) = [\text{Ker}(P)] - [\text{Coker}(P)] \in R(G)$. The Atiyah-Singer index theorem \cite{atiyah1968index2}\cite{atiyah1968index1}  states that in each case, the analytical index defined above is equal to the topological index, which only depends on the topological characteristics of $P$ and $M$.

For the case of transversally elliptic operators, Atiyah \cite{atiyah2006elliptic} introduced the idea of defining the analytical index as a distribution over the Lie group $G$. Berline and Vergne \cite{berline1996chern} introduced a cohomological index formula, equating the analytical index to the cohomological index for transversally elliptic operators.

 Kasparov \cite{kasparov2017elliptic}\cite{kasparov2022k} generalized it for locally compact manifolds under properly supported locally compact Lie group actions, framing a comprehensive theory on the index of elliptic and transversally elliptic operators through $KK$-theory. In this theory both the symbol class $[\sigma(P)]$ and the index class $[P]$ are represented by $KK$-classes, providing the index theorem through $KK$-products.

The study of index theory for hypoelliptic operators is challenging due to the simple condition. Additional conditions are required to define the analytical index accurately.

 Boutet de Monvel \cite{de1978index} gave the topological formula for the analytical index of Toeplitz operators. van Erp \cite{van2010atiyah1}\cite{van2010atiyah2} generalized it to arbitrary contact manifolds, based on the work by Epstein and Melrose \cite{epstein2004heisenberg}, for Heisenberg elliptic operators without vector bundle coefficients. Later van Erp and Baum \cite{baum2014k} generalized van Erp’s formula to allow vector bundle coefficients on contact manifolds. Their formula allows us to compute the index of some differential operators of H\"{o}rmander’s sum of squares of rank $2$. Androulidakis, Mohsen, and Yuncken \cite{androulidakis2022pseudodifferential} generalized the classical regularity theorem of
elliptic operators to maximally hypoelliptic differential operators, and  Mohsen \cite{mohsen2022index} established an index formula for $\ast$-maximally hypoelliptic differential operators on closed manifolds.

 On a closed manifold with filtration, or equivalently on a closed manifold endowed with a regular foliation, the index theorem for Heisenberg elliptic (later on, H-elliptic) operators can be reduced as a special case of Mohsen's maximally hypoelliptic index theorem. In the regular foliation case, the characteristic set defined in \cite{androulidakis2022pseudodifferential} is the entire groupoid $Gr(H)$. Hence, the $\ast$-maximally hypoelliptic condition is consistent with the H-ellipticity, which requires for each nontrivial unitary irreducible representation $\pi$, the action on the H-cosymbol $\sigma_H(P,x,\pi)$ defined by the filtration is bijective. We will reformulate the index theorem for H-elliptic operators from the $KK$-theoretic viewpoint in Section \ref{sec3.2} and \ref{sec3.3}.

In this research, we introduce transversally H-elliptic operators on filtered manifolds with $G$-action. The definition is inspired by Kasparov's tangent-Clifford symbol class $[\sigma^{tcl}(P)]$ in transversally elliptic index theory \cite{kasparov2022k}. Transversal H-ellipticity requires the Rockland condition's fulfillment in conjunction with a leafwise elliptic symbol class. Concretely, We define that:
\begin{defn}\label{def2}
    Let $M$ be a complete filtered Riemannian manifold and $G$ be a locally compact Lie group acting on $M$ properly, isometrically, and preserving the filtered structure. A pseudo-differential operator $P$ of $0$-order is called transversally H-elliptic if at each point $x\in M$, the multiplier term 
    $$(1-\sigma_H^2(P,x))(1+\sigma_H(\Delta_{G},x))^{-1}\in \mathcal{M}(E_x\otimes C^{\ast}(Gr(H)_x)),$$
    as an operator on $C^{\ast}(Gr(H)_x)$-module is compact, where $\Delta_G$ is the sub-Laplacian operator in $G$-orbit directions.
\end{defn}
In Chapter \ref{chap4}, We will prove that Definition \ref{def2} produces an appropriate tangent-Clifford symbol class $[\sigma^{tcl}_H(P)]$ and an index class $[P]$ for each transversally H-elliptic operator $P$. When $P$ is a transversally H-elliptic operator of $0$-order, there is an index formula in $KK$-homology in which $[\sigma^{tcl}_H(P)]$ and $[P]$ can be related using $KK$-products. Here is the main theorem of this paper:
\begin{thm}
Let $M$ be a complete filtered Riemannian manifold and $G$ be a locally compact Lie group isometrically and properly acting on $M$ and preserving the filtered structure. Let $P$ be a $G$-invariant transversally H-elliptic operator on $M$ of order $0$, then we have
 $$[P]=j^G([\sigma^{tcl}_H(P)])\otimes_{C^{\ast}(G,C^{\ast}(Gr(H))\hotimes_{C_0(M)}Cl_{\Gamma}(M))} [\mathcal{D}_{M,{\Gamma}}^{cl'}],$$
    in which $[\mathcal{D}_{M,{\Gamma}}^{cl'}]\in K^0(C^{\ast}(G,C^{\ast}(Gr(H))\hotimes_{C_0(M)}Cl_{\Gamma}(M)))$ indicates the Clifford-Dolbeault element which will be defined in Chapter \ref{chap5}, and the $KK$-product on the right-hand side is considered over
    \begin{equation*}
        \begin{aligned}
            KK(C^{\ast}(G, C_0(M)),C^{\ast}(G,C^{\ast}(Gr(H))\hotimes_{C_0(M)} Cl_{{\Gamma}}(M))\\
            \otimes_{C^{\ast}(G,C^{\ast}(Gr(H))\hotimes_{C_0(M)}Cl_{\Gamma}(M))} KK(C^{\ast}(G,C^{\ast}(Gr(H))\hotimes_{C_0(M)}Cl_{\Gamma}(M)),\C).
        \end{aligned}
    \end{equation*}
    \end{thm}
    Determining whether an arbitrary pseudo-differential operator meets the criteria outlined in our defined theory hinges on the Fourier transform theory for nilpotent group $C^{\ast}$-algebras. For an element $\phi\in C^{\ast}(\G)$, where $\G$ represents a nilpotent Lie group, its representations $\pi(\phi)$ constitute an operator-valued function on the spectrum $\widehat{\G}$, with some specific conditions. This theory extends to the Fourier transform of multiplier elements of the group $C^{\ast}$-algebra, transforming a multiplier of $C^{\ast}(\G)$ into an operator family parametrized by $\widehat{\G}$. As a result, a multiplier of $C^{\ast}(\G)$ is Fourier-transformed to a family of operators parametrized by the representation space $\widehat{\G}$, and the properties of this operator family reflect if the multiplier is an element of $C^{\ast}(\G)$ itself or not. Explicitly, for an element $e\in \mathcal{M}(C^{\ast}(\G))$, if we can verify that the Fourier transform $\mathcal{F}(e)\in \mathcal{F}(\mathcal{M}(C^{\ast}(\G)))$ actually lies in $\mathcal{F}(C^{\ast}(\G))$, then $e\in \mathcal{K}(C^{\ast}(\G))=C^{\ast}(\G)$. See the following diagram:
    $$\begin{tikzcd}
  &\mathcal{M}(C^{\ast}(\G)) \arrow[r, "\mathcal{F}"] & \mathcal{F}(\mathcal{M}(C^{\ast}(\G))) \\
     &C^{\ast}(\G)\arrow[r, "\mathcal{F}"]\arrow[u, phantom, sloped, "\subset"] & \mathcal{F}(C^{\ast}(\G))\arrow[u, phantom, sloped, "\subset"],
\end{tikzcd}$$
    Take $\G=Gr(H)$, we will employ this method to verify the compactness of the $C^{\ast}(Gr(H)_x)$-module multipliers:
    \begin{thm}\label{thm4}
         For each $x\in M$, suppose that $Q$ is a multiplier of the $C^{\ast}(Gr(H)_x)$ module, $Q\in \mathcal{M}(E_x\otimes C^{\ast}(Gr(H)_x))$. Then, $Q$ is compact , i.e., $$Q\in \text{End}(E_x)\otimes C^{\ast}(Gr(H))_x$$ if and only if the operator field defined by
    $$\phi=\{\pi_{\gamma}(Q)\}_{\gamma\in \widehat{Gr(H)}_x} \subseteq \text{End}(E_x)\otimes \mathit{l}^{\infty}(\widehat{Gr(H)}_x)$$
    lies in the sub-algebra $\text{End}(E_x)\otimes B^{\ast}(\widehat{Gr(H)}_x)$. The algebra $B^{\ast}(S)$ will be defined in Definition \ref{defn3}.
    \end{thm}
    
An illustrative example of transversally H-elliptic operators is given by the operator $P_{\gamma}=\Delta_{n-k}+i\gamma T$ on the Heisenberg group $M=\mathbb{H}_{n}$ with transition action by $\R^{2k}$ on the first $2k$-coordinates. Here, $\Delta_{n-k}=-\sum_{i=k+1}^n ({X_i}^2+Y_i^2)$ means the sub-Laplacian operator on the sub-manifold $M'\cong \mathbb{H}_{n-k}$ identified by coordinates $x_{k+1},y_{k+1},\dots, x_n, y_n, z$. The function $\gamma$ on $M$ is assigned such that $P_{\gamma}$ on $M'$ is H-elliptic, see \cite{baum2014k}. Utilizing Theorem \ref{thm4}, we can show that $P_{\gamma}$ satisfies the condition in Definition \ref{def2}, so it is transversal H-elliptic on $M$.

\newpage
\section*{Structure of the paper}
\begin{itemize}
    \item Chapter \ref{chap1} sets the stage with preliminaries on symbolic calculus on filtered manifolds and Kirillov's orbit method for Heisenberg calculus of differential and pseudo-differential operators. It also touches upon Connes' \cite{connes1994noncommutative} groupoid approach for H-elliptic cases.
    
    \item Chapter \ref{chap2} introduces the $KK$-theoretic framework of the index theorem for elliptic and transversally elliptic operators, as detailed in \cite{kasparov2022k}\cite{kasparov2017elliptic}.
    
    \item Chapter \ref{chap3} articulates the general index theorem for H-elliptic operators, extending van Erp's findings on contact manifolds \cite{van2010atiyah1}. We reformulate this theorem from a $KK$-theoretic viewpoint.
   
   \item Chapter \ref{chap4} gives a definition for transversally H-elliptic operators, establishing the framework for defining symbol and index classes.

    \item Chapter \ref{chap5} is dedicated to the main index theorem for transversally H-elliptic operators. We will give a proof of the theorem using a similar approach as in \cite{kasparov2022k} and \cite{kasparov2024coarsepseudodifferentialcalculusindex}.

     \item Chapter \ref{chap6} introduces the Fourier transform theory for $C^{\ast}$-algebra of nilpotent Lie groupoids \cite{beltictua2017fourier} to verify the compactness condition in Definition \ref{def2}.
    
    \item In the Appendix, we exhibit some other possible conditions to define the transversally H-elliptic operators. We include prototypes of these definitions, from the definitions or certain crucial properties of transversally elliptic or H-elliptic operators.

\end{itemize}

%% file: chapters/notations.tex
\begin{itemize}
    \item We denote by $G$ the locally compact Lie group acting on manifold $M$, by $\G$ a groupoid or a bundle of nilpotent Lie groups over the manifold $M$.
    \item We write $H^{\bullet}$ to mean the filtration structure over the manifold $M$, and $\mathcal{H}$ or $\mathcal{H}_{\pi}$ for the representation space of some unitary irreducible representation of the Lie group $\G\to U(\mathcal{H}_{\pi})$.
     \item We write $\sigma_H(P)$ or $\sigma_H(P,x)$ to mean the H-cosymbol of $P$ and $\sigma(P,x,\pi)$ to mean the representation $\pi$ on H-cosymbol $\pi(\sigma_H(P,x))$.
    \item If $E$ is a vector bundle over $M$, we will denote by $C^{\infty}(M; E)$ the space of smooth sections of $E$ and $C_0(M; E)$ the space of continuous sections of $E$ vanishing at infinity. We denote $C_0(E)$ to mean the space of continuous functions on $E$ (as a topological space) vanishing at infinity.
    \item On the tangent bundle $TM$, we denote its coordinate by $(x,\xi)$. Its tangent space $T(TM)$ has coordinate $(x,\xi,\chi,\zeta)$, where $\chi$ is the differential coordinate with respect to $x$, and $\zeta$ is the differential coordinate with respect to $\xi$.
   \item We denote by $\otimes_D$ the Kasparov product 
   $$\otimes_D: KK(A_1, B_1\hat{\otimes} D) \times KK(D \hat{\otimes} A_2, B_2)\to KK(A_1\hat{\otimes} A_2, B_1 \hat{\otimes} B_2).$$
   \item We denote by $\mathcal{M}(A)$ the multiplier algebra for a $C^{\ast}$-algebra $A$.
\end{itemize}

%% file: chapters/chapter01.tex
\label{chap1}
In \cite{van2010atiyah1}, van Erp analyzed the sub-elliptic operators on contact manifolds using the tangent groupoid approach and gave an index theorem for H-elliptic operators. The contact manifolds are special cases of filtered manifolds, and an analog of the tangent groupoid can be defined for filtered manifolds. We can extend the H-elliptic index theory to filtered manifolds.

In this chapter, we introduce some fundamental concepts related to filtered manifolds and the symbolic calculus of operators on filtered manifolds.

\section{Filtered manifold and the osculating groupoid}
\label{sec1.1}
Let $M$ be a smooth complete Riemannian manifold with a filtration of the tangent bundle 
$$0=H^0\subseteq H^1\subseteq H^2\subseteq\cdots \subseteq H^{n-1}\subseteq H^n=TM$$
by a sequence of vector bundles $H^i$, satisfying
$$[C^{\infty}(M;H^i),C^{\infty}(M;H^j)]\subseteq C^{\infty}(M;H^{i+j}),\ \forall 1\leq i,j\leq n.$$
Here, we follow the convention that $H^i = TM$ for $i>n$. 

Assume that $M$ is equipped with a proper isometric action by a compact group $G$.

\begin{defn}
    The filtration structure $H^{\bullet}$ is \textit{$G$-invariant} if for each $g\in G,x\in M$, the differential of the group action $g:M\to M, x\mapsto g(x)$,
\begin{equation*}
        d g:T_x M\to T_{g(x)} M
\end{equation*}
maps $H^i_x$ to $H^i_{g(x)}$ for each $0\leq i\leq n$. 
\end{defn}
We will always consider $G$-invariant filtrations in the following contents.
\begin{lem}
    At each point $x$ on the filtered manifold $M$, there is a simply-connected nilpotent Lie group structure on $T_x M$, called the osculating group at $x$, denoted by $Gr(H)_x$.
\end{lem}
\begin{proof}
    Define the graded bundle over $M$, $$gr(H):=\bigoplus_i H^i/H^{i-1}.$$

At each point $x\in M$, $gr(H)_x=H^i_x/H^{i-1}_x$ has a graded nilpotent Lie algebra structure induced from $$[X(x),Y(x)]:=[X,Y](x)\ \text{mod}\  H^{i+j-1}_x, X\in H^i, Y\in H^j. $$
It can be proved that $[X,Y](x)$ mod $H^{i+j-1}_x$ depends only on $X(x)$ mod $H^{i-1}_x$ and $Y(x)$ mod $H^{j-1}_x$. Therefore, the Lie bracket structure over $gr(H)_x$ is well defined for each $x\in M$. The graded algebra $gr(H)_x$ at each point is called the \textit{osculating Lie algebra}. The corresponding simply-connected nilpotent Lie group for each $x$ is the $Gr(H)_x$ we required.
\end{proof}

\begin{defn}\label{defn9}
    The space $$Gr(H):=\bigsqcup_{x\in M} Gr(H)_x$$ is called the \textit{osculating Lie groupoid} of $M$ with filtration structure $H^{\bullet}$. 
\end{defn}
The osculating groupoid serves as the non-commutative counterpart of the tangent bundle $TM$.
 
\begin{remark}\label{rmk1}
    The graded Lie algebra possesses a natural group of dilation automorphisms. For any $\lambda > 0$, let $\delta_{\lambda} \in \text{Aut}(gr(H))$ denote the bundle automorphism given by multiplication with $\lambda^i$ on the grading component $gr(H)^i$. For each $x \in M$, this restricts to an automorphism $\delta_{\lambda,x}\in\text{Aut}(gr(H)_x)$ of the osculating Lie algebra such that $\lim_{\lambda\to 0} \delta_{\lambda,x} = 0$. Evidently, $\delta_{\lambda_1\lambda_2} = \delta_{\lambda_1}\delta_{\lambda_2}$ for all $\lambda_1, \lambda_2 > 0$.
\end{remark}

\section{Kirillov's orbit method}
\label{sec1.2}
When dealing with the cosymbol on a filtered manifold, which will be defined in Section \ref{sec1.4}, it is necessary to replace the covectors $\xi\in T^{\ast}_x M$ in $\sigma(P,x,\xi)$ by unitary irreducible representations $\pi$ of $Gr(H)_x$.

As outlined in \cite{kirillov2004lectures}, Kirillov's orbit method characterizes the unitary dual $\widehat{G}$ as a topological space and construct the unitary irreducible representations $\pi_{\Omega}$ associated with a coadjoint orbit $\Omega \in \g^{\ast}$.
\begin{defn}
 A \textit{coadjoint orbit} is an orbit of a Lie group $G$ in the space $\g^{\ast}$ dual to $\g = \text{Lie}(G)$, on which the Lie group $G$ acts through the coadjoint representation.
\end{defn}
 The unitary irreducible representation $\pi_{\Omega}$ associated with each orbit $\Omega\subseteq \g^{\ast}$ is obtained as follows:
\begin{itemize}
    \item Select any point $F\in \Omega$;
    \item Identify a subalgebra $\mathfrak{h}$ of maximal dimension that is subordinate to $F$, i.e. $F|_{[\mathfrak{h},\mathfrak{h}]}=0$;
    \item Define a subgroup $H = \exp \mathfrak{h}$ and construct the $1$-dimensional unitary irreducible $\rho_{F,H}$ of $H$ using the formula
    $$\rho_{F,H}(\exp(X))=e^{2\pi i\langle F,X\rangle };$$
    \item $\pi_{\Omega}=\text{Ind}^G_H(\rho_{F,H})$.
\end{itemize}
Kirillov's orbit method asserts that:
\begin{thm}[\cite{kirillov2004lectures}]
    This map from coadjoint orbits $\Omega\in\mathcal{O}(G)$ in $\g^{\ast}$ to unitary irreducible representations of $\pi_{\Omega}\in \widehat{G}$ is a one-to-one correspondence.
\end{thm}

The topology of $\widehat{G}$ is defined by considering the space $\mathcal{O}(G)$ of coadjoint orbits as the quotient space of $\widehat{G}$ with the quotient topology.
\begin{exmp}
   In the case of Heisenberg group $\mathbb{H}_n\cong \R^{2n+1}$, the family of representations, corresponding to all co-adjoint orbits, can be described as 
\begin{itemize}
    \item[1.] of the form $\lambda T$ for $\lambda\in \R-\{0\}$, associated with the Schr\"{o}dinger representation $\pi_{\lambda}$.
    $$\pi_{\lambda}(x,y,t)\xi(s):=e^{-2\pi i \lambda-2\pi i\frac{\lambda}{2}x\cdot y+2\pi i \lambda s\cdot y} \xi(s-x),$$ for $s\in \R^n$, $\xi\in L^2(\R^n)$, $(x,y,t)\in \mathbb{H}_n$.
    The infinitesimal operator of $\pi_{\lambda}$ can be written as follows:
    $$d\pi_{\lambda}(X_j)=i\lambda s_j, d\pi_{\lambda}(Y_j)=\frac{\partial}{\partial s_j}, d\pi_{\lambda}(T)=i\lambda,$$
where $$X_j=\frac{\partial}{\partial x_j}+\frac{1}{2}t y_j \frac{\partial}{\partial t}, Y_j=\frac{\partial}{\partial y_j}-\frac{1}{2}t x_j \frac{\partial}{\partial t}, T=\frac{\partial}{\partial t},j=1,\dots, n.$$
    
    \item[2.] of the form $\sum_{j=1}^n(x_j' X_j' +y_j' Y_j' )$ where $x_j',y_j'\in \R$, associated with the $1$-dimensional representation $$(x, y, t) \mapsto
\exp  (i(x\cdot x' + y\cdot y' )),$$ where $x\cdot x'$ and $y\cdot y'$ denote the canonical scalar product on $\R^n$.
\end{itemize} 
\end{exmp}

\section{Groupoid $C^{\ast}$-algebra}
\label{sec1.3}
The space $Gr(H)$ defined in \ref{defn9} is a groupoid. In place of the algebra $C_0(T^{\ast}M)$, we need to analyze the $C^{\ast}$-algebra $C^{\ast}(Gr(H))$ in the filtered manifold cases. 
In this section, we briefly introduce the groupoid $C^{\ast}$-algebra theory.
\begin{defn}
\begin{itemize}
    \item  A \textit{groupoid} is a set $\G$ together with:
    \begin{itemize}
        \item[(a)] a distinguished subset $\G^{(0)}\subset \G$;
        \item[(b)] source and range maps $r,s: \G\rightrightarrows \G^{(0)}$;
        \item[(c)] a multiplication map $\G^{(2)}\to \G$, $(a,b)\mapsto ab$, where 
        $$\G^{(2)}:=\{(a,b)\in \G\times\G\ |\ s(a)=r(b)\};$$
        \item[(d)] an inverse map $\G\to \G$, $a\mapsto a^{-1}$, such that:
        \begin{itemize}
            \item[1.] $r(x)=x=s(x)$ for all $x\in \G^{(0)}$;
             \item[2.] $r(a)a=a=as(a)$ for all $a\in \G$;
              \item[3.] $r(a^{-1})=s(a)$ for all $a\in \G$;
               \item[4.] $s(a)=a^{-1}a$ and $r(a)=aa^{-1}$ for all $a\in \G$;
                \item[5.] $r(ab)=r(a)$ and $s(ab)=s(b)$ for all $(a,b)\in \G^{(2)}$;
                 \item[6.] $(ab)c=a(bc)$ whenever $s(a) = r(b)$ and $s(b) = r(c)$.
        \end{itemize}
    \end{itemize}
   \item  A \textit{Lie groupoid} is a groupoid $\mathcal{G}$ that is also a smooth manifold, such that the space of units $\mathcal{G}^{(0)}$ is a smooth submanifold of $\G$; source and range maps $s,r: \G \rightarrow \G^{(0)}$ are smooth submersions; multiplication is a smooth map $\G^{(2)} \rightarrow \G$; and the inverse map $x \mapsto x^{-1}$ is smooth $\G \rightarrow \G$. We write $\G^x = r^{-1}({x})$ and $\G_x = s^{-1}({x})$ for the range and source fibers at $x \in \G^{(0)}$, which are submanifolds of $\G$.
\end{itemize}
\end{defn}

\begin{defn}\label{defn1}
\begin{itemize}
    \item  A \textit{Haar system} on a Lie groupoid $\G$ is a family of smooth measures $\{\mu_x\}$ supported on fibers $\G^x$ of $\G$ such that:
   \begin{itemize}
       \item[(a)] For any $f\in C_c(\G)$, the function 
       $$x\mapsto \int_{\G} f(\gamma) d\mu_x(\gamma)$$
       is continuous and compactly supported;
       \item[(b)] For all $\eta \in \G$,
       $$\int_{\G} f(\gamma)d\mu_{r(\eta)}(\gamma)=\int_{\G} f(\eta\gamma)d\mu_{s(\eta)}(\gamma).$$
   \end{itemize}
   \item   Let $f,g\in C_c(\G)$, define the convolution and involution maps:
    $$f\ast g(\gamma):=\int_{\G}f(\eta)g(\eta^{-1}\gamma)d\mu_{r(\eta)}(\gamma),$$
    $$f^{\ast}(\gamma):=\overline{f(\gamma^{-1})}.$$
\end{itemize}
  
\end{defn}

\begin{lem}[\cite{renault2006groupoid}]
    The operations above turn $C_c(\G)$ into a $\ast$-algebra.
\end{lem}
\begin{defn}
    Let $\mu=\{\mu_x\in G^{0)}\}$ be a Haar system on the Lie groupoid $\G$. The \textit{groupoid $C^{\ast}$-algebra} $C^{\ast}(\G)$ is defined to be the completion of the algebra $C_c(\G)$ for the norm 
    $$\lVert f\rVert :=\sup_{\pi}\ \lVert \pi(f)\rVert,$$
    where $\pi$ runs over all non-degenerate representations of $C_c(\G)$.
\end{defn}

\section{Heisenberg calculus for differential 
operators and H-elliptic operators}
\label{sec1.4}
In this section, we will see how to use Kirillov's orbit method in Section \ref{sec1.2} to analyze the principal cosymbol of a differential operator on a filtered manifold $M$.

Let $E, F \rightarrow M$ be the vector bundles, $P:  C^{\infty}(M;E) \rightarrow  C^{\infty}(M;F)$ is a differential operator on $M$. 
\begin{defn}
    We say $P$ is \textit{of H-order $\leq k$} if $P$ can be locally expressed as a sum of monomials of the form $$\nabla_{X_1}\cdots\nabla_{X_s}\Phi,$$
where $\Phi\in  C^{\infty}(M;End(E,F))$, $\nabla$ represents a connection on $F$, $X_i\in  C^{\infty}(M;H^{a_i})$ and $\sum_{i=1}^s a_i\leq k$. We denote by $\text{Diff}^k_H(M, E, F)$ the space of differential operators of H-order $\leq k$.
\end{defn}

\begin{defn}\label{def4}
   Let $\pi : Gr(H)_x \rightarrow U(\mathcal{H}_{\pi})$ be an irreducible unitary representation. If $P \in \text{Diff}^k_H (M, E, F)$ can be locally expressed as the sum of monomials of the form $\nabla_{X_1}\cdots\nabla_{X_s}\Phi$ near $x$ with the maximal order equal to $k$, then the \textit{Heisenberg principal cosymbol of $P$ at $\pi$} can be defined as 
$$\sigma_H^k(P,x,\pi)=\sum_{\text{H-order}=k}d\pi([X_1]_x)\cdots d\pi([X_s]_x)\otimes \Phi_x:\mathcal{H}^{\infty}_{\pi}\otimes E_x\to \mathcal{H}^{\infty}_{\pi}\otimes F_x,$$
where
\begin{itemize}
    \item $d\pi: gr(H)_x\to B(\mathcal{H}^{\infty}_{\pi})$, $d\pi(X)=\frac{d}{dt}|_{t=0}\pi(\exp (tX)) v$ is the differential of the representation $\pi$,
    \item  $[X_i]_x^{a_i}\in gr(H)_x$ is the class of $X_i\in H^{a_i}$ in $gr(H)_x$. 
    \item $\mathcal{H}^{\infty}_{\pi}$ is the subspace of smooth vectors, i.e., $v\in \mathcal{H^{\infty}_{\pi}}$ if the function $G\ni x\mapsto \pi(x)v\in \mathcal{H}_{\pi}$ is smooth.
The smooth vector space $\mathcal{H}^{\infty}_{\pi}$ is generated by $\pi(\phi)=\int_{Gr(H)_x} \pi(g)\phi(g)d g$ for $\phi\in C_c^{\infty}(G)$, refer to \cite{fischer2016quantization} Chapter 1. 
\end{itemize} 
\end{defn}

\begin{remark}
\begin{itemize}
    \item According to Kirillov's orbit method, each point $\xi\in gr(H)_x$ corresponds to a representation $\pi_{\xi}$ of $Gr(H)_x$, indicating that $\sigma(P,x,\pi_{\xi})$ parallels the evaluation of the ordinary symbol $\sigma(P)$ on points $(x,\xi)\in T^{\ast}M$.
    \item  The $KK$-classes we will define later and the resulting map $\sigma_H^k(P,x,\pi)$ depend only on the H-principal component of $P$, so we can select the bundle map $TM\to gr(H)$ to be the simplest one which maps $v\in H^i_x$ to $[v]\in H^i_x/H^{i-1}_x\subset gr(H)_x$.
\end{itemize}
\end{remark}

\begin{lem}[\cite{fischer2016quantization}]\label{lem6}
The differential representation $d\pi$ satisfies
$$\forall Y_1,Y_2\in gr(H)_x,\ d\pi(Y_1)d\pi(Y_2)-d\pi(Y_2)d\pi(Y_1)-d\pi([Y_1,Y_2])=0.$$
Thus, $d\pi$ can be extended to be defined on the universal enveloping algebra of $gr(H)_x$ for each $x\in M$.
\end{lem}

  Let $N=\dim Gr(H)$. Let $\{X_j\}^N_{j=1}$ represents a local basis of the vector bundle $gr(H)$ near $x\in M$. then, $\{Y_j=X_j(x)\}^N_{j=1}$ will be a basis of $gr(H)_x$ at each point near $x$. Each $Y_j\in gr(H)_x$ can be regarded as a first-order left-invariant differential operator on $Gr(H)_x$, the formal sum $$\sum Y_1^{\alpha_1}\dots Y_N^{\alpha_N}$$ gives rise to a higher order left-invariant differential operator on $Gr(H)_x$. In fact, we have the following theorem:

\begin{thm}
    [\textbf{Poincar\'e-Birkhoff-Witt theorem}, \cite{bourbaki1989lie}, Ch 1, Sec.2.7] \label{thm9} Considering $\{X_j\}^n_{j=1}$ as a basis of the Lie algebra $\g$ of nilpotent Lie group $\G$, each vector $X\in \g$ can be regarded as a left-invariant differential operator on $C^{\infty}(G)$ defined as 
    $$X f(x):=\frac{d}{dt}|_{t=0}f(x\cdot \exp_G(tX)).$$
    Conversely, any left-invariant differential operator $T$ on $G$ can be uniquely expressed as a finite sum
    $$T=\sum c_{\alpha}X^{\alpha} $$
for some multi-index $\alpha=(\alpha_1, \dots, \alpha_n)$, where all but a finite number of the coefficients $c_{\alpha} \in \C$ are zero. This establishes an identification between the universal enveloping algebra $\mathcal{U}(\mathfrak{g})$ and the space of the left-invariant differential operators on $\G$.
\end{thm}

\begin{defn}

Let $P=\nabla_{X_1}\cdots\nabla_{X_s}\Phi$, then 
\begin{itemize}
    \item  At each point $x\in M$, we define
 $$P_x=\sum_{\text{order}=k}([X_1]_x)\cdots ([X_s]_x)\otimes \Phi_x\in \mathcal{U}(gr(H)_x)\otimes \text{End}(E_x,F_x)$$
 as the \textit{freezing coefficient} of $P$ at $x$.
 \item  The \textit{H-cosymbol of $P$} is defined to be the family of operators
 $$\sigma_H(P):=\{\sigma_H(P,x)\}_{x\in M}=\{P_x\}_{x\in M}.$$
 on $\{Gr(H)_x\}$ parameterized by $x\in M$.
\end{itemize}
 \end{defn}
By Lemma \ref{lem6} and the formula of $\sigma_H^k(P,x,\xi)$ in Definition \ref{def4}, the {Heisenberg principal cosymbol of $P$ at $\pi$} can be calculated as 
 $$\sigma_H^k(P,x,\pi)=(d\pi\otimes 1)(P_x):\mathcal{H}^{\infty}_{\pi}\otimes E_x\to \mathcal{H}^{\infty}_{\pi}\otimes F_x.$$

\begin{remark}\label{rmk4}
    The Heisenberg principal cosymbol $\sigma_H(P)$ defined above is different from the symbol in the ordinary sense. For an ordinary differential operator of order $m$, $P=\sum_{|\alpha|\leq m} a_{\alpha} \partial_x^{\alpha}$, freezing the coefficient at a point $x\in M$ induces the constant coefficient operator $P_x=\sum_{|\alpha|=m}a_{\alpha}(x)\partial_{\xi}^{\alpha}$ on $T_x M$. This family $\{P_x\}_{x\in M}$ is called the ordinary principal ``cosymbol’’ of $P$, with its Fourier transform being the principal symbol of $P$ in the ordinary sense:
    $$\sigma(P)=\sum_{\lvert\alpha\rvert=m}a_{\alpha}(x)(i\xi)^{\alpha}.$$

 The principal cosymbol is Fourier transformed to the principal symbol of $P$, in which the Fourier transform can be regarded as the representation by elements in $T_x^{\ast} M$ at each point. However, in the filtered manifold cases, $C^{\ast}(Gr(H))$ is not commutative so the Fourier transform is hard to calculate directly. That is why we only defined the Heisenberg cosymbol but not the Heisenberg symbol. The Fourier transform of the H-cosymbols will be discussed in Chapter \ref{chap6}.

The H-cosymbol will be frequently used within the Heisenberg calculus, whereas the terms ``cosymbol’’ and ``symbol’’ will refer to the family of operators of freezing coefficients and its Fourier transform, respectively, in a conventional context. 
\end{remark}

The following analysis illustrates how a representation $\pi$ interacts directly with the complete cosymbol class $\sigma_H(P)=\{\sigma_H(P,x)\}_{x\in M}$, without using the differential representation $d\pi$. 

Let $\pi:Gr(H)_x\to U(\mathcal{H_{\pi}})$ be a unitary representation, we can construct 
$$\pi(\sigma_H(P,x)):E_x\otimes \mathcal{H}_{\pi}\to F_x\otimes \mathcal{H}_{\pi}.$$
as follows:

By Theorem \ref{thm9}, $\mathcal{U}(gr(H)_x)$ elements can be identified with left-invariant differential operators on $Gr(H)_x$. So the cosymbol $\sigma_H(P)$, which is a family of elements in $\U(gr(H)_x)$ parametrized by $x$, can be considered a collection of left-invariant differential operators $\{\sigma_H(P,x)\}_{x\in M}$ on the family of manifolds $\{Gr(H)_x\}_{x\in M}$.

For each $x\in M$, since $\sigma_H(P,x)$ corresponds to a left-invariant differential operator on $Gr(H)_x$, it follows that for $f, g \in C_c^{\infty}(Gr(H)_x)$, $\sigma_H(P, x)(f\ast g)=f\ast \sigma_H(P, x)(g)$. We can extend the operator to the group $C^{\ast}$-algebra $C^{\ast}(Gr(H)_x)$ and yields
 $$\overline{\sigma_H(P,x)}:\text{Dom}(\overline{P})\subseteq E_x\otimes C^{\ast}(Gr(H)_x)\to F_x\otimes C^{\ast}(Gr(H)_x).$$
 This can be considered to be a morphism between $C^{\ast}(Gr(H)_x)$-Hilbert modules. Thus, the cosymbol $\sigma_H(P)=\{\sigma_H(P,x)\}$ parametrized by $x\in M$ can be extended to an unbounded multiplier of $C^{\ast}({Gr(H)})$:
 $$\overline{\sigma_H(P)}: C_0(M;E)\otimes_{C_0(M)} C^{\ast}(Gr(H))\to  C_0(M;F)\otimes_{C_0(M)} C^{\ast}(Gr(H)).$$
 
The representation $\pi$ can be extended to a representation of $C^{\ast}(Gr(H)_x)$ on $\mathcal{H_{\pi}}$, thus $\overline{\sigma_H(P)}\otimes_{\pi} 1$ on $(E_x\otimes C^{\ast}(Gr(H)_x))\otimes_{\pi} \mathcal{H_{\pi}}=E_x\otimes \mathcal{H_{\pi}}$ yields an operator
 $$\pi(\sigma_H(P,x)):E_x\otimes \mathcal{H}_{\pi}\to F_x\otimes \mathcal{H}_{\pi}.$$

 \begin{remark}
 Note that:
 \begin{itemize}
     \item[(i)] The map from $\sigma_H(P)$ to $\pi(\sigma_H(P,x))$ is consistent with the $d\pi$ map defined in Definition \ref{def4}  above by simple calculation, refer to \cite{fischer2016quantization} Proposition 1.7.6.
     \item[(ii)] The range and sourse map on the groupoid $Gr(H)$ are the same map. Therefore, the  groupoid $C^{\ast}$-algebra structure of $C^{\ast}(Gr(H))$ is a smooth family of group $C^{\ast}$-algebras $\{C^{\ast}{(Gr(H)_x)}\}_{x\in M}$. 
 \end{itemize}
\end{remark}
For left-invariant and homogeneous differential operators on a homogeneous Lie group, hypoellipticity is equivalent to the Rockland condition that we will define in the following theorem. The Rockland condition also helps to understand the definition of H-ellipticity.

\begin{thm}[\textbf{Rockland theorem}\label{thm12}, \cite{melin1983parametrix}] The principal cosymbol $\sigma_H(P,x)$ of a differential operator $P$ at $x\in M$ is left invertible on $L^2(M)$ if and only if for every unitary non-trivial representation $\pi$ of the group $Gr(H)_x$, the (unbounded) operator $\pi(\sigma_H(P,x))$ is injective. 
\end{thm}
\begin{defn}
    An operator $P$ satisfying the condition in Theorem \ref{thm12} is said to satisfy \textit{the Rockland condition}.
\end{defn}

\begin{defn}\label{defn8}
    A differential operator $P$ of order $k$ on $M$ from vector bundle $E$ to $F$ is said to be \textit{H-elliptic} if the principal cosymbol $\sigma_H^k(P,x)$ and $\sigma_H^k(P^{\ast},x)$ are both invertible at each point $x\in M$.
\end{defn}

\begin{remark}\label{rmk3}
\begin{itemize}
    \item According to the Rockland theorem, the condition in Definition \ref{defn8} is equivalent to the condition that for all points $x\in M$, and for every non-trivial irreducible representation $\pi$ of $Gr(H)_x$, $\pi(\sigma_H(P,x))$ and $\pi(\sigma_H(P^{\ast},x))$ are injective.
    \item   In the pseudo-differential calculus, such a definition of H-ellipticity is equivalent to the invertibility of the cosymbol of the operator in the convolution distribution algebra $\mathcal{E}_r' (Gr(H))/C^{\infty}_p(Gr(H);\Omega_r)$, see Definition \ref{def3} and \ref{def5}. For details, refer to \cite{van2019groupoid}.
\end{itemize}
   
\end{remark}
\section{Hilbert modules and $KK$-theory}
\label{sec1.5}
In order to analyze the symbol class and index class of H-elliptic operators, we need to understand the $KK$-theory.
In this section, we briefly introduce the Hilbert module theory and Kasparov's $KK$-theory. See \cite{blackadar1998k} for details.
\begin{defn}
    Let $B$ be a $C^{\ast}$-algebra. A \textit{pre-Hilbert module} over $B$ is a right $B$-module $E$ equipped with a $B$-valued ``inner product'', a function $\langle \cdot , \cdot \rangle  : E \times E \rightarrow B$, with these properties:
    \begin{itemize}
        \item[(1)] $\langle \cdot , \cdot \rangle$ is sesquilinear. (conjugate-linear in the first variable and linear in the second variable.)
        \item[(2)] $\langle x , y b \rangle=\langle x, y \rangle b$ for all $x, y \in E, b \in B$.
        \item[(3)]$\langle y,x \rangle= \langle x, y\rangle^{\ast}$.
        \item[(4)] $\langle x,x\rangle\geq 0$, $\langle x,x \rangle=0$ if and only if $x=0$.
    \end{itemize}
    For $x \in E$, put $ \lVert x\rVert = \lvert\langle x, x\rangle\rvert^ {1/2}$. This is a norm on $E$. If $E$ is complete, $E$ is called a Hilbert module over $B$.
\end{defn}
\begin{defn}
    Let $E$ be a Hilbert $B$-module. We define $\mathcal{L}(E)$ to be the set of all module homomorphisms $T: E \rightarrow E$ for which there is an adjoint module homomorphism $T^{\ast}: E \rightarrow E$ with $\langle T x, y\rangle = \langle x, T^{\ast} y\rangle$ for all $x, y \in E$.
\end{defn}
\begin{prop}[\cite{blackadar1998k}]
    Each operator in $\mathcal{L}(E)$ is bounded, and $\mathcal{L}(E)$ is a $C^{\ast}$-algebra with respect to the operator norm. In general, $\mathcal{L}(E)$ does not contain all bounded module-endomorphisms of $E$.
\end{prop}
\begin{defn}
\begin{itemize}
    \item  If $x, y \in E$, let $\theta_{x,y}$ be the operator defined by $\theta_{x,y}(z) = x \langle  y, z\rangle$ . We have $\theta_{x,y} \in \mathcal{L}(E)$; in fact, $\theta_{x,y}^{\ast} = \theta_{y,x}$. If $T \in \mathcal{L}(E)$, then $T \theta_{x,y} = \theta_{T x,y}; \theta_{x,y} T = \theta_{x,T^{\ast}y}$, so the linear span of ${\theta_{x,y}}$ is an ideal in $\mathcal{L}(E)$. The \textit{space of compact operators} $\mathcal{K}(E)$ is defined to be the closure of this linear span.
     \item A \textit{regular operator} on a Hilbert $B$-module $E$ is a densely defined operator $T$ on $E$ with densely defined adjoint $T^{\ast}$, such that $1 + T^{\ast}T$ has dense range in $E$.
\end{itemize}

\end{defn}

\begin{remark}
     If $T$ is a regular operator, then $(1+T^{\ast}T)^{-1}$ extends to an operator in $\mathcal{L}(E)$.
\end{remark}

\begin{defn}
Let $A$ and $B$ be graded $C^{\ast}$-algebras. $\mathbb{E}(A, B)$ is the set of all triples $(E, \phi, F)$, where $E$ is a countably generated graded Hilbert module over $B$, $\phi$ is a graded $\ast$-homomorphism from $A$ to $\mathcal{L}(E)$, and $F$ is an operator in $\mathcal{L}(E)$ of degree $1$, such that $[F, \phi(a)]$, $(F^2-1)\phi(a)$, and $(F -F^{\ast})\phi(a)$ are all in $\mathcal{K}(E)$ for all $a \in A$. The elements of $\mathbb{E}(A, B)$ are called \textit{Kasparov modules for $(A, B)$}.
\end{defn}

\begin{defn}
\begin{itemize}
    \item Two triples $(E_0, \phi_0, F_0)$ and $(E_1, \phi_1, F_1)$ are \textit{unitarily equivalent} if there is a unitary in $\mathcal{L}(E_0, E_1)$, of degree $0$, intertwining the $\phi_i$ and $F_i$. Unitary equivalence is denoted $\sim_u$.
      \item Let $IB$ denote the $C^{\ast}$-algebra of continuous functions from $[0,1]$ to $B$.
    A homotopy connecting $(E_0, \phi_0, F_0)$ and $(E_1, \phi_1, F_1)$ is an element $(E, \phi, F)$ of $\mathbb{E}(A, IB)$ for which $(E \hat{\otimes}_{f_i} B, f_i \circ \phi, {f_i}_{\ast}(F)) \sim_u (E_i, \phi_i, F_i)$, where $f_i$, for $i = 0, 1$, is the evaluation homomorphism from $IB$ to $B$. Homotopy equivalence is denoted $\sim_h$.
        \end{itemize}
\end{defn}
\begin{defn}
    The\textit{ Kasparov $K$-group} $KK(A, B)$ is defined to be the set of equivalence classes of $\mathbb{E}(A, B)$ under $\sim_h$.
    
    When $A$ and $B$ are graded $G$-algebras, the equivariant $KK$-group $KK^G(A, B)$ is defined in a similar manner with the additional condition that $F$ is $G$-continuous and $(g\cdot F-F) \phi(a)\in \mathcal{K}(E)$ for all $a\in A$ and $g\in G$.
\end{defn}

\begin{exmp}
    In the Kasparov $K$-group $KK(A,B)$, 
    \begin{itemize}
        \item if $A=\C$, then we have $KK(\C,B)\cong K_0(B)$, the operator $K$-theory group of $B$ in degree $0$;
        \item if $B=\C$, then we have $KK(A,\C)\cong K^0(A)$, the $K$-homology of $A$ which consists of equivalence classes of $A$-Fredholm modules.
    \end{itemize}
\end{exmp}
In order to define the $KK$-product, we need some special $C^{\ast}$-algebras. 
\begin{defn}
    Let $A$ be a $C^{\ast}$-algebra.
\begin{itemize}
    \item  A positive element $h \in A$ is \textit{strictly positive} if $\phi(h) > 0$ for every state $\phi: A\to\C$ of $A$. A $C^{\ast}$-algebra containing a strictly positive element will be called a \textit{$\sigma$-unital $C^{\ast}$-algebra}.
  \item A \textit{separable} $C^{\ast}$-algebra is a $C^{\ast}$-algebra whose underlying topological space is a separable topological space.
\end{itemize}
\end{defn}

    Kasparov showed that there is a product operation on $KK$-groups. We briefly explain how the $KK$-product is constructed.

 Given $x \in KK(A, D)$, $y \in KK(D, B)$, choose representatives $(E_1, \phi_1, F_1) \in \mathbb{E}(A, D)$ and $(E_2, \phi_2, F_2) \in \mathbb{E}(D, B)$; then define $(E, \phi, F) \in \mathbb{E}(A, B)$ by $E = E_1 \hat{\otimes}_{\phi_2} E_2$, $\phi=\phi_1\hat{\otimes}_{\phi_2} 1$, $F=F_1\# F_2$, where $F_1 \# F_2$ is a suitable combination of $F_1$ and $F_2$. 

 In general, $F$ will be a combination of $F_1\hat{\otimes} 1$ and $1 \hat{\otimes} F_2$ (when these are suitably defined). The coefficients will come from a judiciously chosen “partition of unity”. However, $1 \hat{\otimes} F_2$ cannot be defined in general cases, so we need an operator $F \in \mathcal{L}(E) (E = E_1\hat{\otimes}_{\phi_2} E_2)$ which “plays the role of $1 \hat{\otimes} F_2$ up to compacts.”

For any $x \in E_1$ there is an operator $T_x \in \mathcal{L}(E_2, E)$ defined by $T_x(y) = x \hat{\otimes} y$.
\begin{defn}
\begin{itemize}
    \item     An operator $F \in \mathcal{L}(E)$ is called an \textit{$F_2$-connection} for $E_1$(or an $F_2$-connection on $E$) if, for any $x \in E_1$, 
    \begin{equation*}
        \begin{aligned}
            T_x\circ F_2-(-1)^{\partial x \cdot \partial F_2} F\circ T_x\in \mathcal{K}(E_2,E), \\
      F_2\circ T_x^{\ast}-(-1)^{\partial x \cdot \partial F_2} T_x^{\ast}\circ F\in \mathcal{K}(E,E_2), 
        \end{aligned}
    \end{equation*}
     
      \item Let $F$ be an $F_2$-connection on $E$. $(E, \phi, F)$ is called a \textit{Kasparov product} for $(E_1, \phi_1, F_1)$ and $(E_2, \phi_2, F_2)$ if
      \begin{enumerate}
          \item[(1)] $(E, \phi, F)$ is a Kasparov $(A, B)$-module;
          \item[(2)]  For all $a \in A$, $\phi(a)[F_1\hat{\otimes} 1, F]\phi(a)^{\ast} \geq 0\text{ mod }\mathcal{K}(E)$.
      \end{enumerate}
\end{itemize}
\end{defn}
Let $G$ be any $F_2$-connection of degree $1$. Using Kasparov's Technical Theorem (Theorem 1.4 in \cite{kasparov1988equivariant}), we can choose appropriate $M, N \in \mathcal{L}(E)$ with $0\leq M, N\leq 1$ and $M+N=1$ satisfying several conditions so that the above (1) and (2) are realized, and the product operator $F$ can be constructed to be $F=M^{1/2}(F_1\hat{\otimes} 1)+N^{1/2} G$,
\begin{thm}[\cite{kasparov1988equivariant}]
    Let $A$, $B$ be separable and $D$ be $\sigma$-unital graded $C^{\ast}$-algebras. Let $(E_1, \phi_1, F_1)$ be a Kasparov $(A, D)$-module and $(E_2, \phi_2, F_2)$ be a Kasparov $(D, B)$-module, there is a Kasparov product for $(E_1, \phi_1, F_1)$ and $(E_2, \phi_2, F_2)$, which is a Kasparov $(A,B)$-module, unique up to homotopy $\sim_h$.
    The Kasparov product defines a bilinear function ${\otimes}_D: KK(A, D) \times KK(D, B) \rightarrow KK(A, B)$.
\end{thm}
\begin{remark}
    Such a product between $KK$-groups does not exist in algebraic $K$-theory.
\end{remark}
\begin{cor}[\cite{kasparov1988equivariant}]
    Let $A_1$, $A_2$, $B_1$, $B_2$, and $D$ be graded $C^{\ast}$-algebras with $A_1$ and $A_2$ separable and $B_1$ and $D$ $\sigma$-unital. There is a Kasparov product
  $$\otimes_D: KK(A_1, B_1\hat{\otimes} D) \times KK(D \hat{\otimes} A_2, B_2)\to KK(A_1\hat{\otimes} A_2, B_1 \hat{\otimes} B_2)$$
    which is is bilinear, contravariantly functorial in $A_1$ and $A_2$ and covariantly functorial in $B_1$ and $B_2$.
\end{cor}
In the discussions in Chapter \ref{chap5}, we also need the $\RKK$-group which is defined in \cite{kasparov1988equivariant}. We give a brief introduction here.
\begin{defn}\label{defi3}
    For a $\sigma$-compact $G$-space $M$ and two $G-C_0(M)$-algebras $A$ and $B$, define the group $\RKK^G(M;A,B)$ in the same way as $KK^G(A,B)$, with the additional condition on $(E, \phi, F)\in \mathbb{E}^G(A,B)$ that: for any $a\in A$, $b\in B$, $f\in C_0(M)$, $e\in E$, we have the equality $(fa)eb=ae(fb)$ in $E$.
\end{defn}

\section{Heisenberg calculus for pseudo-differential operators}
\label{sec1.6}
The normalization of operators to $0$-order and the inversion of cosymbol classes modulo smoothing terms are essential in our theory, but they are outside of the class of differential operators which we discussed in Section \ref{sec1.4}. In this section, we consider the Heisenberg calculus for pseudo-differential operators.
\subsection*{General Lie groupoid}
In \cite{van2019groupoid}, van Erp and Yuncken gave the pseudo-differential calculus on a filtered manifold through the Heisenberg tangent groupoid. We briefly review it here.

\begin{defn}
    Let $\mathcal{G}$ be a Lie groupoid. An \textit{$r$-fibered distribution} on a Lie groupoid $\G$ is a continuous $C^{\infty}(\G^{(0)})$-linear map
    $$u:C^{\infty}(\G)\to C^{\infty}(\G^{(0)}),$$
    where the $C^{\infty}(\G^{(0)})$-module structure on $C^{\infty}(\G)$ is induced by the pullback of functions via $r$. The set of $r$-fibered distributions on $\G$ is denoted $\mathcal{E}'_r(\G)$. An $r$-fibered distribution $u$ determines a smooth family $\{u_x\}_{x\in \G^{(0)}}$ of distributions on the $r$-fibers $u_x\in \mathcal{E}'(\G^x)$,
    $$\langle u,f\rangle(x)=\langle u_x, f|_{\G^x}\rangle, \quad  f\in C^{\infty}(\G).$$
\end{defn}
\begin{defn}
    A subset $X$ of a groupoid $\G$ is \textit{proper} if both $r: X \rightarrow M$ and $s: X \rightarrow M$ are proper maps. Denote by $\Omega_r$ and $\Omega_s$ the bundle of smooth densities tangent to the range and source fibers of $\G$ respectively. The space of smooth sections of $\Omega_r$ whose support is a proper subset of $\G$ is denoted $C_p^{\infty}(\G; \Omega_r)$. 
\end{defn}
\begin{lem}[\cite{lescure2017convolution}]
    $C_p^{\infty}(\G; \Omega_r)$ is a right ideal in $\mathcal{E}'_r(\G)$.
\end{lem}
In \cite{connes1994noncommutative}, Connes constructed the tangent groupoid $$\mathbb{T}M=TM\times \{0\}\sqcup (M\times M)\times \R^{\times},$$
and used it to prove the Atiyah-Singer's index theorem in $K$-theory. In the filtered case, a similar concept can be defined as follows.
\begin{defn}
    Define $\mathbb{G}r(H)$ to be the disjoint union of the osculating groupoid $Gr(H)$ with a family of pair groupoids indexed by $\R^{\times}=\R\setminus \{0\}$,
    $$\mathbb{G}r(H)=Gr(H)\times \{0\}\sqcup (M\times M)\times \R^{\times},$$
    it is called the \textit{H-tangent groupoid} of the osculating groupoid $Gr(H)$.
    \begin{thm}[\cite{van2017tangent}]
        There is a compatible global smooth structure on $\mathbb{G}r(H)$ that makes $\mathbb{G}r(H)$ into a Lie groupoid.
    \end{thm}
    There is an action $\alpha:\R_+\to \text{Aut}(\mathbb{G}r(H))$, called the \textit{zoom action}, defined to be a smooth one-parameter family of Lie groupoid automorphisms
\begin{equation*}
    \begin{aligned}
        &\alpha_{\lambda}(x,y,t)=(x,y,\lambda^{-1}t)\quad \text{if} &(x,y) \in M\times M,\ t\neq 0\\
        &\alpha_{\lambda}(x,\xi,0)=(x,\delta_{\lambda}(\xi),0)\quad \text{if} &(x,\xi)\in Gr(H),\ t=0
    \end{aligned}
\end{equation*}
where $\delta_{\lambda}$ is defined in Remark \ref{rmk1}.
\end{defn}

\begin{defn}\label{def3}
Define that:
\begin{itemize} 
    \item  A properly supported $r$-fibered distribution $\mathbb{P}\in \mathcal{E}'_r(\mathbb{G} r(H))$ is called \textit{essentially homogeneous of weight $m \in \R$} if
    $${\alpha_{\lambda}}_{\ast} \mathbb{P}-\lambda^m \mathbb{P}\in C^{\infty}_p(\mathbb{G} r(H),\Omega_r).$$
    The space of essentially homogeneous distributions of weight $m$ is denoted as ${\Psi}_H^m(M)$.
    \item A Schwartz kernel $P\in \mathcal{E}'_r(M\times M)$  is an H-pseudo-differential kernel of order $\leq m$ if $P =\mathbb{P}|_{t=1}$ for certain $\mathbb{P}\in {\Psi}_H^m(M)$. The set of H-pseudo-differential kernels of order $\leq m$ will be denoted $\psi^m_H(M)$. 
    \item  The \textit{cosymbol space} is defined as:
    $$\Sigma^m_H(M):=\{K\in \mathcal{E}'_r(Gr(H))/C^{\infty}_p(Gr(H),\Omega_r)|\ {\delta_{\lambda}}_{\ast}K = \lambda^m K \text{ for all } \lambda\in \R_+\}.$$
    \item The \textit{cosymbol map} 
    $$\sigma_m: \psi^m_H(M)\to \Sigma^m_H(M)$$
    is defined to select any $\mathbb{P} \in {\Psi}^m_H(M)$ with $\mathbb{P}|_{t=1} = P$ and set $$\sigma_m(P) := \mathbb{P}|_{t=0} \text{ modulo }C_p^{\infty}(T_H M; \Omega_r).$$
\end{itemize}
\end{defn}

H-elliptic pseudo-differential operators can be characterized using the cosymbol defined above: 
\begin{defn}\label{def5}
    An H-pseudo-differential operator $P \in \psi^m_H(M)$ is called \textit{H-elliptic} if its principal cosymbol $\sigma_m(P)$ admits a convolution inverse in the cosymbol space $\Sigma^m_H(M)=\mathcal{E}'_r(Gr(H))/ C_p^{\infty}(Gr(H); \Omega_r)$.
\end{defn}

\begin{remark}
Note that:
\begin{itemize} 
    \item[(i)] The Rockland condition is equivalent to the existence of a left parametrix for the cosymbol $\sigma_H(P)$. In the H-pseudo-differential calculus, if $P\in \psi_H^m(M)$ satisfies the Rockland criterion, then the operator $P$ itself acquires a left parametrix $Q\in \psi_H^{-m}(M)$ because of the existence of the left parametrix of the cosymbol $\sigma_H(P)$. Consequently, H-elliptic pseudo-differential operators are hypoelliptic.
    \item[(ii)] As explained in Section 10 of \cite{van2019groupoid}, differential operators are special cases in pseudo-differential theory. General analysis in pseudo-differential calculus is crucial, although differential calculus facilitates an intuitive comprehension of the cosymbol and index theory of H-elliptic operators.
\end{itemize}
\end{remark}

The cosymbol of a pseudo-differential operator can be viewed as a family of operators on the manifolds $Gr(H)_x$ for $x\in M$. So specifically, we need to analyze the special case $\G=Gr(H)$. 

\subsection*{Family of nilpotent Lie groups on a manifold}
Note that $Gr(H)$ is a family of Lie groups over $M$ with a graded structure on $gr(H)$ and $\sigma_H(P)$ is a family of left-invariant pseudo-differential operators on $Gr(H)_x$.

So we conclude some results for left-invariant pseudo-differential operators on a nilpotent Lie groupoid with a graded vector bundle $\g=\bigoplus_{i=1}^n \g_i\to M$, such that $[\g_i, \g_j ] \subseteq \g_{i+j}$ for all $i, j \in \N$ with $\g_k = 0$ if $k > n$. Assume that the bundle $\g$ is equipped with an Euclidean metric such that different $\g_i$’s are orthogonal to each other. The corresponding Lie group bundle is denoted by $\mathcal{G}$. The cosymbol operators are thus categorized as pseudo-differential operators in this calculus by considering $\G$ to be the osculating groupoid $Gr(H)$.
\begin{defn}
    Let $S^k(\mathcal{G})$ be the space of distributions $u \in \mathcal{E}'(\mathcal{G})$ such that $u$ is transversal to the map $\pi : \mathcal{G} \to M$ and for each $\lambda \in \R^{+}$,
$${\alpha_{\lambda}}_{\ast}u-\lambda^k u\in C_c^{\infty}(\mathcal{G}).$$
\end{defn}

Let $u \in  S^k(\mathcal{G})$. For each $x \in M$, the operator can be defined as
\begin{equation*}
    \begin{aligned}
         Op(u_x): C_c(&\mathcal{G}_x) &\to &\ C_c(\mathcal{G}_x)\\
       &f &\mapsto &(l\to (u_x,f(l\cdot)))
    \end{aligned}
\end{equation*}
and they glue together to form an operator 
$$ Op(u):C_c(\mathcal{G})\to C_c(\mathcal{G}).$$
It can be extended to an operator on $C^{\ast}(\mathcal{G})$,
$$\overline{Op(u)}:C^{\ast}(\mathcal{G})\to C^{\ast}(\mathcal{G}).$$

\begin{thm}[\cite{mohsen2020index}]
    Let $u\in S^k(\mathcal{G})$, then
    \begin{itemize}
        \item If $k \leq 0$, then $\overline{Op(u)} \in \mathcal{L}(C^{\ast}(\mathcal{G}))$;
        \item If $k < 0$, then $\overline{Op(u)} \in \mathcal{K}(C^{\ast}(\mathcal{G}))=C^{\ast}(\mathcal{G})$;
        \item If $k > 0$, and $u$ satisfy the Rockland condition, then $\overline{Op(u)}$ is a regular operator.
    \end{itemize}
\end{thm}

If $P \in \U^k(\g)$ be a left-invariant differential operator on $\G$, then the distribution $(u, f) := Pf(e)$ belongs to $S^k(\G)$. Since $P$ is left-invariant, $Op(u) = P$. Thus we can consider $\U^k(\g) \subseteq S^k(\G)$.

For left-invariant differential operators on $\mathcal{G}$, the Rockland condition facilitates the consideration in $KK$-theory:
\begin{cor}[\cite{mohsen2020index}]\label{cor2}
    Let $P\in \mathcal{U}^k(\g)\otimes \text{Hom}(E, F)$ be a left-invariant differential operator on the bundle of Lie groups $\mathcal{G}$ satisfying the Rockland condition with $\alpha_{\lambda}(P)=\lambda^k(P)$ for $k\geq 0$. Then $\overline{P}$ is a regular operator in the sense of Baaj-Julg, and $\overline{P}^{\ast}=\overline{P^{\ast}}$. The operator
    $$D=\overline{P}/(1+\overline{P}^{\ast}\overline{P})^{1/2}\in \mathcal{L}( C_0(M;E)\otimes_{C_0(M)} C^{\ast}(\mathcal{G}),  C_0(M;F)\otimes_{C_0(M)} C^{\ast}(\mathcal{G}))$$
    is a Fredholm operator in the sense of Kasparov $C^{\ast}(\mathcal{G})$-modules.
\end{cor}
\begin{cor}
    For any arbitrary H-elliptic operator $P$, the principal cosymbol $\sigma_H(P)$ represents a $KK$-element $[\sigma_H(P)]\in KK(C_0(M),C^{\ast}(Gr(H))$.
\end{cor}
\begin{proof}
By Corollary \ref{cor2}, $$\Big[ C_0(M;E_0\oplus E_1)\otimes_{C_0(M)}C^{\ast}(\mathcal{G}),
\begin{pmatrix}
    0 & \overline{P}^{\ast}/(1+\overline{P}\overline{P}^{\ast})^{1/2}\\
    \overline{P}/(1+\overline{P}^{\ast}\overline{P})^{1/2} & 0
\end{pmatrix}
\Big]$$ represents an element in $KK(C_0(M),C^{\ast}(\mathcal{G}))$.
Apply this result to $\sigma_H(P)$, the family of left-invariant differential operators on $Gr(H)$, so it represents an element in $KK(C_0(M),$ $ C^{\ast}(Gr(H))$.
\end{proof}

This framework can be extended to pseudo-differential operators. For a pseudo-differential operator $P$ on $M$, the principal cosymbol $\sigma_H(P)$ is an unbounded multiplier of the $C^{\ast}$-algebra $C^{\ast}(Gr(H))$, and the H-order of $P$ less or equal to $0$ decide the boundedness or compactness of $P$ as follows:
\begin{thm}[\cite{androulidakis2022pseudodifferential}]\label{thm6}
    Let $P$ be a pseudo-differential operators of H-order $k$ on compact manifold $M$, then
    \begin{itemize}
        \item If $k < 0$, then $P$ extends to a compact operator $L^2(M,E)\to L^2(M,F)$;
        \item If $k = 0$, then $P$ extends to a bounded operator $L^2(M,E)\to L^2(M,F)$.
    \end{itemize}
\end{thm}

\begin{remark}\label{rmk8}
    In case $M$ is locally compact, denote the extension of $P$ to $L^2(M,E)$ by $\bar{P}$, then the result will becomes
    \begin{itemize}
        \item If $k < 0$, then $f\cdot \bar{P}$ is a compact operator for each $f\in C_0(M)$;
        \item If $k = 0$, then $f\cdot \bar{P}$ is a bounded operator for each $f\in C_0(M)$.
    \end{itemize}
    
\end{remark}
\section{Index map for H-elliptic operators}
\label{sec1.7}
    Let $M$ be a compact manifold with the filtration structure $H^{\bullet}$ and $E, F$ vector bundles over $M$, a sequence of Sobolev spaces $\tilde{H}^s(M, E, F)$ are defined in \cite{androulidakis2022pseudodifferential}. For a non-negative integer $s$, it encompasses the set of vector bundle $E$-valued distributions $u$ such that $Du\in L^2(F)$ for any differential operators $D\in \text{Diff}^s_{{H}}(M,E,F)$. It can be naturally extended to real numbers $s\geq 0$, and for $s < 0$, $\tilde{H}^s(M,E,F)$ is defined to be the dual of $\tilde{H}^{-s}(M,E,F)$. 

\begin{thm}[\cite{androulidakis2022pseudodifferential}]
    Let $M$ be a smooth compact manifold, and $P$ be a differential operator $C^{\infty}(M)\to C^{\infty}(M)$ of order $k$. Then the following are equivalent:
\begin{itemize}
    \item[(1)]   For every $x \in M$ and $\pi \in gr(H)_x$, $\sigma_H(P,x,\pi)$ is injective on $C^{\infty}(H_{\pi})$.
    \item[(2)] For every (or for some) $s \in \R$, and every distribution $u$ on $M$, $Pu \in  \tilde{H}^s (M)$ implies
$u \in \tilde{H}^{s+k}$.
\item[(3)] For every (or for some) $s \in \R$, the extension $\bar{P}:\tilde{H}^{s+k}(M)\to\tilde{H}^s(M)$ is left invertible modulo compact operators.
\end{itemize}
\end{thm}
Then, by Definition \ref{defn8} and Remark \ref{rmk3}, we have
\begin{cor}
    If $P$ is an H-elliptic differential operator of order $k$, $\bar{P}:\tilde{H}^s(M)\to \tilde{H}^{s-k}(M)$ is a Fredholm operator.
\end{cor}
\begin{defn}
     Let $P$ be an H-elliptic differential operator of order $k$. For arbitrary $s\in \R$, $P$ can be extended to the Sobolev space $\tilde{H}^s(M)$,
     $$\bar{P}:\tilde{H}^{s}(M)\to\tilde{H}^{s-k}(M).$$
     We define the \textit{analytical index of $P$} to be $\text{Ind}(P)=\dim \ker \bar{P}-\dim\ker \bar{P}^{\ast}$.
\end{defn}
 Since H-elliptic operators are hypoelliptic, the kernel of $\bar{P}$ and $\bar{P}^{\ast}$ lie in $C^{\infty}(M)=\bigcap_{s\in \R} \tilde{H}^s(M)$, so the analytical index is well-defined.

 \begin{prop}[\cite{mohsen2022index}]\label{prop1}
     Let $M$ be a compact filtered manifold. The analytical index is independent of the choice of $s$ and only depends on the symbol class $[\sigma_H(P)]\in KK(C(M), C^{\ast}(Gr(H)))$.
     The map from the symbol class to the analytical index is denoted by the composition
$$\text{Ind}_{H^{\bullet}}:KK(C(M), C^{\ast}(Gr(H)))\to K^0(C(M))\to \Z,$$
where the second map is induced from the inclusion $\C\subset C(M)$.

 In the $G$-invariant case, all the $KK$-classes are replaced by equivariant $KK$-classes, and the analytical index $\text{Ind}(P)=[\text{Ker}\ P]-[\text{Coker}\ P]$ in the representation ring $R(G):=KK^G(\C, \C)$ only depends on the symbol class $[\sigma_H(P)]\in KK^G(C(M), C^{\ast}(Gr(H)))$, i.e, we have the analytical index map
$$\text{Ind}_{H^{\bullet}}:KK^G(C(M), C^{\ast}(Gr(H)))\to K^0(C^{\ast}(G,C(M)))\to R(G).$$
 \end{prop}

\section{Groupoid approach and Connes-Thom's isomorphism}
\label{sec1.8}
In this section, we explain the groupoid approach which was invented by Connes in \cite{connes1994noncommutative}. We will also introduce the Connes-Thom's isomorphism, which is important in van Erp and Mohsen's index theory for H-elliptic operators.

\begin{lem}\label{lem3}
    Let $C$ be a family of $C^{\ast}$-algebras parametrized by $[0,1]$, such that the fiber over $(0,1]$ is the $C^{\ast}$-algebra $B$ and the fiber over $0$ is the $C^{\ast}$-algebra $A$. There is an associated exact sequence
 $$0\to B\otimes C_0(0,1]\to C\xrightarrow{ev_0} A\to 0.$$ 
 Then the induced map of $ev_0$ in $K$-theory is an isomorphism. 
 Consequently, we can define a map in $K$-theory:
 $$\mu:=(ev_1)_{\ast}\circ (ev_0)_{\ast}^{-1}:K_{\ast}(A)\to  K_{\ast}(B),$$
 where $ev_t$ is the restriction of $C$ to the $C^{\ast}$-algebra fibered over $\{t\}$.
\end{lem}

 \begin{proof}
   We have the $6$-term exact sequence in $K$-theory:
 \[
 \begin{tikzcd}[center picture]
&K_1(B\otimes C_0(0,1]) \arrow[r] & K_1(C) \arrow[r] &K_1(A) \arrow[d]\\
     &K_0(A)\arrow[u] & K_0(C) \arrow[l] & K_0(B\otimes C_0(0,1]) \arrow[l].
 \end{tikzcd}
 \]
  Since $C_0(0,1]$ is contractible, we see that $K(ev_0):K_{\ast}(C)\to K_{\ast}(A)$ is an isomorphism.    
\end{proof}

\begin{lem}\label{lemma5}
If the global $C^{\ast}$-algebra $A$ can be decomposed as $A=\{A_{ij}\}_{(i,j)\in [0,1]\times [0,1]}$, a family of $C^{\ast}$-algebras over $[0,1]\times [0,1]$ with parameters $s,t\in [0,1]$, whose fibers are:
\begin{itemize}
    \item $A_{00}$: fibered over $\{0\}\times\{0\}$,
    \item $A_{01}$: fibered over $\{0\}\times (0,1]$,
    \item $A_{10}$: fibered over $(0,1]\times\{0\}$,
    \item $A_{11}$: fibered over $(0,1]\times(0,1]$.
\end{itemize}
   Then, we obtain a commutative diagram:
 \[
 \begin{tikzcd}[center picture]
&K_{\ast}(A_{00})\arrow[d,"\mu_{s=0}"] \arrow[r,"\mu_{t=0}"] & K_{\ast}(A_{10}) \arrow[d,"\mu_{s=1}"]\\
     &K_{\ast}(A_{01})\arrow[r,"\mu_{t=1}"] & K_{\ast}(A_{11}).
 \end{tikzcd}
 \]
  in which we apply Lemma \ref{lem3} to each side of  $\partial{([0,1]\times [0,1])}$.
\end{lem}
 \begin{proof}
     Each component $A_{ij},i,j=0,1$ is a restriction of the global $C^{\ast}$-algebra $A$. The compositions along both side of the commutative diagram from $K_{\ast}(A_{00})$ to $K_{\ast}(A_{11})$ illustrate the induced map on $K$-theory, composition of the lift map from $A_{00}$ to the global algebra $A$, and the restriction map from $A$ to $A_{11}$. The result then follows.
 \end{proof}
 \begin{remark}
     The same procedure can be performed by replacing $K_{\ast}(\cdot)$ with $KK(A',\cdot)$ for $A'$ an arbitrary $C^{\ast}$-algebra.
 \end{remark}

\begin{thm}[\textbf{Connes-Thom's isomorphism}, \cite{nistor2003index}, Lemma 4.3]\label{thm8}
    Let $\G$ represent a bundle of simply-connected, solvable Lie group over $M$, then
    $$K_i(C^{\ast}(\G))\cong K_i(C_0(\g^{\ast}))\cong K^i(\g^{\ast}).$$
\end{thm}
This Connes-Thom's isomorphism can be written as the $KK$-product with some special element as follows:
\begin{prop}\label{prop3}
    There is a $KK$-element $[\mathcal{E}_0]\in KK(C^{\ast}(\G),C_0(\g^{\ast}))$, called the Connes-Thom class, such that the $KK$-product with $[\mathcal{E}_0]$ over $C^{\ast}(\G)$ yields the Connes-Thom's isomorphism in Theorem \ref{thm8}, i.e.
     $$K_i(C^{\ast}(\G))\xrightarrow[\cong]{\otimes_{C^{\ast}(G)}[\mathcal{E}_0]} K_i(C_0(\g^{\ast})).$$
\end{prop}
    
\begin{proof}
    By Proposition 1 in \cite{fack1981connes}, the standard Connes-Thom's isomorphism for $C^{\ast}$-dynamic system $(A,\R,\alpha)$:
    $$KK^i(B, A)\to KK^{i+1}(B, A\rtimes_{\alpha} \R)$$ can be represented as the Kasparov product with an element in $KK^1(A, A\rtimes_{\alpha}\R)$. 
By the construction in the proof of Lemma 4.3 in \cite{nistor2003index}, the Connes-Thom's isomorphism in Theorem \ref{thm8} is a composite of such standard Connes-Thom's isomorphisms. By the associativity of $KK$-products, we see that the composition
$$KK_i(A,C^{\ast}(\G))\cong KK_i(A,C_0(\g^{\ast}))$$
can be calculated by the Kasparov product with a certain element in $KK(C^{\ast}(\G), C_0(\g^{\ast}))$.
\end{proof}

Now we give a concrete construction for the Connes-Thom's isomorphism in Theorem \ref{thm8}. Let $\G$ represent a bundle of nilpotent Lie groups over a locally compact manifold $M$. We can define a bundle of nilpotent Lie groups over $M \times [0, 1]$, denoted by $\mathbb{G}$ and called the adiabatic groupoid for $\G$:
$$\mathbb{G}= \G\times(0, 1] \cup \g \times \{0\},$$
with its topology induced by the isomorphism 
\begin{equation*}
\begin{aligned}
\g\times [0,1]&\to \mathbb{G}\\
    (X,t) &\mapsto (\exp(tX),t) &\text{for}\ t>0,\\
    (X,0) &\mapsto (X,0) &\text{for}\ t=0.        
\end{aligned}
\end{equation*}
Its $C^{\ast}$-algebra $C^{\ast}(\mathbb{G})$ lies in the exact sequence
$$0\to C_0(0,1]\otimes C^{\ast}\G\to C^{\ast}(\mathbb{G})\xrightarrow{ev_0} C_0(\g^{\ast})\to 0.$$

 By the arguments in Lemma \ref{lem3}, the map $ev_0$ induces an isomorphism in Kasparov's $KK$-theory. 
 \begin{lem}[\cite{nistor2003index}, Theorem 2.22]\label{lem1}
     For arbitrary $C^{\ast}$-algebra $A$, the composition $$(ev_1)_{\ast}\circ (ev_0)_{\ast}^{-1}:KK(A,C_0(\g^{\ast}))\to KK(A,C^{\ast}(\mathbb{G}))\to KK(A,C^{\ast}\G)$$
 is an isomorphism. It is consistent with the inverse of the Connes-Thom's isomorphism and equal to the $KK$-product with $[ev_0]^{-1}\in KK(C_0(\g^{\ast}), C^{\ast}(\mathbb{G}))$ and $[ev_1]\in KK(C^{\ast}(\mathbb{G}), C^{\ast}(\G))$.
 \end{lem}

For each $x\in M$, The Lie group $Gr(H)_x$ is nilpotent and thus solvable; it is also simply-connected due to its construction. Thus, let $\G=Gr(H)_x$, we have
\begin{cor}\label{cor3} For any $C^{\ast}$-algebra $A$, there is an isomorphism
    $$KK(A,C_0(gr(H)^{\ast}))\cong KK(A,C^{\ast}(Gr(H))).$$
\end{cor}
 It is also called the Connes-Thom's isomorphism. It will be used in Chapter \ref{chap3} in Mohsen's H-elliptic index theory.

%% file: chapters/chapter02.tex
\label{chap2}
In order to give an appropriate definition for transversally H-elliptic operators from H-elliptic theory, we first observe the differences between the elliptic and transversally elliptic operators in $KK$-theory. We refer the elliptic theory to \cite{kasparov2017elliptic} and the transversal elliptic theory to \cite{kasparov2022k}, so that we can avoid the errors in transversal elliptic theory in \cite{kasparov2017elliptic}, which has been clarified by Kasparov in \cite{kasparov2022k}. We will focus on the index of operators in $KK_0$, the $KK_1$ index for self-adjoint operators on non-graded bundles can be analogously analyzed.

\section{Elliptic operators from $KK$-theoretic viewpoint}
\label{sec2.1}
In this section, we introduce the index theory of elliptic operators discussed in \cite{kasparov2017elliptic}, elucidating the symbol class and index class in $KK$-homology.

\begin{defn}[\cite{kasparov2017elliptic}]
     A pseudo-differential operator $P$ of order $0$ on $M$ is said to be \textit{elliptic} if 
     $$\lVert\sigma(P,x,\xi)\sigma(P^{\ast},x,\xi)-1\rVert\to 0,$$ $$\lVert\sigma(P^{\ast},x,\xi)\sigma(P,x,\xi)-1\rVert\to 0 $$
     when $\xi\to \infty$.
\end{defn}
\begin{remark} 
The definition asserting that $\sigma(P)$ is invertible at infinity aligns well with this definition. If $\sigma(P)$ is invertibility at infinity in $\xi$ , it can be normalized by replacing it with $((\sigma(P)\sigma(P^{\ast}))^{-1/2}\sigma(P)$. The new symbol is unitary at infinity in $\xi$ and homotopic to the initial one.
Moreover, upon replacing $P$ with the self-adjoint operator 
$$\begin{pmatrix}
    0 & P^{\ast}\\
    P & 0
\end{pmatrix},$$ 
any elliptic operator \(P\) can be reformulated into a zero-order elliptic self-adjoint operator on $\Z_2$-graded bundle $E=E_0\oplus E_1$, which we still denote by $P$, the definition of ellipticity can be replaced as 
\begin{align*}
    \lVert\sigma(P,x,\xi)^2-1\rVert\to 0. \tag{\textasteriskcentered} \label{ast}
\end{align*}

\end{remark}

Henceforth, we assume that elliptic operators satisfy the estimate \eqref{ast}. This definition helps to give a direct comprehension of the $KK$-classes related to elliptic operators. 
\begin{defn}
\label{defn10}
We define that:
\begin{itemize}
    \item  The \textit{symbol class} for a $0$-order $G$-invariant elliptic operator is defined by the $KK$-pairing
$$[\sigma(P)]=[( C_0(T^{\ast}M; p^{\ast}(E)),\sigma(P)]\in KK^G(C_0(M), C_0(T^{\ast}M)),$$
because $1-\sigma^2(P)\in C_0(T^{\ast}M; p^{\ast}(E))$ is a compact multiplier on $ C_0(T^{\ast}M;p^{\ast}(E))$;
 \item The \textit{index class} is defined by the $KK$-pairing
$$[P]=[(L^2(E),P)]\in K_G^{0}(C_0(M)).$$
\end{itemize}
   
\end{defn}
By estimate \eqref{ast}, these definitions are well-defined.

The index theorem is expressed as: 
\begin{thm}[\cite{kasparov2022k}]\label{thm5}
    Let $M$ be a complete Riemannian manifold and $G$ be a locally compact Lie group acting properly and isometrically on $M$. Let $P$ be a $G$-invariant bounded elliptic operator of order $0$ on $M$, then
    $$[P]=[\sigma(P)]\otimes_{C_0(T^{\ast}M)}[\mathcal{D}_M]\in K^0_G(C_0(M)),$$
    where $[\mathcal{D}_M]\in K^0_G(C_0(TM))$ denotes the $KK$-class of the Dolbeault operator, the $KK$-product on the right-hand side is considered over $$KK^G(C_0(M),C_0(T^{\ast}M))\times KK^G(C_0(T^{\ast}M),\C)\xrightarrow{{\otimes}_{C_0(T^{\ast}M)}} K^0_G(C_0(M)).$$ 
    Here, the Dolbeault operator represents the canonical Dirac operator $D_M$ on $TM$ with the symbol
    $$\sigma_{D_M}(x,\xi,\chi,\zeta)=(\text{ext}\ \zeta+\text{int}\ \zeta)+i(\text{ext}\ \chi-\text{int}\ \chi)).$$
    The coordinates in the cotangent space $T^{\ast}_{(x,\xi)}(TM)$ is denoted as $(\chi,\zeta)$.
\end{thm}
\begin{proof}
    The proof is given in \cite{kasparov2022k}, and we briefly conclude it here.
    
    It is important to note that the operator $P$ is differential but the symbol $\sigma(P)$ is a multiplication operator. So we shall construct a rotation homotopy relating the operator and symbol.
    
    More explicitly, the Kasparov product of the Bott--Dirac element $[\mathcal{C}]$ and $[P]$ yields an operator on $TM$ that combines $\mathcal{C}\hat{\otimes} 1$ and the lifting $\tilde{P}$ of the operator $P$ to $TM$, where $\tilde{P}$ denotes a $KK$-theoretic connection of $P$. The resulting operator acts on $TM$. 
    The Kasparov product of $[\sigma_F]$ and $[\mathcal{D}_M]$ can be represented as the operator combining the lifting $\widetilde{\mathcal{D}_M}$ and $1\hat{\otimes} \sigma_F$, functioning on $T^{\ast}M$.
    
   The symbols of these two operators are homotopic through a rotation homotopy of $(\xi,\chi)$-coordinate. The Bott--Dirac element $[\mathcal{C}]$ represent the trivial element in $KK$-homology, so we see that 
    $$[P]=[\sigma(P)]\otimes_{C_0(T^{\ast}M)}[\mathcal{D}_M].$$
\end{proof}

\section{Transversally elliptic operators from $KK$-theoretic viewpoint}
\label{sec2.2}
For transversally elliptic operators, it turns out that the direct replacement of the $C_0(TM)$ space by $C_0(T_G M)$ space is not enough. As discussed in \cite{berline1996chern}, additional conditions are necessary to consider the ``good’’ symbols due to the inadequate nature of the algebra $C_0(T_G M)$ for formulating an index theorem.

Kasparov's index theorem for transversally elliptic operators in $KK$-theory is introduced in \cite{kasparov2022k}, including the concept of the tangent-Clifford symbol, which underpins our definition of transversal H-ellipticity.

Let $M$ be a complete Riemannian manifold equipped with a smooth, proper, and isometric action of a locally compact Lie group $G$. Let $x\in M$ and $f_x: G \to M$ illustrate the map defined by $g \mapsto g(x)$. The Lie algebra of the group $G$ is denoted as $\g$, and accordingly, $f'_x: \g\to T_x M$, which accounts for the derivative of $f_x$ at $e$. The $G$-action on $M$ is proper, and thus, there exists a $G$-invariant Riemannian metric on $\g_M=M\times \g$. 
   There is a $G$-invariant quadratic form $q$ on covectors $\xi \in  T_x^{\ast} M$ by $$q_x(\xi)=\lVert{f'_x}^{\ast}(\xi)\rVert^2, $$ and we have the following $G$-invariant map 
$$\varphi_x:=(f'_x {f'_x}^{\ast})^{1/2}:T_x M\to T_x M.$$ 

\begin{remark}
    Note that $\xi\in  T_x M$ is orthogonal to the $G$-orbit passing through $x$ if and only if $q_x(\xi)=0$, or if and only if $\varphi_x(\xi)=0$, denote the space of such $\xi$ by $T_G M$, it is called the \textit{transversal part of $TM$} with $G$-action. 
\end{remark}

Kasparov's definition of transversally elliptic operators is stated as follows:

\begin{defn}\label{defn7}
    A $G$-invariant pseudo-differential operator $P$ of order $0$ on a vector bundle $E$ over $M$ is called \textit{transversally elliptic} if its symbol $\sigma(P)$ satisfies the following condition: for any $\epsilon> 0$, there exists $c > 0$ such that for any $x \in M$, $\xi\in T^{\ast}_x M$,
    $$\lVert\sigma(P)^2(x,\xi)-1\rVert\leq c (1+q_x(\xi))(1+\lVert\xi\rVert^2)^{-1}+\epsilon.$$
\end{defn}
\begin{remark}
\begin{itemize}
    \item For a $G$-invariant properly supported normalized pseudo-differential operator $P$ of order $0$, the original definition of the transversally elliptic is that $\lVert\sigma(P,x,\xi)^2-1\rVert\to 0$ uniformly in $x \in M$ as $(x,\xi)\to \infty$ in $T_G M$. Note that we have replaced the operator $P$ by the normalization
    $$D=\begin{pmatrix}
        0 & P^{\ast}\\
        P & 0
    \end{pmatrix}
    ,$$ such that it acts on the $\Z_2$-graded bundle $E=E_0\oplus E_1$ of $0$-order and is self-adjoint. By \cite{kasparov2017elliptic}, Lemma 6.2, these two definitions coincide for $0$-order pseudo-differential operators.
    
    \item In Definition 7.1 of \cite{kasparov2022k}, there is a space of negative order symbol $\mathcal{S}_{tr}(M)$, defined as the sub-algebra of $C_b(TM)$ comprising functions $b(x,\xi)$ with the pseudo-differential estimate, and such that for any $\epsilon>0$,
$$|b(x,\xi)|\leq c (1+q_x(\xi))(1+\lVert\xi\rVert^2)^{-1}+{\epsilon}$$
for some certain $c>0$. The symbol $\sigma(P)$ of a transversally elliptic operator will induce a class $[\sigma(P)]\in KK^G(C_0(M), \mathcal{S}_{tr}(M))$.
\end{itemize}
\end{remark}

 We need to consider the tangent-Clifford symbol of $P$, which is vital for the transversally H-elliptic analysis.

    Let $V$ be a real vector bundle over $M$ equipped with a $G$-invariant Riemannian metric. Take $\text{Cliff}(V, Q)$ the Clifford algebra bundle associated with the quadratic form $Q(v)=\lVert v\rVert^2$ on $V$.
    The norm on fibers of the Clifford bundle $\text{Cliff}(V, Q)$ is defined by its embedding into the bundle of endomorphisms $\mathcal{L}(\Lambda^{\ast}(V))$ of the exterior algebra of $V$.

    \begin{defn}\label{def6}
    \begin{itemize}
        \item  Define $Cl_{V}(M)$ to be the complexification of the algebra of continuous sections of $\text{Cliff}(V, Q)$ over $M$ vanishing at infinity of $M$, forming a $C^{\ast}$-algebra with the sup-norm on sections. We will use a particular case of this algebra when $V=\tau=T^{\ast}M$. This algebra will be denoted $Cl_{\tau}(M)$.
        \item  For the manifold $M$ with group $G$-action, define the bundle $$\Gamma=\{f'(s)|\ s\in C_0(M;M\times \g)\},$$
where $f'$ denotes the collection of maps $f'_x:\g\to T_x M$ defined at beginning of this section. Then $Cl_{\Gamma}(M)$ is the algebra of continuous sections of the sub-algebra of $\text{Cliff}(TM)$ which is generated by vectors tangent to the orbit direction. 
    \end{itemize}

\end{defn}

Similar to Definition \ref{defn10}, we can define the index class and symbol class for transversally elliptic operators.
\begin{defn}\label{defn6}
Let $P$ be a $0$-order $G$-invariant transversally elliptic operator on $E$ over $M$.
\begin{enumerate}
    \item[(a)] The \textit{index class} $[P]$ is defined as $[(L^2(E),P)]\in K^{\ast}(C^{\ast}(G, C_0(M)))$.
    \item[(b)] The \textit{tangent-Clifford symbol class} 
    $$[\sigma^{tcl}(P)]:=[( C_0(T^{\ast}M; p^{\ast}E)\hat{\otimes}_{C_0(M)} Cl_{\Gamma}(M),\sigma^{tcl}(P))]\in KK^G(C_0(M), Cl_{\Gamma}(TM)),$$
    is defined by the operator
    $$\sigma^{tcl}(P)(x,\xi):=\sigma(P)(x,\xi)\hat{\otimes} 1+(1-\sigma(P)^2)^{1/2}\hat{\otimes}f_{\Gamma}(x,\xi)\in \mathcal{L}((C_0(T^{\ast}M; p^{\ast}E)\hat{\otimes}_{C_0(M)} Cl_{\Gamma}(M)),$$
    where 
    \begin{itemize}
        \item $Cl_{\Gamma}(TM):=Cl_{\Gamma}(M)\otimes_{C_0(M)} C_0(TM)$, 
        \item $f_{\Gamma}(x,\xi)=\varphi_x(\xi)/(1+q_x(\xi))^{1/2}$ denotes a covector field on $T^{\ast} M$, which represents the Clifford multiplication operator.
    \end{itemize}
      The tangent-Clifford symbol class is denoted by $[\sigma^{tcl}(P)]\in KK^G(C_0(M), Cl_{\Gamma}(TM))$.

\end{enumerate}
\end{defn}

\begin{remark}
 
    By the $KK^G$-equivalence in some sense between the algebras $Cl_{\Gamma}(TM)$ and $\mathcal{S}_{tr}(M)$ (see Theorem 7.10 in \cite{kasparov2022k}), there is a one-to-one correspondence between the symbol class $[\sigma(P)]\in KK^G(C_0(M),\mathcal{S}_{tr}(M))$ and the tangent-Clifford symbol class, $[\sigma^{tcl}(P)]\in KK^G(C_0(M),Cl_{\Gamma}(TM))$. We will always analyze the transversally elliptic symbols by using the tangent-Clifford symbol class $[\sigma^{tcl}(P)]$ in the following contents.
\end{remark}

    The definition of $\mathcal{S}_{tr}(M)$ space depends on the inequality estimate, rendering it unsuitable for the transversally H-elliptic case as $C^{\ast}(Gr(H))$ is noncommutative. On the other hand, the tangent-Clifford symbol class only depends on the orbit direction in $T_x M$ for each $x\in M$, which is considerable in the transversally H-elliptic case, and we will use it to give the definition in Section \ref{sec4.5}.

\begin{defn}\label{def7}
        The \textit{Clifford-Dolbeault element} $[D^{cl}_{M,\Gamma}]\in K^0(C^{\ast}(G,Cl_{\Gamma}(TM)))$ is defined to be the 
    $KK$-element
    $[(\mathcal{H},\pi, F)]$, where
    \begin{itemize}
        \item $\mathcal{H}=L^2(\Lambda_{\C}^{\ast}(TM))$;
        \item $F=\mathcal{D}_M$ is the same operator as the Dolbeault element $[\mathcal{D}_M]\in K^{0}_G(C_0(TM))$, which has been defined in Theorem \ref{thm5};
        \item The representation $\pi$ is extended from $C_0(TM)$ to $C^{\ast}(G,Cl_{\Gamma}(TM))$, by $\zeta\to (-\text{ext}(\zeta)+\text{int}(\zeta))$ for $\zeta$ a section of $\Gamma$, and the natural group action $G$ on $L^2(\Lambda_{\C}^{\ast}(TM))$.
    \end{itemize}
\end{defn}

From these definitions, the transversally elliptic index theorem is stated as follows:
\begin{thm}[\cite{kasparov2022k}]
    Let $M$ be a complete Riemannian manifold and $G$ be a locally Lie group which acts on $M$ properly and isometrically, Let $P$ represent a $G$-invariant, $L^2$-bounded transversally elliptic operator on $M$ of order $0$, then
    $$[P]=j^G([\sigma^{tcl}(P)])\otimes_{C^{\ast}(G,Cl_{\Gamma}(TM))} [\mathcal{D}_{M,\Gamma}^{cl}],$$
where $j^G$ denote the natural map from $G$-equivariant $KK$-homology to $KK$-homology of crossed product with $G$, refer to \cite{blackadar1998k}.
    The $KK$-product on the right-hand side is considered  over 
    $$KK(C^{\ast}(G, C_0(M)),C^{\ast}(G,Cl_{\Gamma}(TM))\times K^0(C^{\ast}(G,Cl_{\Gamma}(TM))\xrightarrow{\otimes_{C^{\ast}(G,Cl_{\Gamma}(TM)} } K^0(C^{\ast}(G, C_0(M))).$$
\end{thm}

%% file: chapters/chapter03.tex
\label{chap3}

In Chapter \ref{chap2}, we have seen the difference between elliptic and transversally elliptic operators. Before defining the transversally H-elliptic operators, we introduce the H-elliptic index theory and write the index theorem from the $KK$-theoretic viewpoint.

\section{The groupoid approach for H-elliptic operators}
\label{sec3.1}

In \cite{mohsen2022index}, Mohsen constructed a diagram of groupoids over $[0,1]^2$, such that the groupoid structure on each side independently forms a Lie groupoid structure. A $C^{\ast}$-algebra is defined on this groupoid, analyzed, and integrated to deduce the index theorem for H-elliptic differential operators. 

Let $M$ be a smooth closed filtered manifold.
We consider the global groupoid structure which is constructed over $[0,1]^2$ with units in $M\times [0,1]^2$ as follows:
\begin{itemize}
    \item the pair groupoid $M \times M$ if $(t, s) \in (0, 1]^2$;
    \item the tangent bundle $TM$ if $(t,s)\in (0,1]\times \{0\}$;
    \item the bundle of osculating groups $Gr(H):=\bigsqcup_{x\in M} Gr(H)_x$ if $(t,s)\in \{0\}\times(0,1] $;
    \item the bundle of abelian groups $gr(H):=\bigsqcup_{x\in M} gr(H)_x$ if $(t,s)=(0,0)$.
\end{itemize}
For some fixed $t_0$, the tangent bundle $TM$ over $(t_0,0)$ and pair groupoid $M\times M$ over $(t_0,s)$ for $s\in (0,1]$ can be attached by the convergence relation
$$x_s,y_s\in M\times M,\ V\in T_x M,\ (x_s,y_s,s)\to (V,0)\ \ \text{if}\ \   s^{-1} (y_s-x_s)\to V$$
to form the tangent groupoid $$\mathbb{T}M=TM\times \{0\}\cup M\times M\times (0,1].$$
In the same way, the groupoids constructed above can be attached by the diagram
 \[
 \begin{tikzcd}[center picture]
&gr(H)\arrow[d,no head,dashed,"s^{-1}"] \arrow[r,no head, dashed, "\delta_t^{-1}"] & TM \arrow[d,no head,dashed,"s^{-1}"]\\
     &Gr(H)\arrow[r,no head, dashed, "\delta_t^{-1}"] & M\times M,
 \end{tikzcd}
\] 
where $s^{-1}$ and $\delta_t^{-1}$ indicates the the convergences relations.

Each side of the diagram independently constitutes a groupoid structure. At $t=1$, the attached groupoid is the tangent groupoid of $T M$, and at $t=0$, it is the adiabatic groupoid constructed in Section \ref{sec1.8}. At $s=0$, it is perceived as the $C^{\infty}(M \times \R)$-submodule of $ C^{\infty}(M\times \R;TM \times \R)$ consisting of sections:
$$\big\{\mathbb{X}\in  C^{\infty}(M\times \R; TM\times \R)\big|\ \partial_t^k \mathbb{X}|_{t=0} \in C^{\infty}(M;H^k) \text{ for all } t\geq 0\big\}.$$
At $s=1$, the attached groupoid is the Heisenberg tangent groupoid, or is called the holonomy groupoid of the adiabatic foliation $\mathfrak{a}H$ associated with $H$. Further information on holonomy groupoid and adiabatic foliation is available in \cite{androulidakis2022pseudodifferential}.

We apply the restriction of the global $C^{\ast}$-algebra of the constructed groupoid over $[0,1]^2$ to each side of the square. Then, as discussed in Section \ref{sec1.8}, for a family of $C^{\ast}$-algebras over $[0,1]$, with $B$ fibered over $(0,1]$ and $A$ over $0$, we obtain an exact sequence that induces a map in $K$-theory or $KK$-theory that
$$(ev_1)_{\ast}\circ (ev_0)_{\ast}^{-1}:K_{\ast}(A)\to  K_{\ast}(B).$$
Using this construction for the $C^{\ast}$-algebra of the groupoids over each side, by Lemma \ref{lemma5}, we can obtain a commutative diagram:

 \[
 \begin{tikzcd}[center picture]
&K^0(gr(H)^{\ast})\arrow[d,"\mu_{t=0}"] \arrow[r, "\mu_{s=0}"] & K^0(T^{\ast}M) \arrow[d,"\mu_{t=1}"]\\
     &K_0(C^{\ast}(Gr(H)))\arrow[r, "\mu_{s=1}"] & \Z.
 \end{tikzcd}
 \]
 
\begin{remark}
    Note that the $\mu_{s=0}$ is induced by the attaching between $gr(H)^{\ast}$ and $T^{\ast}M$ facilitated by $\delta_t^{-1}$. Subsequently, we explore the map between the $KK$-groups of $C^{\ast}$-algebras is induced by the isomorphism from the bundle $TM$ to $gr(H)$ via orthogonal decomposition---such an isomorphism is called a splitting isomorphism (\cite{van2017tangent}). The induced map in $K$-theory remains invariant regardless of the selected isomorphism, leading to the consistency of these two induced maps through direct computation.
\end{remark}

\begin{lem}
    The $\mu_{t=0}$ map is the Connes-Thom's isomorphism in Corollary \ref{cor3}.
    The $\mu_{t=1}$ map is the ordinary analytical index map given by Connes in \cite{connes1994noncommutative}. The map $\mu_{s=1}$ is the analytical index map $\text{Ind}_{H^{\bullet}}$ in the Heisenberg calculus, which is defined in Section \ref{sec1.8}. 
\end{lem}
\begin{proof}
    By the concrete construction for Connes-Thom's class before Lemma \ref{lem1}, the $\mu_{t=0}$ map is the Connes-Thom's isomorphism. The $\mu_{t=1}$ is same as the Atiyah-Singer analytical index map $\text{Ind}_a$ defined in Lemma 2.5.6 in \cite{connes1994noncommutative}. The last part is proved in Section 3 in \cite{van2010atiyah1}.
\end{proof}

The commutativity of the diagram and Lemma 4.1 in \cite{mohsen2022index} implies that
\begin{thm}[\cite{mohsen2022index}]\label{thm10}
The index map $\text{Ind}_{H^{\bullet}}:K_0(C^{\ast}(Gr(H))\to \Z$ can be evaluated as 
$$\text{Ind}_{H^{\bullet}}=\text{Ind}_{a}\circ\mu_{s=0}\circ \mu_{t=0}^{-1},$$
where $\mu_{t=0}^{-1}:K_0(C^{\ast}(Gr(H)))\to K^0(gr(H)^{\ast})$ representing Connes-Thom's isomorphism. 
\end{thm}
This constitutes the index theorem for H-elliptic differential operators.
 An alternative formulation of this index theorem will be given in the next section. 

\section{Mohsen's Index theorem for H-elliptic operators from $KK$-theoretic viewpoint}
\label{sec3.2}
H-elliptic operators can also be analyzed using Kasparov's $KK$-theory. In this section, we transform the index theorem for H-elliptic differential operators (Theorem \ref{thm10}), to be expressed as a $KK$-product representation.

In the $G$-invariant cases, there is a Lie group $G$ action on $TM$ induced by the $G$-action on $M$. We assume that $G$ preserves the H-filtration structure. Then, we have
\begin{lem}
    There is a natural $G$-action on the bundle $gr(H)$ by $g[v]=[gv]$ for arbitrary $g\in G$ and $v\in H^i_x$. It induces a $G$-action on the bundle of groups $Gr(H)$ by the exponential map. 
\end{lem}

For a general $G$-invariant H-elliptic differential operator $P_0$, we consider the bounded transform $$P=
\begin{pmatrix}
    0 & P_0^{\ast}/(1+P_0 P_0^{\ast})^{1/2} \\
    P_0/(1+P_0^{\ast}P_0)^{1/2} & 0 
\end{pmatrix}
,$$ 
which is a $G$-invariant Fredholm operator of order $0$, and for $a\in C_0(M)$, $[a, P]$ and $a(1-P^2)$ are operators with negative order symbols, so they are both compact.
Hence, we can define that
\begin{defn}
     \textit{The index class} $[P_0]$ of $P_0$ is defined to be the $KK$-class
     $[(L^2(E_0\oplus E_1), P)]\in K^0_G(C_0(M))$. 
\end{defn}

\begin{remark}
    Henceforth, we will always assume that the considered operator $P_0$ is transformed as above to a self-adjoint operator $P$ of $0$-order. We identify the index classes $[P]=[P_0]$.
\end{remark}

By Proposition \ref{prop1} , a $0$-order H-elliptic operator $P$ induces the symbol class $$[\sigma_H(P)]\in KK(C_0(M),C^{\ast}(Gr(H))).$$ 
and in the $G$-invariant case,
$$[\sigma_H(P)]\in KK^G(C_0(M),C^{\ast}(Gr(H))).$$ 
can be defined in a similar manner.

The commutative diagram in Section \ref{sec3.1} can be obtained for $KK((C_0(M)), \cdot)$ using a similar argument, then we have the following diagram:
 \[
 \begin{tikzcd}[center picture]
&KK(C_0(M), C_0(gr(H)^{\ast}))\arrow[d,"\mu_{t=0}"] \arrow[r, "\mu_{s=0}"] & KK(C_0(M), C_0(T^{\ast}M)) \arrow[d,"\mu_{t=1}"]\\
     &KK(C_0(M), C^{\ast}(Gr(H)))\arrow[r, "\mu_{s=1}"] & KK(C_0(M),\C).
 \end{tikzcd}
 \]
The Connes-Thom's isomorphism can be represented by a $KK$-product with a certain element  $[\mathcal{E}_0]\in KK(C^{\ast} (Gr(H)), C_0(gr(H)^{\ast}))$. Let $$[\mathcal{E}]=[\mathcal{E}_0]\otimes_{C^{\ast}(gr(H)^{\ast})} [ev_0]^{-1}\otimes [ev_1]\in KK(C^{\ast} (Gr(H)),C_0(T^{\ast}M)),$$ where $[ev_1]$ and $[ev_0]$ represent the $KK$-element given by the $C^{\ast}$-algebra homomorphisms $ev_1$ and $ev_0$ in $s=0$ side.

\begin{thm}
    The analytical index map $\text{Ind}_{H^{\bullet}}$ for H-elliptic operators in Proposition \ref{prop1} can be expressed by $KK$-products as $$\text{Ind}_{H^{\bullet}}=\otimes_{C^{\ast}(Gr(H))} [\mathcal{E}]\otimes_{C_0(T^{\ast}M)} [\mathcal{D}_M].$$ 
\end{thm}
\begin{proof}
    The Kasparov product with $[\mathcal{E}]$ facilitates an isomorphism between $KK(C_0(M), C^{\ast}(Gr(H))$ and $KK(C_0(M), C_0(T^{\ast}M))$, which represents the composition of the Connes-Thom's isomorphism and the induced isomorphism of $K^0$ between $T^{\ast}M$ and $gr(H)^{\ast}$, i.e.,
$$\mu_{s=0}\circ\mu^{-1}_{t=0}=\otimes_{C^{\ast}(Gr(H))} [\mathcal{E}].$$
    By Theorem \ref{thm5}, the Kasparov product with $[\mathcal{D}_M]$ gives the ordinary analytical index map, i.e., $$\text{Ind}_a=\mu_{t=1}=\otimes_{C_0(T^{\ast}M)} [\mathcal{D}_M].$$
    Apply these two equations to Theorem \ref{thm10}, we see that
    \begin{equation*}
        \begin{aligned}
\text{Ind}_{H^{\bullet}}&=\text{Ind}_{a}\circ\mu_{s=0}\circ \mu_{t=0}^{-1}\\
        &=\otimes_{C^{\ast}(Gr(H))} [\mathcal{E}]\otimes_{C_0(T^{\ast}M)} [\mathcal{D}_M].
        \end{aligned}
    \end{equation*}
\end{proof}

By the natural construction of $G$-action on $gr(H)$, the Connes-Thom's isomorphism and the induced isomorphism $\mu_{s=0}$ from $gr(H)$ to $TM$ is $G$-invariant.
Consequently, 
\begin{prop}
    $[\mathcal{E}]$ represent a $G$-invariant $KK$-class in $KK^G(C^{\ast}(Gr(H)),C_0(T^{\ast}M))$. The Kasparov product with $[\mathcal{E}]$ provides an isomorphism between $KK^G(C_0(M), C^{\ast}(Gr(H))$ and $KK^G(C_0(M),C_0(T^{\ast}M))$.
\end{prop}

Thus, we can give an equivalent statement of the index theorem for H-elliptic operators (Theorem \ref{thm10}) in $KK$-theory as follows:
\begin{thm}\label{thm2}
     Let $M$ be a complete filtered Riemannian manifold and $G$ be a locally compact Lie group, which isometrically acts on $M$ and preserves the filtered structure. Let $P$ is an H-elliptic $G$-invariant differential operator on $M$ of order $0$, then we have
    $$[P]=[\sigma_H(P)]\otimes_{C^{\ast}(Gr(H))} [\mathcal{E}]\otimes_{C_0(T^{\ast}M)}[\mathcal{D}_M],$$
    where $[\mathcal{D}_M]\in K^0_G(C_0(TM))$ denote the $KK$-class of the Dolbeault operator. \\
The $KK$-product on the right-hand side is executed over 
\begin{equation*}
    \begin{aligned}
KK^G(C_0(M),C^{\ast}(Gr(H))&\otimes_{C^{\ast}(Gr(H))} KK^G(C^{\ast}(Gr(H)),C_0(T^{\ast}M))\\
&\otimes_{C_0(T^{\ast}M)} KK^G(C_0(T^{\ast}M),\C).
    \end{aligned}
\end{equation*}

\end{thm}

\begin{remark}  \label{rmk2}
 For the multiplier $\sigma_H(P)$ of $C^{\ast}(Gr(H))$, it represents an element of $KK(\C, C^{\ast}(Gr(H)))$, its $KK$-product with $[\mathcal{E}]$ lies in $K^0(T^{\ast}M)$ and is represented by an elliptic symbol $\sigma'(P)(x,\xi)$ on $C_0(TM)$. The construction from $P$ to $\sigma'(P)$ is natural, however, it is quite difficult to give an explicit formula for this elliptic symbol from the H-elliptic operator $P$. Even in the simplest case that $M$ is a contact manifold, the construction of this elliptic symbol is complicated, see \cite{baum2014k}. That is why it is difficult to analyze the H-elliptic index theory explicitly in $KK$-theory.

\end{remark}

\section{Index theorem for H-elliptic operators from $KK$-theoretic viewpoint}
\label{sec3.3}
In this section, we will see the index theory for hypoelliptic operators discussed in \cite{kasparov2024coarsepseudodifferentialcalculusindex}. Note that in our case that $M$ is a filtered manifold, we always consider H-elliptic operators.
First, we introduce some basic $KK$-elements.

    Cover $M$ with a smooth family of small balls $\{U_x\}$ with centers $x\in M$ of radius $r_x$. Assume that the radii $r_x$ vary smoothly over $M$.  This family of balls actually defines an open neighborhood $U$ of the diagonal in $M\times M$. We assume that both coordinate projections of the closure of $U$ into $M$ are proper maps.
    \begin{defn}
  Recall that in Definition \ref{def6}, we have defined the algebra $Cl_{\tau}(M)$. Define the radical covector field 
    $\Theta_x(y)=\rho(x,y)d_y\rho(x,y)/r_x$
    with $y\in U_x$.
    Consider $\Theta_x(y)$ as an element of the Clifford algebra fiber of $Cl_{\tau}(U_x)$ at the point $y\in U_x$. By definition $\Theta_x^2-1\in C_0(U_x)\subset Cl_{\tau}(U_x)$, so globally over $M$, the family of Clifford multiplications by covector fields $\Theta_x$ defines an element
    $[\Theta_{M}]\in KK^G(C_0(M),C_0(U)\cdot C_0(M)\otimes Cl_{\tau}(M))$, and consequently, an
element of $KK^G(C_0(M),C_0(M)\otimes Cl_{\tau}(M))$. It is called the \textit{local dual Dirac element}.

\end{defn}

\begin{defn}\label{defi4}
The \textit{Dirac element} $[d_M]\in K^0_G(Cl_{\tau}(M))$ is defined as follows. Let $H=L^2(\Lambda^{\ast}(M))$ be the Hilbert space of complex-valued $L^2$-forms on $M$ graded by the even-odd form decomposition. The homomorphism $Cl_{\tau}(M)\to B(H)$ is given on (real) covector fields by the Clifford multiplication operators $v\mapsto \text{ext}(v)+\text{int}(v)$. The (unbounded) operator $d_M$ is the operator of the exterior derivation on $H$. The operator $D_M=  d_M + d^{\ast}_M$ is essentially self-adjoint. We define $[d_M]$ as the pair 
    $(H, D_M(1+D_M^2)^{-1/2})$.

\end{defn}
\begin{thm}[\cite{kasparov1988equivariant}, Theorem 4.8] Let $M$ be a $G$-manifold. We have
    $$[\Theta_M]\otimes_{Cl_{\tau}(M)}[d_M]=1_M \in KK^G( C_0(M),C_0(M)).$$
\end{thm}

From 5.1 in \cite{kasparov1988equivariant}, for a Lie group $\G$ with maximal compact subgroup $\mathcal{K}$, there exists an element $\eta_{\G/\K}\in K^{\G}_0(Cl_{\tau}(\G/\K))$ which satisfies the following two equivalent conditions:
$$[d_{\G/\K}]\otimes_{\C}[\eta_{\G/\K}]=1_{Cl_{\tau}(\G/\K)}\in KK^{\G}(Cl_{\tau}(\G/\K), Cl_{\tau}(\G/\K)),$$
$$1_{C_0(\G/\K)}\otimes_{\C} [\eta_{\G/\K}]\in [\Theta_{\G/\K}].$$
Such element $[\eta_{\G/\K}]$ is always represented by a pair $(Cl_{\tau} (\G/\K)), c(\eta))$, where $\eta$ is a covector field on $\G/\K$ and $c$ is the Clifford multiplication. 

In the case that $\G=Gr(H)_x$ for $x\in M$, we see that it is nilpotent and simply-connected, so $\K$ will always be equal to $0$.
\begin{defn}
\begin{itemize}\label{defi2}
    \item There is a continuous (in $x$) family of covector fields $\eta_{Gr(H)_x}$ given by the above inductive construction. We denote this family of covector fields by $[\eta_{Gr(H)}]\in \RKK^{Gr(H)}(M; C_0(M), Cl_{\tau}(Gr(H)))$.
    Apply the family of homomorphisms $j^{Gr(H)_x}$ to
the element $[\eta_{Gr(H)}]$ pointwise for each $x$. The second variable of the resulting $\RKK$-group corresponding to the family of $C^{\ast}$-algebras $\{C^{\ast}(Gr(H)_x,$ $ Cl_{\tau}(Gr(H)_x))\}_{x\in M}$, which is isomorphic to $\mathcal{K}(L^2(Gr(H)_x)\otimes Cl_{\tau_x})$. By the Morita equivalence between $\mathcal{K}(L^2(Gr(H)_x)$ and $\C$, we obtain an element of the group
$\mathcal{R}KK(M; C^{\ast}(Gr(H)), Cl_{\tau}(M))$. We denote it by 
$[\eta_{C^{\ast}(Gr(H))}]$.
\item  For any $x \in M$, consider the Hilbert module $H_x=L^2(T_x^{\ast}M)\otimes Cl_{\tau_x}$. The \textit{Dirac operator} on this Hilbert module is defined by $D_x=\sum_{x=1}^{dim M}(-i)c(e_k)\partial/\partial \xi_k$
    in any orthonormal basis $\{e_k\}$ of $T^{\ast}_x M$, where $c(e_k)$ are the left Clifford
multiplication operators and $\xi_k$ are the coordinates of $T^{\ast}_x M$ in the basis $\{e_k\}$.
Denote $F_x=D_x (1+D_x^2)^{-1/2}$ on $H_x$. The family of Hilbert modules $H_x$, parametrized by $x\in M$, defines a Hilbert module of continuous
sections over the algebra $Cl_{\tau}(M)$, vanishing at infinity of $M$, which will be denoted
by $\mathcal{E}$. The grading of $\mathcal{E}$ is defined by the grading of the algebra $Cl_{\tau}(M)$. The family of operators $F_x$ gives an operator $\Phi$ of degree $1$ on $\mathcal{E}$. The algebra $C_0(TM)$ acts on $\mathcal{E}$ on the left by multiplication. We define the \textit{fiberwise Dirac element} $[d_{\xi}]\in KK^G(C_0(TM),Cl_{\tau}(M))$ by the pairing $(\mathcal{E},\Phi)$.
\item The \textit{Dolbeault element} $[D'_{M}]\in K^0(C^{\ast}(Gr(H)))$ is defined to be the $KK$-product of $[\eta_{C^{\ast}(Gr(H))}]\in \mathcal{R}KK(M; C^{\ast}(Gr(H)), Cl_{\tau}(M))$ and  $[d_{M}]\in K^0(Cl_{\tau}(M)))$.
\end{itemize}
\end{defn}

\begin{remark}
\label{rmk7}
By Theorem 2.10 in \cite{kasparov2017elliptic}, the old Dobeault element $[D_M]=[d_{\xi}]\otimes _{Cl_{\tau}(M)}[d_M]$, so we see that the new defined Dolbeault element $[D'_{M}]$ is taken by replacing $[d_{\xi}]$ with $[\eta_{C^{\ast}(Gr(H))}]$. We can understand that $[\eta_{C^{\ast}(Gr(H))}]$ is the Fourier transform of $[d_{\xi}]$ in some sense.

\end{remark}

Now we state the index theorem for H-elliptic operators by Kasparov in \cite{kasparov2024coarsepseudodifferentialcalculusindex}.
\begin{defn}
    The \textit{Clifford cosymbol} of an H-elliptic operator $P$ is defined as
    $$[\sigma_H^{cl}(P)]:=[\sigma_H(P)]\otimes_{C^{\ast}(Gr(H))} [\eta_{C^{\ast}(Gr(H))}]\in KK(C_0(M), Cl_{\tau}(M)).$$
\end{defn}
\begin{thm}[Inverse Clifford Index Theorem]
    let $P$ be a $0$-order properly supported H-elliptic operator on $M$, then 
    $$[\sigma_H^{cl}(P)]=[\Theta_M]\otimes_{C_0(M)} [P]\in KK(C_0(M), Cl_{\tau}(M)).$$
\end{thm}
\begin{thm}[H-elliptic Index Theorem]
    Let $P$ be a $0$-order properly supported H-elliptic operator on $M$, then the formula for the index $[P]$ of this operator is
    $$[P]=[\sigma_H(P)]\otimes_{C^{\ast}(Gr(H))} [D'_M]\in K^{0}(C_0(M)).$$
\end{thm}

\begin{remark}
    Comparing with Theorem \ref{thm2}, we see that this newly defined Dolbeault element $[D'_M]$ is actually the product of the Connes-Thom element $[\mathcal{E}]$ and the old Dolbeault $[D_M]$.
\end{remark}

Finally, we conclude the index classes and symbol classes in different cases, which lead to a basic expectation for transversally H-elliptic operators. The construction of symbol and index classes in $KK$-homology, and their correlation to the analytical and topological index of the operator on a locally compact manifold are summarized as follows:

\begin{itemize}
    \item The index class, expressed as the $KK$-pairing $[(L^2(E),P)]\in KK^G(A,\C)$, with different preference of the $C^{\ast}$-algebra $A$ in varying cases. The index class represents the analytical index for operators.
    \item The symbol class is represented as the $KK$-pairing $[( C_0(M;E)\otimes_{C_0(M)} B, \sigma(P))]\in KK^G(C_0(M),B)$, where $B$ is selected to be the symbol algebra of ``negative order’’ terms in different cases. It is constructed from the symbol depending on the topological properties of $P$ and $M$, such symbol class can be understood as the topological index. Especially, in the transversally elliptic case, the ``negative order’’ algebra $S_{tr}(M)$ is replaced by the $KK$-equivalent Clifford algebra $Cl_{\Gamma}(TM)=C_0(TM)\otimes_{C_0(M)} Cl_{\Gamma}(M)$.
\end{itemize}

Thus, in the transversally H-elliptic case defined in the following section, we expect to construct some appropriate algebras $A$, $B$ and construct the symbol class and index class in a similar fashion:
$$[P]\in KK^G(A, \C),\ [\sigma_H(P)]\in KK^G(C_0(M), B).$$

%% file: chapters/chapter04.tex
\label{chap4}
From now on, we begin to construct the definition of transversally H-elliptic operators from H-elliptic theory.

\section{Basic notions}
\label{sec4.1}
To articulate the concept of transversally H-elliptic operators, preliminary notions related to the transversally elliptic operators are required.

Let $M$ be a complete filtered Riemannian manifold equipped with a smooth, proper, and isometric action of a locally compact Lie group $G$. Recall that for each point $x\in M$, we have the mapping $f_x: G \to M$, $g \mapsto g(x)$. The derivative of $f_x$ at $e\in G$ is denoted by $f_x':\g\to T_x M$.

To transform the concepts in transversal elliptic theory to transversal H-elliptic theory, we need to construct a natural isomorphism between the tangent bundle $TM$ and the bundle of osculating Lie algebras $gr(H)=\bigsqcup_x gr(H)_x$.
\begin{defn}\label{defn5}
    Define $\psi_x: T_x M\to gr(H)_x$ to be the orthogonal bundle map, which maps $v=v_0+v_1+\dots+v_n\in T_x M$, where all $v_i, v_j$ are orthogonal in the Riemannian structure of $M$ and $v_i\in H^i$ for each $i$, to $$[v]:=[v_0]+[v_1]+\dots+[v_n]\in gr(H)_x=\bigoplus_{i=1}^n H^i_x/H^{i-1}_x.$$
We set
    $$h_x:=\psi_x\circ f_x':\g\to T_x M\to gr(H)_x.$$ 
\end{defn}

It induces a $G$-invariant quadratic form $\tilde{q}$ on $\xi \in  gr(H)_x^{\ast}$ by $$\tilde{q}_x(\xi)=\lVert h_x^{\ast}(\xi)\rVert^2, $$
and we obtain a $G$-invariant map $\tilde{\varphi}_x=(h_x h_x^{\ast})^{1/2}:gr(H)^{\ast}_x\to gr(H)^{\ast}_x$.
\begin{remark}
    Note that $\xi\in  gr(H)^{\ast}_x$ corresponds to covectors that are orthogonal to the orbit passing through $x$ if and only if $\tilde{q}_x(\xi)=0$, or if and only if $\tilde{\varphi}_x(\xi)=0$.
\end{remark}

Let $Y_1, \dots, Y_n$ be an orthonormal basis for $\g$ and the same notation is used for the corresponding left-invariant vector fields on $G$. Let $\tilde{Y_1},\dots,\tilde{Y_n}$ be the corresponding first-order differential operators (Lie derivatives by vector fields on $M$ corresponding to $Y_i$'s).
\begin{defn}\label{defn4}
\textit{The orbital Laplacian operator} $\Delta_G$ with ordinary symbol $\sigma(\Delta_G,x,\xi)={q}_x(\xi)$ is defined by
    $$\Delta_G:=\sum_{i} -\tilde{Y_i^2}.$$
\end{defn}

\section{Construction of leaf-wise Dirac operator}
\label{sec4.2}
The difference between H-elliptic and transversally H-elliptic operators lies in the leaf-wise part of the $G$-action. So we need a leaf-wise elliptic operator, the orbital Dirac operator, so that it will intertwine with transversally H-elliptic operators in some sense to become H-elliptic operators.

In this section, we give the construction of the orbital Dirac operator with respect to the $G$-action on $M$ in \cite{kasparov2022k}. The construction can be derived by the following three steps:
\begin{itemize}
    \item For each $x\in M$, we first construct a canonical differential Dirac operator on the quotient $G/M_x$, where $M_x$ is the stability subgroup of $G$ at $x\in M$. The Lie algebra of the stability subgroup $M_x$ will be denoted $\mathfrak{m}_x$.
\begin{lem}[\cite{kasparov2022k}]
The space of $L^2$-sections of the vector bundle
$\Lambda^{\ast}(G/M_x)$ is isomorphic to the $M_x$-invariant subspace $(L^2(G) \otimes \Lambda^{\ast}(\mathfrak{m}^{\bot}_x))^{M_x}$, where $M_x$ acts on $L^2(G)$ by right translations and $\Lambda^{\ast}$ means the exterior algebra.
\end{lem}

\begin{defn}
    For any $v \in \g$, let $g_v(t)$ be the one-parameter subgroup of $G$ corresponding
to the vector $v$.
The left-invariant differential operator $\partial/\partial v$ on smooth elements of $L^2(G)$ is defined to be the infinitesimal right translation by $g_v(t)$ on $L^2(G)$.
\end{defn}

\begin{remark}
     Note that when $v \in \mathfrak{m}_x$, the differential operator $\partial/\partial v$ is $0$ on $C_c^{\infty}(G/M_x)$ and also on $(C_c^{\infty}(G) \otimes \Lambda^{\ast}(\mathfrak{m}^{\bot}_x))^{M_x}$.
\end{remark}
Choose any basis $\{\tilde{v}_k\}$ in $\mathfrak{m}^{\bot}_x$ and the dual basis $\{v_k\}$ in $\g/\mathfrak{m}_x$. Define
the differential Dirac operator on $(C_c^{\infty}(G) \otimes \Lambda^{\ast}(\mathfrak{m}^{\bot}_x))^{M_x}$ by the formula
$$\mathcal{D}_{G/M_x}=-i\sum_k (\text{ext}(\tilde{v}_k)+\text{int}(\tilde{v}_k))\frac{\partial}{\partial v_k}.$$
\begin{lem}[\cite{kasparov2022k}]\label{lem7}
    For all $x\in M$, $\mathcal{D}_{G/M_x}$ are elliptic, and their principal symbols are given by $\sigma_{\mathcal{D}_{G/M_x}} = \text{ext} (\eta) + \text{int} (\eta)$.
\end{lem}
\begin{proof}
    By direct calculation, the principal symbol of $\mathcal{D}_{G/M_x}$ at $\eta$ is 
    $$-i\sum_k i\langle \eta, v_k\rangle (\text{ext}(\tilde{v}_k)+\text{int}(\tilde{v}_k))=\text{ext} (\eta) + \text{int} (\eta).$$
    It is invertible for aribitrary $\eta\neq 0$, thus $\mathcal{D}_{G/M_x}$ is elliptic.
\end{proof}
\item Then we construct the Dirac operator acting on sections of the vector bundle $\Lambda^{\ast}(\mathcal{O}_x)$ for each orbit $\mathcal{O}_x \cong G/M_x$. 
Using the homeomorphisms $f_x : G/{M_x} \cong \mathcal{O}_x$, we can push forward all these operators $\D_{G/M_x}$ to the orbits $\mathcal{O}_x$ in $M$. The operator $\mathcal{D}$ on the orbit $\mathcal{O}$ will be denoted $\mathcal{D}_{\mathcal{O}}$. By the construction above, it depends only on the orbit, not on the point $x$ of the orbit.
\item Using the embedding $\mathfrak{m}^{\bot}_x \cong T_x({\mathcal{O}}_x) \subset T_x(M)$, one can redefine the operators $\mathcal{D}_{{\mathcal{O}}_x}$ as operators acting on the bundles $\Lambda^{\ast}(M)|_{{\mathcal{O}}_x}$. This will give the operators
$\mathcal{D}_{\mathcal{O}}$ on sections of the bundle $\Lambda^{\ast}(M)$ over the orbit ${\mathcal{O}}$. 
\begin{defn}
    The \textit{orbital Dirac operator} $\D_G$ on $\Lambda^{\ast}(M)$ is formed by the family of operators $\{\mathcal{D}_{{\mathcal{O}}_x}\}$ on $\Lambda^{\ast}(M)|_{{\mathcal{O}}_x}$.
\end{defn}
\end{itemize}

The above statements can be concluded as follows:
\begin{prop}\label{prop2}
    There exist an orbital Dirac operator $\mathcal{D}_G$ on $\Lambda^{\ast}(M)$ whose ordinary principal symbol is $\text{ext }(\varphi_x(\xi))+\text{int }(\varphi_x(\xi))$.
\end{prop}
\begin{proof}
The existence has been verified already. The ordinary principal symbol can be calculated as follows. By Lemma \ref{lem7}, the principal symbol of $D_{\mathcal{O}}$ on $\Lambda^{\ast}(M)|_{{\mathcal{O}}}$ can be written as 

$$\sigma({D_\mathcal{O}})(x,\xi)= \text{ext}({f'_x}^{\ast}(\xi))+\text{int}({f'_x}^{\ast}(\xi)).$$
Note that ${f'_x}^{\ast}(\xi)$ lies in $\g_x^{\ast}$, and we want to transform it back to $T_x M$. Since $f'_x$ is a linear map from $\g$ to $T_x M$, the image of the linear map $f'_x$ and the image of its adjoint ${f'_x}^{\ast}$ are isometrically isomorphic by
$$(f'_x{f'_x}^{\ast})^{-1/2}f'_x: \text{Im}(f'_x)\to \text{Im}({f'_x}^{\ast}).$$
       Applying this isometry to ${f'_x}^{\ast}$, we can replace it by $$(f'_x{f'_x}^{\ast})^{-1/2}f'_x= (f'_x{f'_x}^{\ast})^{1/2}(\xi)=\varphi_x(\xi)$$
        which is defined in Section \ref{sec2.2}.
        Therefore, we have
        $$\sigma(\D_G)(x,\xi)=\text{ext}(\varphi_x(\xi))+\text{int}(\varphi_x(\xi)).$$

\end{proof}

\section{Adjoint in different viewpoints}
\label{sec4.3}
The H-principal cosymbol $\sigma_H(\D_G,x)$ at one point $x$ should be an operator on $\G=Gr(H)_x$. It defines an operator on $C_c^{\infty}(\G)$ and extends to $C^{\ast}(\G)$. In order to give a well-defined definition for transversally H-elliptic operators (Definition \ref{def1}), it is reasonable to show the essential self-adjointness of 
 $$\sigma_H(\D_G,x)= -i \sum_k X_k (\text{ext}(e_k)+\text{int}(e_k)),$$
     where $X_k$ are elements in $gr(H)_x$ corresponding to generators of $\g$ of highest order in the filtration structure and $e_k$ are the dual to $X_k$.

We can define the adjoint operator of a left-invariant differential operator $P=\sigma_H(\D_G,x)$, for example, in the sense of $C^{\ast}$-module morphism or in the sense of universal algebra. Before we verify the essential self-adjointness of $\sigma_H(\D_G,x)$, it suffices to show that these adjoint maps $P\to P^{\ast}$ are in fact the same.

\begin{defn} There are three definitions for the \textit{adjoint} of an element in $P\in \mathcal{U}(gr(H)_x)$:
    \begin{itemize}
    \item[(a)] $P$ is identified as an element in the universal algebra $\mathcal{U}(gr(H)_x)$. On the universal algebra $\mathcal{U}(gr(H)_x)$, there is an involution map $$-{id}:gr(H)_x \to gr(H)_x$$
    which maps all $X_i$ to $-X_i$. It extended to get an anti-automorphism
    $$i:\mathcal{U}(gr(H)_x)\to \mathcal{U}(gr(H)_x).$$
    The image $P^{\ast}:=i(P)$ is called the adjoint of $P$. This construction fits with all the algebraic structures we want in H-elliptic theory and satisfies $$\sigma_H(D^{\ast})=\sigma_H(D)^{\ast}$$
    for a differential operator $D$ on $M$.
    \item[(b)]  $P$ is identified as a morphism of the Hilbert $C^{\ast}(Gr(H)_x)$-module $E=C^{\ast}(Gr(H)_x)$ itself with inner product $\langle a,b\rangle =a^{\ast}b$. Then we can define the adjoint $P^{\ast}$ by $$\langle Pf,g\rangle=\langle f,P^{\ast}g\rangle$$ by the dense range of $P$.
    \item[(c)] $P$ is identified as a left-invariant differential operator, which can be transformed to an essentially homogeneous distribution on $Gr(H)_x$ by
    $(u,f)=(Pf)(e)$ for $f\in C_c^{\infty}(Gr(H)_x)$. It corresponds to an operator $Op(u)$ by
    $$(Op(u)(f))(x)=(u,f\circ l_x).$$
    By the left-invariance of $P$, we see that $Op(u)=P$.
    The involution operator is defined by 
    $$(u^{\ast},f):=\overline{(u,f^{\ast})}.$$
    Such involution on $u$ induces a adjoint $P^{\ast}$ for the corresponding operator $P=Op(u)$ of $u$.
\end{itemize}
\end{defn}

\begin{thm}
    The above three definitions for the adjoint operator are equivalent.
\end{thm}
\begin{proof}
We prove for a general nilpotent Lie group $\G$. We consider only unimodular groups, which means that the Haar measure is both left- and right-invariant. 
  
  Recall that we have defined the convolution and involution maps on $C^{\infty}_{c}(\G)$ in Definition \ref{defn1}.
Let $f,g \in C_c^{\infty}(\G)$, the convolution $f \ast g$ is defined by
$$f\ast g(x):=\int_G f(y)g(y^{-1}x)dy=\int_G f(xy^{-1})g(y)dy.$$
The involution $\ast$-operator is defined by
$$f^{\ast}:=\overline{f(t^{-1})}.$$
On $C^{\ast}(\G)$-module $E=C^{\ast}(\G)$, the product is given by 
$$\langle f,g\rangle=f^{\ast}\ast g.$$
    $(a)\Leftrightarrow(b)$: It suffices to prove that $\langle Xf,g\rangle =\langle f,-Xg\rangle $ for any first-order differential operator $X\in \g$ and $f,g \in C_c^{\infty}(\G)$.
It can be proved by direct calculation that
\begin{equation*}
    \begin{aligned}
        \langle Xf,g\rangle +\langle f,Xg\rangle &=\int_G \overline{(Xf)(y)}g(y x) dy+\int_G \overline{f(y)}(Xg)(y x) dy\\
        &=\int_G X(\overline{f(y)}g(y x))) dy\\
        &=0
    \end{aligned}
\end{equation*}
by the invariance of the Haar measure.

$(a)\Leftrightarrow(c)$: Assume $f \in C_c^{\infty}(\G)$. It suffice to prove that for $u$ corresponds to $X\in \g$ by $(u,f)=Xf(e)$, $Op(u^{\ast})$ transform $f$ to $-Xf$.
 $Op(u^{\ast})f$ maps $x\in \G$ to 
\begin{equation*}
    \begin{aligned}
        &\overline{X(\bar{f}\circ l_{x}\circ\iota)}(e)\\
        =&d l_{x}\circ d\iota_e \circ \overline{X(\bar{f}(\cdot)})(e)\\
        =&d l_{x}\circ  (-X)(f(\cdot))(e)\\
        =&-X(f)(x).
    \end{aligned}
\end{equation*}
The second equation holds because $d\iota_e=-\text{id}_{T_e \G}$.\\
Thus the adjoint constructions in $(a),(b),(c)$ are equivalent.

\end{proof}

\section{Essential self-adjointness of $\sigma_H(\D_G)$}
\label{sec4.4}
In this section, we show the essential self-adjointness of $\sigma_H(\D_G)$, where $\D_G$ is the orbital Dirac operator on $\Lambda^{\ast}(M)$ constructed in Section \ref{sec4.2}.

For each $x\in M$, take the $C^{\ast}(Gr(H)_x)$-module $E=C^{\ast}(Gr(H)_x)\otimes \Lambda^{\ast}_x(M)$. Then the H-principal cosymbol of $\D_G$ is an unbounded operator on the Hilbert module $E$
 $$\sigma_H(\D_G,x): C^{\ast}(Gr(H)_x)\otimes \Lambda^{\ast}_x(M)\to C^{\ast}(Gr(H)_x)\otimes \Lambda^{\ast}_x(M).$$
 
In this section, we will prove the following theorem:
  \begin{thm}\label{thm7}
       For each $\lambda\geq 1$ large enough, the image of $\sigma_H(\D_G,x)\pm i\lambda$ is dense in the algebra $C^{\ast}(Gr(H)_x)\otimes \Lambda^{\ast}_x(M)$.
       
       Thus $\sigma_H(\D_G,x)$ is a regular operator as a morphism on the $C^{\ast}(Gr(H)_x)$-module $E=C^{\ast}(Gr(H)_x)\otimes \Lambda^{\ast}_x(M)$.
   \end{thm}

 By the formal adjoint argument in the last section, $$\sigma_H(\D_G,x)= -i \sum_k X_k (\text{ext}(e_k)+\text{int}(e_k))$$
    is formally self-adjoint, where $X_k$ are elements in $gr(H)_x$ corresponding to generators of $\g$ of highest order in the filtration structure and $e_k$ are the dual to $X_k$.
\begin{lem}\label{lem8}
   Let $D$ be a first-order formally self-adjoint differential operator on $Gr(H)_x$ which has finite propagation speed. Namely, $c_D<\infty$, where
   $$c_D=\sup_{z\in {Gr(H)_x}}\sup_{\xi\in T_z^{\ast}{Gr(H)_x},\lVert\xi\rVert=1}\lVert\sigma_D(z,\xi)\rVert.$$
   Then, $D$ is essentially self-adjoint on the Hilbert space $ L^{2}(Gr(H)_x) \otimes_{\C}\Lambda_x^{\ast}(M)$.
\end{lem}

\begin{proof}
 First, for each compact subset $K\subset Gr(H)_x$, there exists a family of Friedrich’s mollifiers $L^2(K)\otimes V\to L^2(Gr(H)_x)\otimes V$, which is a sequence of bounded operators with the following properties:
    \begin{itemize}
        \item[(i)]  each $F_n$ is a contraction;
        \item[(ii)] the image of each $F_n$ is contained in $C_c^{\infty}({Gr(H)_x})\otimes {\Lambda^{\ast}_x(M)}$;
        \item[(iii)] for all $v\in L^2({Gr(H)_x})\otimes {\Lambda^{\ast}_x(M)}$, $F_n v\to v$ and $F_n^{\ast}v\to v$ in the $L^2({Gr(H)_x})$-norm as $n\to \infty$;
        \item[(iv)] for any first-order differential operator $D$, sequence $[D,F_n]=[D,F_n^{\ast}]$ of operators on $C_c(K)\otimes V$ is uniformly bounded in operator norm.
    \end{itemize}
    
     Assume that $K$ is a compact subset of $\R^d$, Let $h : \R^d \to \R$ be any smooth, non-negative compactly supported function that satisfies 
$\int_{\R^d} h(x)dx=1$. For each $n \geq 1$,
define $h_n(x)=n^d h(nx)$ , and let 
\begin{equation*}
    \begin{aligned}
        F_n=F^{0}_n\otimes 1 &: L^2(K)\otimes  {\Lambda^{\ast}_x(M)}\to L^2(\R^d)\otimes {\Lambda^{\ast}_x(M)},\\
        (F^0_n u)(x):&=\int_{\R^n}h_n(x-y)u(y)dy.
    \end{aligned}
\end{equation*}

 Cover $K \subset Gr(H)_x$ by finitely many coordinate charts, and use a finite smooth
partition of unity consisting of compactly supported functions to patch together
mollifiers constructed as above.

    To prove the essential self-adjointness of $D$, we need to show that for $v, w\in L^2({Gr(H)_x})\otimes {\Lambda^{\ast}_x(M)}$ such that for all $u\in C_c^{\infty}({Gr(H)_x})\otimes {\Lambda^{\ast}_x(M)}$,
    $$\langle Du,v\rangle=\langle u,w\rangle.$$
    (say $v$ is in the maximal domain of $D$ if there is a $w$ satisfying this condition), there exists a sequence $v_n$ in $C_c^{\infty}({Gr(H)_x})\otimes {\Lambda^{\ast}_x(M)}$ converging to $v$ in $L^2({Gr(H)_x})\otimes {\Lambda^{\ast}_x(M)}$ and such that $Dv_n$ converges to $w$ in $L^2({Gr(H)_x})\otimes {\Lambda^{\ast}_x(M)}$ (say $v$ is in the minimal domain of $D$ if there is a $w$ satisfying this condition).
    
    Assume first that $v$ is compactly supported in some $K\subset Gr(H)_x$, let $(F_n)$
be a family of Friedrich’s mollifiers for $K$. Consider the sequence $F_n v$, which is in $C_c^{\infty}({Gr(H)_x})\otimes {\Lambda^{\ast}_x(M)}$ and converges in $L^2({Gr(H)_x})$-norm to $v$. For any $u\in C_c^{\infty}({Gr(H)_x})\otimes {\Lambda^{\ast}_x(M)}$,
 \begin{equation*}
 \begin{aligned}
           \langle D F_n v, u\rangle &=\langle v,F_n^{\ast}D u\rangle \\
           &= \langle v,[F_n^{\ast},D]u\rangle+ \langle v,D F_n^{\ast}u\rangle \\
           &=\langle v,[F_n^{\ast},D]u\rangle+\langle w,F_n^{\ast}u\rangle .
 \end{aligned}
 \end{equation*}
 Hence $\lVert D F_n v\rVert\leq 
\lVert [F_n^{\ast}, D] \rVert \lVert v\rVert+ \lVert w\rVert$ for all $n$, the sequence $(D F_n v)$ is uniformly bounded, thus has a convergent subsequence and still denote it by $(D F_n v)$. Since
 $$\langle u,D F_n v\rangle =\langle D u,F_n v\rangle \to \langle D u,v\rangle =\langle u,w\rangle ,$$
the limit is $w$. Here we need the following lemma:
\begin{lem}[Mazur's lemma]\label{lem5}
    Let $(X,\lVert\cdot\rVert)$ be a normed vector space and let $\{x_j\}_{j\in\N} \subset X$ be a sequence converges weakly to some $x\in X$. Then there exists a sequence $\{y_k\}_{k\in \N}\subset X$ made up of finite convex combination of the $x_j$'s of the form
$$y_k=\sum_{j\geq k} \lambda_j^{(k)}x_j$$
such that $y_k$ is strongly convergent to $x$, i.e., $\lVert y_k-x \rVert\to 0$ as $k\to\infty$.
\end{lem}

By Lemma \ref{lem5}, there is a sequence $(v_n)$ in $C_c^{\infty}({Gr(H)_x})\otimes {\Lambda^{\ast}_x(M)}$ consisting of convex combinations of the sequence, such that $v_n\to v$ and $F_n v\to w$ which strongly converges to $w$ as $n\to\infty$. Thus we have proved that when $v$ in the maximal domain is compactly supported, $v$ is in the minimal domain of $D$.

It is obvious that ${Gr(H)_x}$ is complete and simply connected and $D $ has finite propagation speed 
$c_{D} =\sup_{z\in {Gr(H)_x}}\sup_{\xi\in T_z^{\ast}{Gr(H)_x},\lVert\xi\rVert=1}\lVert\sigma_{D}(z,\xi)\rVert$.

 Let $(f_n: \R \to [0, 1])$ be a sequence of smooth, compactly supported functions such that $f(t) = 1$ for all $t \leq n$, and such that the sequence $(\sup_{t\in \R} |f_n'(t)|)$ of real numbers tends to zero as $n$ tends to infinity. Fix $z_0\in {Gr(H)_x}$, and define 
 $$g_n(z):=f_n(d(z,z_0)).$$
 Then $g_n$ is smooth and compactly supported, the sequence $\sup_{z\in Gr(H)_x} \lVert dg_n(z)\rVert_{n=0}^{\infty}$ tends to zero
as $n$ tends to infinity, and that $(g_n)$ tends to one uniformly on compact subsets
of ${Gr(H)_x}$.

For general elements $v,w\in L^2({Gr(H)_x})\otimes {\Lambda^{\ast}_x(M)}$ with the condition that $v$ is in the maximal domain, let $v_n=g_n v$. Then 
\begin{equation*}
    \begin{aligned}
        \langle {D}u,v_n\rangle &=\langle g_n {D}u, v\rangle \\
        &=\langle [g_n,{D}] u, v\rangle +\langle {D} g_n u,   v\rangle \\
        &=\langle \sigma_{D}(dg_n) u,  v\rangle +\langle g_n u, w\rangle \\
        &=\langle u,\sigma_{D}(dg_n)v+g_n w\rangle 
    \end{aligned}
\end{equation*}
for all $u\in C_c^{\infty}({Gr(H)_x})\otimes {\Lambda^{\ast}_x(M)}$. So $v_n$ is in the maximal domain of ${D}$ for all $n$. Since $v_n$ is compactly supported by compact support of $g_n$, $v_n$ is in the minimal domain of ${D}$. So we can apply the closure of ${D}$ on $v_n$ and for any $u\in C_c^{\infty}({Gr(H)_x})\otimes {\Lambda^{\ast}_x(M)}$,
\begin{equation*}
    \begin{aligned}
        \langle {D}v_n,u\rangle &=\langle  v,g_n {D} u\rangle \\
        &=\langle v,[g_n,{D}] u\rangle +\langle  v,{D}g_n u\rangle \\
        &=\langle  v, \sigma_{D}(dg_n) u\rangle +\langle w,g_nu\rangle, 
    \end{aligned}
\end{equation*}
and 
\begin{equation*}
    \begin{aligned}
    |\langle v,[g_n,{D}]u\rangle |&\leq c_{D}\sup_{z\in Gr(H)_x} \lVert dg_n(z)\rVert_{n=0}^{\infty} \lVert v\rVert_{L^2} \lVert u\rVert _{L^2}\to 0\\
    \end{aligned}
\end{equation*}
as $n\to \infty$.
\begin{equation*}
    \begin{aligned}
\lim_{n\to\infty} \langle {D}v_n,u\rangle =\lim_{n\infty}\langle w,g_n u\rangle =\langle w,u\rangle .
    \end{aligned}
\end{equation*}

Since $u$ is arbitrary, $({D}v_n)$ converges to $w$. So $v$ is in the minimal domain of ${D}$.
So the maximal domain of $D$ is contained in the minimal domain of $D$, thus $D$ is essentially self-adjoint.
\end{proof}

 For any $\kappa,\lambda\in\R,\lambda\neq 0$, let $$T_{\kappa,\lambda}:=\sigma_H(\D_G,x)+i\kappa\sigma_{\sigma_H(\D_G,x)}(z,d\rho(z))+i\lambda.$$
\begin{cor}\label{cor5}
    Assume that $2|\kappa|\cdot\lVert c_{\sigma_H(\D_G,x)}\rVert < |\lambda|$, then $T_{\kappa,\lambda}(C_c^{\infty}(Gr(H)_x)\otimes \Lambda^{\ast}_x(M))$ is dense in $ L^{2}(Gr(H)_x) \otimes \Lambda_x^{\ast}(M)$. \\
Moreover, there is an inverse $T_{\kappa,\lambda}^{-1}$ which is a bounded operator on $ L^{2}(Gr(H)_x) \otimes \Lambda_x^{\ast}(M)$.
\end{cor}
   \begin{proof}
    Apply to Lemma \ref{lem8} with $D=\sigma_H(\D_G,x)+i\kappa\sigma_{\sigma_H(\D_G,x)}(z,d\rho(z))$, it is formally self-adjoint and has finite propagation speed since $\sigma_H(\D_G,x)$ has finite propagation speed. So $D$ is essentially self-adjoint and thus $D\pm \lambda i$ has dense image in $L^{2}(Gr(H)_x) \otimes \Lambda_x^{\ast}(M)$ for $\lambda$ large enough.
    By direct calculation,
   $$\lVert T_{\kappa,\lambda}(u) \rVert^2\leq (\lambda^2-2|\kappa\lambda|\cdot \lVert c_D\rVert)\cdot \lVert u\rVert^2. $$
    So the operator $T_{\kappa,\lambda}^{-1}$ is a bounded on $ L^{2}(Gr(H)_x) \otimes \Lambda_x^{\ast}(M)$.
   \end{proof}
  
   \begin{proof}(\textit{of Theorem \ref{thm7}})
      Denote the closure of $\sigma_H(\D_G,x)$ with $L^2$-inner product by $D_1$ and the closure with $C^{\ast}$-inner product by $D_2$.
      
      Choose $\kappa$ positive and large enough so that the function of $z\in Gr(H)_x$, $e^{-\kappa\rho(z)}$ is in $L^2(Gr(H)_x)$, where $\rho(z)$ the distance function from the identity $e$ to $z\in Gr(H)_x$.
      
      Let $v\in C_c^{\infty}(Gr(H)_x)\otimes \Lambda^{\ast}_x(M)$. By Corollary \ref{cor5}, the operator $T_{0,\pm\lambda}=D_1\pm i\lambda$ and $T_{\kappa,\pm\lambda}=D_1\pm i\lambda$ are invertible on $L^{2}(Gr(H)_x) \otimes \Lambda_x^{\ast}(M)$. So we can solve the equation $(D_1\pm i\lambda)(u)=v$ and $T_{\kappa,\pm\lambda}(u)=e^{\kappa\rho(z)}v$ in $L^{2}(Gr(H)_x) \otimes \Lambda_x^{\ast}(M)$.
      By direct calculation $(D_1\pm i\lambda)(u-e^{-\kappa\rho(z)}u_1)=0$, so $u=e^{-\kappa\rho(z)}u_1$.
      Now we have $e^{-\kappa\rho(z)}$ in $L^2(Gr(H)_x)$ and $u_1$ lie in $ L^{2}(Gr(H)_x) \otimes \Lambda_x^{\ast}(M)$. Thus the product $u=e^{-\kappa\rho(z)}u_1$ lies in $L^1(Gr(H)_x)\otimes \Lambda_x^{\ast}(M)$ and its $L^1$-norm is restricted by
      $$\lVert u\rVert_{L^1\otimes \Lambda_x^{\ast}(M)}\leq \lVert e^{-\kappa\rho(z)}\rVert_{L^2} \lVert u_1\rVert_{L^2\otimes \Lambda_x^{\ast}(M)}. $$
      Thus $u\in C^{\ast}(Gr(H)_x)\otimes \Lambda_x^{\ast}(M)$.
      
      Similarly, we see that 
      $$\lVert (1-g_n)u\rVert_{C^{\ast}\otimes \Lambda_x^{\ast}(M)}\leq\lVert (1-g_n)u\rVert_{L^1\otimes \Lambda_x^{\ast}(M)}\leq \lVert e^{-\kappa\rho(z)}\rVert_{L^2} \lVert (1-g_n)u_1\rVert_{L^2\otimes \Lambda_x^{\ast}(M)}, $$
      where $g_n$ is defined in the proof of Lemma \ref{lem8}.

      So when $n\to \infty$, we see that the elements $g_n u\in C_c(G)$ converges to $u$ in the $C^{\ast}(Gr(H)_x)$-module $C^{\ast}(Gr(H)_x)\otimes \Lambda_x^{\ast}(M)$. 
    
      Let $u_n=g_n\ast g_n u$, since $g_n u$ is in $L^1$ and compactly supported, $g_n$ is smooth, we have $u_n\in C_c^{\infty}(G)\otimes \Lambda_x^{\ast}(M)$. Since $g_n\to 1$ in $L^1$ and thus in $C^{\ast}$, $1-g_n$ is uniformly bounded and decreasing, and $g_n u\to u$ also in $L^1$ and thus in $C^{\ast}$ when $n\to \infty$, we see that $u_n$ converges to $u$ in $C^{\ast}(Gr(H)_x)\otimes \Lambda_x^{\ast}(M)$.

     Since $D_2=D_1=\sigma_H(\D_G,x)$ on $C_c^{\infty}(G)\otimes \Lambda_x^{\ast}(M)$, we have 
     \begin{equation*}
         \begin{aligned}
             &(D_2\pm i\lambda)(u_n)\\
             =&(D_1\pm i\lambda)(u_n)\\
             =&g_n\ast D_1 (g_n u) \pm i\lambda u_n\\
             =&g_n\ast \sigma_{D_1}(x,dg_n) u+ g_n \ast g_n D_1 u \pm g_n\ast g_n(i\lambda u)\\
             =&g_n\ast \sigma_{D_1}(x,dg_n) u+ g_n\ast g_n v.
         \end{aligned}
     \end{equation*}
     
      Note that $\sigma_{D_1}(x,dg_n)\to 0$, 
      $$\lim_{n\to\infty}(D_2\pm i\lambda) (u_n)=v.$$
      So the images of $D_2\pm i\lambda$ are dense in $C_c^{\infty}(Gr(H)_x)\otimes \Lambda^{\ast}_x(M)$. 
      As a result the images of $\sigma_H(\D_G,x)\pm i\lambda$ are dense in the $C^{\ast}(Gr(H)_x)\otimes \Lambda^{\ast}_x(M)$, i.e., $\sigma_H(\D_G,x)$ is a regular operator as a morphism on the $C^{\ast}(Gr(H)_x)$-module $E=C^{\ast}(Gr(H)_x)\otimes \Lambda^{\ast}_x(M)$.
      \end{proof}
      
\section{Definition for transversally H-elliptic operators}
\label{sec4.5}
Now we can give the definition for transversally H-elliptic operators.
From the basic idea of transversally elliptic theory, we shall define the transversal H-ellipticity by requiring operators to be H-elliptic after intertwining with the leafwise elliptic operator $\Delta_G$.

 The fundamental definition is motivated by the tangent-Clifford symbol for transversally elliptic operators defined in Definition \ref{defn6}. 
 Recall that for a transversally elliptic operator $P$, its tangent-Clifford symbol $[\sigma^{tcl}(P)]$ is in the $KK$-group $KK^G(C_0(M),Cl_{\Gamma}(TM))=KK^G(C_0(M), Cl_{\Gamma}(M)\otimes_{C_0(M)}C_0(TM))$.
  
  Replacing $C_0(TM)\cong C^{\ast}(TM)$ by $C^{\ast}(Gr(H))$ in the group $KK^G(C_0(M), Cl_{\Gamma}(TM))=KK^G(C_0(M), Cl_{\Gamma}(M)\otimes_{C_0(M)} C_0(TM))$, we intend to derive the formal formula of the tangent-Clifford symbol for operators that can induce a $KK$-class $[\sigma_H^{tcl}(P)]$ in $$KK^G(C_0(M), Cl_{{\Gamma}}(M)\hotimes_{C_0(M)}C^{\ast}(Gr(H))).$$ 
  
  The definition of $KK$-classes requires $1-\sigma^{tcl}_H(P)^2$ to be a compact multiplier, which motivates the definition for transversally H-elliptic operators:

\begin{defn}\label{def1}
    Let $M$ be a complete filtered Riemannian manifold and $G$ be a locally compact Lie group which acts on $M$ properly, isometrically, and preserves the filtered structure. We call a pseudo-differential operator of $0$-order $P$ is \textit{transversally H-elliptic} if at each point $x\in M$, the multiplier term 
    $$(1-\sigma_H^2(P,x))(1+\sigma_H(\Delta_{G},x))^{-1}\in \mathcal{M}(E_x\otimes C^{\ast}(Gr(H)_x))$$
    is compact as operators on  $C^{\ast}(Gr(H)_x)$-modules.
\end{defn}
\begin{remark}
Note that
\begin{itemize}
    \item In the multiplier term $(1-\sigma_H^2(P,x))(1+\sigma_H(\Delta_{G},x))^{-1}$, the first term $(1-\sigma_H^2(P,x))$ is a multiplier of $E_x\otimes C^{\ast}(Gr(H)_x)$ by definition of H-cosymbol, and the second term $(1+\sigma_H(\Delta_{G},x))^{-1}$, by the essentially self-adjointness of $\sigma_H(\D_G,x)$ as a $C^{\ast}(Gr(H)_x)$ morphism in Theorem \ref{thm7}, also represent a multiplier of $C^{\ast}(Gr(H)_x)$.
    When the vector bundle $E$ is trivial (in general cases we take the local trivialization such that $\text{End}(E)$ part becomes a matrix), the condition in Definition \ref{def1} requires that for all $x\in M$, the product multiplier element $(1-\sigma_H^2(P,x))(1+\sigma_H(\Delta_{G},x))^{-1}$ lies in the $C^{\ast}$-algebra $C^{\ast}(Gr(H)_x)$. 
    \item The leaf-wise sub-Laplacian operator $\Delta_G$ is employed instead of a leaf-wise H-elliptic operator, because the leaf-wise direction contains several different order terms if no additional conditions are taken. Thus, an appropriate operator like $\Delta_G$ with H-elliptic property in the orbit direction cannot be readily presumed in the leaf-wise direction in general cases.
\end{itemize}
\end{remark}

\begin{remark}
    If the group $G$ acts trivially on $M$, then the condition will be reduced to $1-\sigma_H^2(P,x)$ is compact. In this case, $1-\sigma_H^2(P,x)$ is homogeneous up to smooth so in the pseudo-differential calculus, the compactness of a multiplier is equivalent to of negative order. The condition is equivalent to the existence of parametrix for $\sigma_H(P,x)$, which implies the Rockland condition at each point $x\in M$ and leads to the H-ellipticity, Conversely, it can be proved that the cosymbols of H-elliptic operators are unbounded Fredholm multipliers (refer to  \cite{joachim2003unbounded} for the definition). Considering the bounded transform expressed in Section \ref{sec3.2}, then $1-\sigma^2_H(P)$ is a compact multiplier.
\end{remark}

The symbol class and index class can be defined for such transversally H-elliptic operators.
\begin{defn}\label{defi1}
   We define the \textit{tangent-Clifford symbol} $\sigma_H^{tcl}(P)$ by
$$\sigma^{tcl}_H(P):=\sigma_H(P)\hat{\otimes} 1+(1-\sigma_H(P)^2)^{1/2}\hat{\otimes} f'_{{\Gamma}}\in \mathcal{L}( C_0(M;E)\hat{\otimes}_{C_0(M)}C^{\ast}(Gr(H))\hotimes_{C_0(M)} Cl_{{\Gamma}}(M)).$$
Here $f'_{{\Gamma}}$ is a multiplier of $ C^{\ast}(Gr(H)) \hotimes_{C_0(M)} Cl_{{\Gamma}}(M)$ defined as $\sigma_H(\mathcal{D}_G
)\hotimes (1+\sigma_H(\Delta_G))^{-1/2}$. By the discussion of $\sigma_H(\D_G)$ in Section \ref{sec4.4}, at each point $x$, $\sigma_H(
\mathcal{D}_G,x)$ consists of the sum of the tensor product of an element in $gr(H)_x$, which represents a multiplier on $C^{\ast}(Gr(H)_x)$, and a Clifford-product element on $\Lambda^{\ast}(T_x M)$. The Clifford-product part is leaf-wise, so it represents a multiplier on $Cl_{{\Gamma_x}}(M)$. The remaining part $(1+\sigma_H(\Delta_{G},x))^{-1/2}$ is a multiplier on $C^{\ast}(Gr(H)_x)$. As a result, we can take $f'_{\Gamma}=\sigma_H(\mathcal{D}_G
)\hotimes (1+\sigma_H(\Delta_G))^{-1/2}$ as a multiplier on $ C^{\ast}(Gr(H)) \hotimes_{C_0(M)} Cl_{{\Gamma}}(M)$.
\end{defn}

By the construction of $\sigma_H^{tcl}(P)$, we see that
\begin{prop} Modulo compact operators on $Gr(H)_x$ for each $x$, we have
    $$1-(\sigma^{tcl}_H(P))^2=(1-\sigma^2_H(P))(1-{f'_{\Gamma}}^2)=(1-\sigma_H^2(P))(1+\sigma_H(\Delta_{G}))^{-1}$$
\end{prop}
\begin{proof}
By direct calculation,
 $$1-(\sigma^{tcl}_H(P))^2=1-\sigma^2_H(P)\hat{\otimes} 1-(1-\sigma_H(P)^2)\hat{\otimes} {f'_{\Gamma}}^2-(1-\sigma_H(P)^2)^{1/2}[\sigma_H(P)\hat{\otimes} 1,1\hat{\otimes} f'_{\Gamma}].$$
    Since the cosymbol $\sigma_H(P)\hat{\otimes} 1$ acts on bundle $E$ and $1\hat{\otimes} \sigma_H(\D_G)$ represents on $Cl_{{\Gamma}}(M)$, the commutator  
    $[\sigma_H(P)\hat{\otimes} 1,1\hat{\otimes} f'_{\Gamma}]=0$ modulo negative order terms, i.e. modulo compact operators on $Gr(H)_x$ for each $x\in M$ because all these operators are left-invariant on $Gr(H)_x$. So, modulo compact operators, we have 
    \begin{equation*}
       \begin{aligned}
           1-(\sigma^{tcl}_H(P))^2=&1-\sigma^2_H(P)\hat{\otimes} 1-(1-\sigma_H(P)^2)\hat{\otimes} {f'_{\Gamma}}^2\\
           =&(1-\sigma^2_H(P))(1-{f'_{\Gamma}}^2)\\
           =&(1-\sigma_H^2(P))(1+\sigma_H(\Delta_{G}))^{-1}.&
       \end{aligned} 
    \end{equation*}
\end{proof}
 Thus, by Definition \ref{def1}, $1-(\sigma_H^{tcl}(P))^2$ is compact for transversally H-elliptic operators. So we can define that:
\begin{defn}\label{defn12}
Let $P$ be a transversally H-elliptic operator. The \textit{tangent-Clifford cosymbol class} $[\sigma_H^{tcl}(P)]\in KK^G(C_0(M),  C^{\ast}(Gr(H))\hotimes_{C_0(M)} Cl_{{\Gamma}}(M))$ is defined by the $KK$-pairing 
  $$[\sigma_H^{tcl}(P)]:=(C_0(M;E)\hat{\otimes}_{C_0(M)}C^{\ast}(Gr(H)){\hotimes_{C_0(M)}} Cl_{{\Gamma}}(M),\sigma^{tcl}_H(P)).$$
 
\end{defn}

The index class is given as follows:
\begin{thm}\label{thm13}
Let $P$ be a transversally H-elliptic operator. Then $(L^2(E),P)$ defines an element in the group $K^{0}(C^{\ast}(G, C_0(M))$. It is denoted $[P]$ and called the index class of $P$.
\end{thm}
\begin{proof}
Let $\{v_k\}$ be a basis for the Lie algebra $\g$ and $v_k^{\ast}$ the dual basis for $\g$. Define the operator 
\begin{equation*}
\begin{aligned}
     d_G:C_c^{\infty}(E)&\to C_c^{\infty}(E\otimes \g_M^{\ast})\\
    f&\mapsto \sum_k \frac{\partial f}{\partial v_k}\otimes v_k^{\ast},
\end{aligned}
\end{equation*}
where $\frac{\partial}{\partial v_k}$ indicates the derivative along the one-parameter subgroup of $G$ corresponding to the vector $v_k$.

  Let $K$ be the operator with the cosymbol $(1-\sigma_H^2(P))(1+\sigma_H(\Delta_{G}))^{-1}$ that is compact at each point $x\in M$.
   
    Using Remark \ref{rmk8}, we see that the operator $K$ itself is of negative order. To prove that $a(1-P^2)e$ is compact for all $a\in C_c^{\infty}(M)$ and $e\in C_c^{\infty}(G)$, let $F=d_G^{\ast}K d_G$. As $\sigma_H(d_G^{\ast}d_G)=\sigma_H(\Delta_G)$, we see that $\sigma_H(1-P^2)-\sigma_H(F)$ is of negative order, and thus $1-P^2$ is equal to $F$ modulo operators of negative order.
    
   By Lemma 6.6 in \cite{kasparov2017elliptic}, for any arbitrary $e\in C_c(G)$, the representation of $e$ on $F$, $a F e=a\cdot d_G^{\ast} K d_G\circ e$ is equal to the product of the compact operator $a\cdot d_G^{\ast} K$ and the bounded convolution operator with $d(e)$. Therefore, $a F e$ and thus $a(1-P^2)e$ is compact for arbitrary $e\in C_c(G)$, and by natural extension, for arbitrary $a\in C_0(M)$ and $e\in C^{\ast}(G)$, establishing it as a $KK$-class in $K^{0}(C^{\ast}(G, C_0(M)))$.
\end{proof}
\begin{remark}\label{rmk5}
      In the case that $M$ is compact, the index element $[P]\in K^{0}(C^{\ast}(G, C_0(M)))$ can be simplified by the induced map of the inclusion $\C\subset C_0(M)$ and the resulting index class in $K^0(C^{\ast}(G))$ is called the distribution index of $P$.
\end{remark}

%% file: chapters/chapter05.tex
\label{chap5}

In this chapter, we give the main index theorem for transversally H-elliptic operators and prove it using a similar approach in \cite{kasparov2022k} and \cite{kasparov2024coarsepseudodifferentialcalculusindex}.

\section{Main index theorem for transversally H-elliptic operators}
\label{sec5.1}
We denote the algebra $Cl_{\tau}(M)\hotimes_{C_0(M)}Cl_{\Gamma}(M)$ by $Cl_{\tau\oplus\Gamma}(M)$. First, we define the Dirac element with Clifford action and use it to define the Clifford-Dolbeault element similar to the construction we have done in Definition \ref{defi2}.
\begin{defn}
    The \textit{Dirac element} $[d_{M,\Gamma}]\in K^0(C^{\ast}(G,Cl_{\tau\oplus\Gamma}(M)))$
    is given by the pair $(H,F_M)$, where $H=L^2(\Lambda^{\ast}(M))$, the action of $Cl_{\tau\oplus\Gamma}(M)$ on $H$ is defined on (real) covectors by 
    $$\xi_1\oplus \xi_2\mapsto \text{ext}(\xi_1)+\text{int}(\xi_1)+i (\text{ext}(\xi_2)-\text{int}(\xi_2)).$$
Here $\xi_1$ is a section of  $\tau=TM$ and $\xi_2$ is a section of $\Gamma$. The representation
of $C^{\ast}(G, Cl_{\tau\oplus\Gamma}(M))$ on $H$ is induced by the covariant representation defined by the action of $G$ and the above action of $Cl_{\tau\oplus\Gamma}(M)$. We take $D_M=d_M+d^{\ast}_M$ and let $F_M:=D_M(1+D_M^2)^{-1/2}$.
\end{defn}
 
\begin{defn}
    The \textit{Clifford-Dolbeault element} $[{\mathcal{D}_{M,\Gamma}^{cl'}}]\in K^0_G(C^{\ast}(G,C^{\ast}(Gr(H))\hotimes_{C_0(M)}$ $ Cl_{\Gamma}(M)))$ is the $KK$-product of the element $j^G([\eta_{C^{\ast}(Gr(H))}]\otimes_{C_0(M)}1_{Cl_{\Gamma}(M)})\in  \mathcal{R}KK(M;$ $ C^{\ast}(Gr(H))\hotimes_{C_0(M)}Cl_{\Gamma}(M), Cl_{\tau\oplus \Gamma}(M))$ and the Dirac element $[d_{M,\Gamma}]\in K_G^0(C^{\ast}(G,Cl_{\tau\oplus\Gamma}(M)))$.
\end{defn}
\begin{remark}
    This Clifford-Dolbeault element is different from the one we defined in Definition \ref{def7}. In fact, similar to the argument in Remark \ref{rmk7}, $[\mathcal{D}_{M,\Gamma}^{cl}]$ is actually the $KK$-product of the fiberwise Dirac and the Dirac element with respect to the Clifford action (see Definition 8.17 in \cite{kasparov2017elliptic}), and the new Clifford-Dolbeault element $[{\mathcal{D}_{M,\Gamma}^{cl'}}]$ is defined replacing the fiberwise Dirac by the global dual Dirac element $\eta$.
\end{remark}

For a transversally H-elliptic operator $P$, by Definition \ref{defn12} and Theorem \ref{thm13}, we have the tangent-Clifford symbol class
$$[\sigma^{tcl}_H(P)]\in KK^G(C_0(M), C^{\ast}(Gr(H))\hotimes_{C_0(M)} Cl_{{\Gamma}}(M)),$$ and the index class 
$$[P]\in K^0(C^{\ast}(G,C_0(M))).$$

The index theorem for transversally H-elliptic operators is presented as follows:
\begin{thm}\label{thm3}
Let $M$ be a complete filtered Riemannian manifold and $G$ be a locally compact Lie group properly and isometrically acting on $M$ and preserving the filtered structure. Let $P$ be a $G$-invariant transversally H-elliptic operator on $M$ of order $0$, then we have
    $$[P]=j^G([\sigma^{tcl}_H(P)])\otimes_{C^{\ast}(G,C^{\ast}(Gr(H))\hotimes_{C_0(M)}Cl_{\Gamma}(M))}  [\mathcal{D}_{M,{\Gamma}}^{cl'}],$$
    where the $KK$-product on the right-hand side is considered over
    \begin{equation*}
        \begin{aligned}
            KK(C^{\ast}(G, C_0(M)),C^{\ast}(G,C^{\ast}(Gr(H))\hotimes_{C_0(M)} Cl_{{\Gamma}}(M))\\
            \otimes_{C^{\ast}(G,C^{\ast}(Gr(H))\hotimes_{C_0(M)}Cl_{\Gamma}(M))} KK(C^{\ast}(G,C^{\ast}(Gr(H))\hotimes_{C_0(M)}Cl_{\Gamma}(M)),\C).
        \end{aligned}
    \end{equation*}
  
\end{thm}
\section{Some preliminaries for the proof}
\label{sec5.2}
\subsection{Local homeomorphisms} 
\label{subsec5.2.1}
Remember that there is an exponential map $\exp_x$ which maps a small open ball $V_x \subset \tau_x=T_x M$ with center $0 \in \tau_x$ diffeomorphically onto a
small neighborhood $U_x$ of the point $x\in M$. The tangent map $(\exp_x)_{\ast}$ is close to an isometry when $V_x$ is sufficiently small, and the smaller $V_x$ is, the closer to an isometry this map is.
For any $x \in M$, there is also a similar exponential map $\text{Exp}_x: gr(H)_x\to Gr(H)_x$ which is a diffeomorphism of a small neighborhood of $0 \in gr(H)_x$ to a small neighborhood of $0\in Gr(H)_x$.
By the identification of $gr(H)_x$ with $\tau_x$ by the isomorphism, we can choose the neighborhood $V_x$ small enough so that $\text{Exp}_x$ maps this neighborhood $V_x\subset gr(H)_x$ to some neighborhood $\mathcal{V}_x\subset Gr(H)_x$. We can identify all three $L^2$-spaces: $L^2(U_x)$, $L^2(V_x)$ and $L^2(\mathcal{V}_x)$ by the isometric isomorphisms.

\subsection{The construction from $\sigma_H(P)$ to $P$}
\label{subsec5.2.2}
Let $\sigma_H$ be a cosymbol in $\mathcal{L}(E)\otimes \mathcal{M}(C^{\ast}(Gr(H)))$. Choose a continuous family of open balls $\{U_x, x \in M\}$ which satisfy the following properties: the vector bundle $E$
is trivial over each $U_x$ (and this trivialization depends continuously on $x$), and
the maps $\exp^{-1}_x: U_x \to V_x$ and $\text{Exp}_x: V_x \to \mathcal{V}_x$ are well defined diffeomorphisms.\\
Pick a function $\nu\in C_0([0,1))$ such that $0\leq \nu\leq 1$, $\nu(t) = 1$ for $t \leq 1/2$, and $\nu(t) = 0$ for $t \geq 2/3$. Assume that $V_x \subset \tau_x$ is a Euclidean ball of radius $r_x$, define the function $\nu_x(v)$ on $V_x$ by $\nu_x(v)=\nu(\lVert x\rVert /r_x)$.\\
Using the isomorphisms $L^2(U_x)\cong L^2(V_x)\cong L^2(\mathcal{V}_x)$, we can transplant $\nu_x$ into $L^2(\mathcal{V}_x)$, then transplant the operator $\nu_x\sigma_{H,x}\nu_x$ into $\mathcal{L}(E_x\otimes L^2(U_x))$. This operator will be called $P(x)$.\\
 If the family $\{P(x)\}\subset \mathcal{L}(L^2(E))$ is norm-continuous, by the operator integration discussion in Section 2 in \cite{kasparov2022k}, we have an operator $P=\int_M P(x) d\phi$, where $\phi$ is the natural action of $C_0(M)$ on $L^2(E)$ by multiplication. By Theorem 3.12 in \cite{kasparov2022k}, this constructed $P$ is well-defined up to negative order terms.

    \section{Proof of the main theorem}
    \label{sec5.3}

    \begin{proof}(\textit{of Theorem \ref{thm3}})
        Assume that $P$ is a $G$-invariant transversally H-elliptic differential operator of order $0$.
        The proof can be separated into the following steps:
        \begin{enumerate}
        \item[(i)]   

    We first verify each element of the $KK$-product on the right-hand side.
    \begin{itemize}
        \item The tangent-Clifford cosymbol element is $[\sigma^{tcl}_H(P)]\in KK^G(C_0(M), C_0(M;E)\hotimes_{C_0(M)}$ $C^{\ast}(Gr(H))\hotimes_{C_0(M)} Cl_{\Gamma}(M))$. The Hilbert module is the family $\{E_x\hat{\otimes}C^{\ast}(Gr(H)_x)$ $\hotimes Cl_{\Gamma_x}(M)\}_{x\in M}$, and the operator is the family of operators $\{\sigma^{tcl}_H(P,x)\}_{x\in M}$ we defined in Definition \ref{defi1}.
        \item $[\eta_{C^{\ast}(Gr(H))}]\otimes_{C_0(M)} 1_{Cl_{\Gamma}(M)}\in \mathcal{R}KK(M; C^{\ast}(Gr(H))\hotimes_{C_0(M)}Cl_{\Gamma}(M), Cl_{\tau\oplus \Gamma}(M))$. The Hilbert module is the family $\{L^2(Gr(H)_x)\hat{\otimes}Cl_{\tau_x\oplus \Gamma_x}(M)\}_{x\in M}$, and for each $x$ the operator is the covector fields $\eta_x(y)$ over $Gr(H)_x$.
        \item $[d_{M,\Gamma}]\in K^{0}(C^{\ast}(G, Cl_{\tau\oplus\Gamma }(M)))$. The Hilbert space is $H=L^2(\Lambda^{\ast}(M))$ with $Cl_{\tau\oplus\Gamma}(M)$ action by Clifford multiplication operator on covectors by $$\xi_1\oplus \xi_2\mapsto \text{ext} (\xi_1)+\text{int}(\xi_1)+i (\text{ext}(\xi_2)-\text{int}(\xi_2)),$$ and the operator is $F_M=D_M/(1+D_M^2)^{1/2}$, where $D_M=d_M+d_M^{\ast}$ and $d_M$ is the operator of exterior derivation on $H$.
    \end{itemize}
The $KK$-product on the right-hand side can be represented on the Hilbert space $\mathcal{H}=\{E_x\hat{\otimes} L^2(Gr(H)_x)\hotimes Cl_{\tau_x\oplus \Gamma_x}(M)\}_{x\in M}\hat{\otimes}_{Cl_{\tau\oplus\Gamma}(M)} L^2(\Lambda^{\ast}(M))$ with a $C^{\ast}(G, C_0(M))$ action. The tangent-Clifford cosymbol $\sigma^{tcl}_H(P,x)$ acts on $\{E_x\hat{\otimes}$ $ L^2(Gr(H)_x)\hotimes Cl_{\Gamma_x}(M)\}_{x\in M}$, the covector field $\eta_x$ acts on $\{L^2(Gr(H)_x)\hotimes Cl_{\tau_x\oplus \Gamma_x}(M)\}_{x\in M}$  $\hat{\otimes}_{Cl_{\tau\oplus\Gamma}(M)} L^2(\Lambda^{\ast}(M))$, and $F_M$ acts on $L^2(\Lambda^{\ast}(M))$.

The triple $KK$-product can be represented as the operator on the Hilbert space $\mathcal{H}=\{E_x\hat{\otimes} L^2(Gr(H)_x)\hotimes Cl_{\tau_x\oplus \Gamma_{x}}(M)\} \hat{\otimes}_{Cl_{\tau\oplus\Gamma}(M)} L^2(\Lambda^{\ast}(M))$:
$$Q=1\hotimes \eta\hotimes 1 + (1-\eta^2)^{1/2} (N_1 \sigma_H^{tcl}(P)\hotimes 1\hotimes 1 +N_2 1\hotimes 1\hotimes F_M),$$
where $\eta=\{\eta_x\}_{x\in M}$ and $\sigma^{tcl}_H(P)=\{\sigma^{tcl}_H(P,x)\}_{x\in M}$, with appropriate operators $N_1$ and $N_2$ such that $N_1^2+N_2^2=1$.


\item[(ii)] 
We can choose a small radical neighborhood $\mathcal{V}_x$ of $x$ with the radii $r_x$ varying smoothly for $x\in M$ as in Subsection \ref{subsec5.2.1} and \ref{subsec5.2.2}. Assume that $\lVert \eta_x\rVert=1$  outside of $\mathcal{V}_x$ (or a smaller neighborhood) by renormalization. 
By this renormalization of $\eta_x$, the product $Q$ only depends on the value of $\eta_x$ and $\sigma^{tcl}_H(P,x)$ on $\mathcal{V}_x\subset Gr(H)_x$. 
 By the local homeomorphism between $\mathcal{V}_x\subset Gr(H)_x$ and $U_x\subset M$, $\sigma^{tcl}_H(P,x)$ will be transformed into the operator $T^x$ with constant cosymbol $\sigma^{tcl}_H(P,x)$ on $U_x$, and $\eta_x(y)$, which is a covector field over $\mathcal{V}_x\subset Gr(H)_x$, will be transformed into a covector field on $U_x$. We still denote it by $\eta_x(y)$ for $y\in U_x\subset M$. As a result, we can transform all these operators to a small neighborhood $U$ of the diagonal $M\times M$ with $U_x=U\cap (M\times \{x\})$.

Then after the local homeomorphism, the operator $\sigma_H(P,x)$ and $\eta_x$ are defined on $E_x\hotimes L^2(Gr(H)_x)$ will be transformed into operators on $E_x\hotimes L^2(U_x)\cong L^2(E|_{U_x})$ by the trivialization condition of $E$ on $U_x$. So the whole operator $Q$ is transformed into 
$$Q'=1\hotimes \eta\hotimes 1 + (1-\eta^2)^{1/2} (N_1 \{T^x\}_{x\in M}\hotimes 1\hotimes 1 +N_2 1\hotimes 1\hotimes F_M),$$
acting on the Hilbert space $\mathcal{H}'=\{L^2(E|_{U_x}) \hat{\otimes} Cl_{\tau_x\oplus \Gamma_x}(M)\}_{x\in M} \hotimes_{Cl_{\tau\oplus\Gamma}(M)} L^2(\Lambda^{\ast}(M))=\{L^2(E) $ $\hat{\otimes} L^2(\Lambda^{\ast}(U_x))\}_{x\in M}$. Consequently, $Q'$ can be considered as a family of operators $Q'_x$ on $\mathcal{H}'_x=L^2(E)\hat{\otimes} L^2(\Lambda^{\ast}(U_x))$, parametrized by $x\in M$:
$$Q'_x=1\hotimes \eta_x\hotimes 1 + (1-\eta_x^2)^{1/2} (N_1 T^x\hotimes 1\hotimes 1 +N_2 1\hotimes 1\hotimes F_M).$$
The operator $T^x$ can be represented as 
 $$T^x=P(x)\hotimes 1+(1-P(x)^2)^{1/2}\hotimes \mathfrak{F}_G(x),$$
 in which $P(x)$ is the operator we constructed in subsection \ref{subsec5.2.2} from the cosymbol $\sigma_H(P,x)$ at $x$ and $\mathfrak{F}_G(x)$ is obtained similarly from the cosymbol $\sigma_H(\mathfrak{F}_G, x)$. We can write $P(x)$ (and also $\mathfrak{F}_G(x)$) as the integral representation $\int_{y\in M} P(x) d\phi$ and do the simple homotopy, from the operator $P(x)$ with constant cosymbol $\sigma_H(P,x)$ on $U_x$, to the operator $P$ on $U_x$ (up to negative order terms) by 
$$t\mapsto \int_{y\in M} P(ty+(1-t)x) d\phi,\ 0\leq t\leq 1$$
for all $y\in U_x$. By the arguments in Section 5.3 in \cite{kasparov2022k}, the function from orbit $\mathcal{O}$ to the operator $\mathfrak{F}_{\mathcal{O}}=\D_{\mathcal{O}}/(1+\D^2_{\mathcal{O}})^{1/2}$ defined in Section \ref{sec4.2} is continuous in the $\ast$-strong topology, and $1-\mathfrak{F}_{\mathcal{O}}^2$ is continuous in norm, so the above homotopy provides a standard homotopy between $KK$-cycles, from $T^x$ to $T=P\hotimes 1+(1-P^2)^{1/2}\hotimes \mathfrak{F}_G$ on $U_x$. 
During the homotopy, we shall note that these operators are considered on the consistent Hilbert space $L^2(E)$.

It suffices to show that $a(1-T^2)$ is a compact operator for arbitrary $a\in C_0(M)$, so that such an operator $T$ actually represents a $KK$-cycle. As in the proof of Theorem 10.3 of \cite{kasparov2022k}, the commutator term $[P\hat{\otimes}1,1\hat{\otimes}\mathfrak{F}_G]$ is estimated on two sides by the operator $(1+\D_{\mathcal{O}}^2)^{-1/2}$, i.e., for some positive constant number $c$, we have $$-c (1+\D_{\mathcal{O}}^2)^{-1/2}\leq [P\hat{\otimes}1,1\hat{\otimes}\mathfrak{F}_G]\leq c (1+\D_{\mathcal{O}}^2)^{-1/2}.$$ By calculation, $1-T^2=(1-P^2)(1-\mathfrak{F}_G^2)-(1-P^2)^{1/2}[P\hat{\otimes}1,1\hat{\otimes}\mathfrak{F}_G]$. Note that by definition, the transversally H-elliptic operator $P$ satisfies that $(1-\sigma_H(P)^2)(1+\sigma_H(\Delta_G))^{-1}$ is compact and thus has negative H-order. So by Remark \ref{rmk8}, $(1-P^2)(1+\D_{\mathcal{O}}^2)^{-1}$ also has negative order. Thus $a(1-T^2)$ is compact for arbitrary $a\in C_0(M)$ acting on the first copy $M$ of $U\subset M\times M$ where $T$ lies.

 The $C^{\ast}(G, C_0(M))$ action in the present $KK$-product is on the second copy of $M$ in $U\subset M\times M$, since the first $M$ is obtained from the transformation of $Gr(H)_x$ part and the second $M$ is the origin manifold $M$. We need a `rotation' homotopy between the $C_0(M)$ action on these two different variables $M$. For any $(y,x)\in U\subset M\times M$, let $p_t(x,y),\ 0\leq t\leq 1$ be the geodesic segment joining $x$ and $y$. Then for arbitrary $a\in C_0(M)$, $$\phi_t: a\mapsto p_t^{\ast}(a)\in C_b(M\times M)$$ gives a homotopy between the homomorphism of multiplication on $C_b(U)$ by $C_0(M)$ on the first and second $M$ variable.

 Simultaneously, we need to take the homotopy of the operator $t\mapsto (1-t){\mathfrak{F}_{G}},\ 0\leq t\leq 1$ from $\mathfrak{F}_{G}$ to $0$ with the same $t$ above. 
From the definition, $(1-\sigma_H(P)^2)(1+\sigma_H(\Delta_G))^{-1}$ is compact and thus has negative H-order. So by Remark \ref{rmk8}, $(1-P^2)(1+\Delta_G)^{-1}$ becomes compact after being multiplied by any function from $C_0(M)$ on the first $M$. For each fixed $y\in M$, the function $p_t^{\ast}(a)$ will be a $C_0$ function in $x$ for $0< t\leq 1$. So $t(1-P^2)^{1/2}\hotimes \mathfrak{F}_{G}$ will be killed by the multiplication by  $p_t^{\ast}(a)\cdot (1+\Delta_G)^{-1/2}\in C^{\ast}(G, C_0(M))$ during the homotopy. For the same reason for fixed $x\in M$, the commutator of $t\mathfrak{F}_{G}$ and $F_M$ is also compact with multiplication by $p_t^{\ast}(a)$ during the homotopy. By this homotopy, the orbital Dirac operator $\mathfrak{F}_G$ will be killed, and the multiplication of $C_0(M)$ will act on the first copy of $M$.


Then, by the homotopies constructed above, $Q'$ is transformed into a family of operators $Q''=\{Q_x''\}_{x\in M}$:
$$Q''_x=1\hotimes \Theta_x\hotimes 1 + (1-\Theta_x^2)^{1/2} (N_1 P\hotimes 1\hotimes 1 +N_2 1\hotimes 1\hotimes F_M),$$
acting on the Hilbert space $\mathcal{H}'_x=L^2(E) \hat{\otimes} L^2(\Lambda^{\ast}(U_x))$ for each $x\in M$.
\item[(iii)]
 We calculate the $KK$-product $[\Theta_M]\otimes_{ Cl_{\tau}(M)} [d_M]\otimes_{C_0(M)}[P]$, in which $[P]$ takes the product with $[\Theta_M]\in KK(C_0(M), C_0(M)\hotimes Cl_{\tau}(M))$ over $C_0(M)$ in the second variable and $[d_M]$ takes the product with $[\Theta_M]$ over $Cl_{\tau}(M)$.  First, we verify each element of this product.
 \begin{itemize}
     \item The local dual Dirac element $[\Theta_M]\in KK(C_0(M), C_0(M)\hotimes Cl_{\tau}(M))$ is defined by the field of pairs $\{(Cl_{\tau}(U_x), \Theta_x)_{x\in M}\}$.
     \item The index class $[P]\in K^0(C^{\ast}(G,C_0(M)))$ is defined by the pair $(L^2(E),P)$.
     \item The Dirac element $[d_M]\in K^0(Cl_{\tau}(M))$ is defined by the pair $(L^2(\Lambda^{\ast}(M)), F_M=D_M/(1+D_M^2)^{1/2})$, where $D_M=d_M+d_M^{\ast}$.
 \end{itemize}
 
 So the triple $KK$-product $[\Theta_M]\otimes_{Cl_{\tau}(M)} [d_M]\otimes_{C_0(M)}[P]$ can be written as the pair $$(C^{\ast}(G, C_0(M))\to \mathcal{H}',\mathcal{Q}),$$ where $\mathcal{H}'$ is the Hilbert space $\{L^2(E)\hotimes L^2(\Lambda^{\ast}(U_x))\}_{x\in M}$, which can be considered as a family of Hilbert spaces $\mathcal{H}'_x=L^2(E)\hotimes L^2(\Lambda^{\ast}(U_x)$, and the operator $\mathcal{Q}=\{\mathcal{Q}_x\}_{x\in M}$ is a family of operators $\mathcal{Q}_x$ on $\mathcal{H}'_x$ parametrized by $x\in M$:
 $$\mathcal{Q}_x=N_1 P\hotimes 1\hotimes 1+ N_2 [ 1\hotimes \Theta_x \hotimes 1 + (1-\Theta_x^2)^{1/2}  1\hotimes 1\hotimes F_M]$$
with $N_1^2+N_2^2=1$.

By the construction of $KK$-product (Theorem 2.11 in \cite{kasparov1988equivariant}), the operators $N_1$ and $N_2$ are assigned to commute with $C_0(M)$, $P\hotimes 1\hotimes 1$, $1\hotimes \Theta_x \hotimes 1$ and $1\hotimes 1\hotimes F_M$ module $\mathcal{K}(L^2(M\times M))$, and they satisfy $N_1(\mathcal{K}(L^2(M))\otimes 1)\subset \mathcal{K}(L^2(M\times M))$ so $N_1$ vanishes $[P\hotimes 1\hotimes 1, 1\hotimes \Theta_x\hotimes 1]$ and $1-\Theta_x^2$,
$N_2(1\otimes \mathcal{K}(L^2(M))\subset \mathcal{K}(L^2(M\times M))$ and $N_2$ vanishes $[1\hotimes \Theta_x\hotimes 1, 1\hotimes1\hotimes F_M]$ and $[P\hotimes 1\hotimes 1, 1\hotimes1\hotimes F_M]$.
So we see that 
$$a[Q'', \mathcal{Q}]a^{\ast}\geq 0\ \text{mod}\ \mathcal{K}(\mathcal{H}')$$
for all $a\in C^{\ast}(G, C_0(M))$. Then by Lemma 11 in \cite{skandalis1984some}, the $KK$-cycles $(\mathcal{H}', Q'')$ and $(\mathcal{H}', \mathcal{Q})$ are operator homotopic. As a result, we obtain the following equation:
 \begin{equation*}
     \begin{aligned}
         &[\Theta_M]\otimes_{Cl_{\tau}(M)} [d_M]\otimes_{C_0(M)}[P]\\
         =&j^G([\sigma^{tcl}_H(P)])\otimes_{C^{\ast}(G,C^{\ast}(Gr(H))\hotimes_{C_0(M)}Cl_{\Gamma}(M))}  [\mathcal{D}_{M,{\Gamma}}^{cl'}]\in K^0(C^{\ast}(G,C_0(M))).
     \end{aligned}
 \end{equation*}

\item[(iv)]

 
By Theorem 4.8 in \cite{kasparov1988equivariant}, 
$$[\Theta_M]\otimes_{Cl_{\tau}(M)}[d_M]=[1]\in KK(C_0(M),  C_0(M)),$$
so its product with index class $[P]\in K^0(C^{\ast}(G, C_0(M)))$ over $C_0(M)$ will be $[P]$ itself. 
So we obtain the index class $[P]\in K^0(C^{\ast}(G, C_0(M)))$ on the left-hand side as we desired.
        \end{enumerate}
         That finishes the proof.         
         
    \end{proof}
\begin{remark}
     It is also considerable to establish the index theory for transversally H-elliptic operators via an adiabatic groupoid approach, similar to the construction in Section \ref{sec3.1}. For each representation of $G$, it entails contemplating distinct $C^{\ast}$-algebras for each fiber of the adiabatic groupoid, so that it can represent a commutative diagram and we obtain an index in $\Z$. By such construction, the $C^{\ast}$-algebra diagram induces an element in $\text{Hom}(R(G),\Z)\cong \widehat{R(G)}\cong K^0(C^{\ast}(G))$, and it should be consistent with the distributional index class of $P$ we discussed in Remark \ref{rmk5}. A different proof of the index theorem using such a groupoid approach is left for future research.
\end{remark}

%% file: chapters/chapter06.tex
\label{chap6}
In this chapter, we focus on verifying the compactness condition outlined in Definition \ref{def1}, i.e.,
\begin{quest}\label{que1}
    Let $e$ be an fixed element in the multiplier algebra of the $C^{\ast}(Gr(H)_x)$-module $E=C^{\ast}(Gr(H)_x)$,
$e\in \mathcal{M}(C^{\ast}(Gr(H)_x))$,
 how to show whether or not $e$ lies in the compact operator space $\mathcal{K}(E)=E$.
\end{quest}

\section{Fourier transform of nilpotent Lie group $C^{\ast}$-algebras}
\label{sec6.1}

\subsection*{Fourier transform from ordinary cosymbol to symbol}
Recall that for a differential operator $P$, let $F$ be the ordinary cosymbol of $P$ we discussed in Remark \ref{rmk4}. We can take the Fourier transformation of $F\in \mathcal{M}(C^{\ast} (TM))$ to the symbol $f=\widehat{F}$ in $\mathcal{M}(C_0(T^{\ast} M))=C_b(T^{\ast} M)$. Let us describe an idea to show the compactness of $F$ roughly as follows: We show that, when $\xi$ converges to infinity, $f(x,\xi)$ converges to $0$, and thus, it lies in $C_0(T^{\ast} M)$ and then the original cosymbol operator $F$ is a compact multiplier of $C^{\ast}(TM)$. See the commutative diagram below:
$$\begin{tikzcd}
&\mathcal{M}(C^{\ast}(T^{\ast} M)) \arrow[r, "\mathcal{F}"] & C_b(T^{\ast}M) \\
     &C^{\ast}(T^{\ast}M)\arrow[r, "\mathcal{F}"]\arrow[u, phantom, sloped, "\subset"] & C_0(T^{\ast}M)\arrow[u, phantom, sloped, "\subset"].
\end{tikzcd}$$

In case of a differential operator $P$ on the vector bundle $E$, the ordinary symbol can similarly be understood as an element in $C_b(M; p^{\ast}E)\cong C_b(T^{\ast}M)\otimes_{C_0(M)} C_0(M; E)$, with a compact symbol in the subspace $C_0(M;p^{\ast}E)\cong C_0(T^{\ast}M)\otimes_{C_0(M)} C_0(M; E)$.

\subsection*{Main idea to verify the compactness condition}

 For an element $e\in \mathcal{M}(C^{\ast}(\G))$, we show its Fourier transform $\mathcal{F}(e)\in \mathcal{F}(\mathcal{M}(C^{\ast}(\G)))$ lies in $\mathcal{F}(C^{\ast}(\G))$ by using the $C^{\ast}$-structure of the resulting space of Fourier transform. Then, $e\in \mathcal{K}(C^{\ast}(\G))=C^{\ast}(\G)$.
A parallel diagram illustrates the relationship:
 
$$\begin{tikzcd}
&\mathcal{M}(C^{\ast}(\G)) \arrow[r, "\mathcal{F}"] & \mathcal{F}(\mathcal{M}(C^{\ast}(\G))) \\
     &C^{\ast}(\G)\arrow[r, "\mathcal{F}"]\arrow[u, phantom, sloped, "\subset"] & \mathcal{F}(C^{\ast}(\G))\arrow[u, phantom, sloped, "\subset"].
\end{tikzcd}$$

  We can assign $\G$ to be $Gr(H)_x$ for each $x$, which is nilpotent. In this case, we can verify the compactness of a multiplier on $C^{\ast}(\G)$ for a nilpotent Lie group $\G$ by using the theory regarding Fourier transform of nilpotent Lie group $C^{\ast}$-algebra, which has been treated in \cite{beltictua2017fourier}.

\subsection*{Fourier transform of nilpotent Lie group $C^{\ast}$-algebras}

First, we introduce some definitions given in \cite{beltictua2017fourier}.
\begin{defn}\label{defn2}
    Let $A$ be a $C^{\ast}$-algebra with \textit{spectrum} $\widehat{A}$, which consists of the unitary equivalence classes of irreducible representations. We choose for every $\gamma \in \widehat{A}$ a representation $\pi_{\gamma} : A \to B(\mathcal{H}_{\gamma})$ in the equivalence class $\gamma$. 
    
    Define $\mathit{l}^ {\infty}(\widehat{A})$ to be the algebra of all bounded operator fields defined over $\widehat{A}$ by $$\mathit{l}^{\infty}(\widehat{A}):=\big\{ \phi=(\phi({\gamma}))\in \mathcal{B}(\mathcal{H}_{\gamma})_{\gamma\in \widehat{A}}\ |\ \lVert \phi\rVert =\sup_{\gamma}\ \lVert \phi(\gamma)\rVert_{B(\mathcal{H}_{\gamma})}<\infty \big\}.$$
    
    Then we define for $a \in A$ its \textit{Fourier transform} $\mathcal{F}_A(a) = \widehat{a}\in \mathit{l}^ {\infty}(\widehat{A})$ by
    $$\mathcal{F}_A(a)(\gamma)=\widehat{a}(\gamma)=\pi_{\gamma}(a).$$
\end{defn}

\begin{defn}
\begin{itemize}
    \item[(i)] Let $S$ be a topological space. We say that $S$ is \textit{locally compact of step $\leq d$} if there exists a finite increasing family
  $$\emptyset \neq S_{d} \subset S_{d-1} \subset \cdots \subset S_0 = S $$ of closed subsets of $S$, such that the subsets $\Gamma_d=S_d$, $\Gamma_i=S_{i-1}\setminus S_{i}$ for $1\leq i\leq d-1$ are locally compact and Hausdorff in their relative topologies.
  \item[(ii)] Let $S$ be locally compact of step $ \leq d$, and let $\{\mathcal{H}_i\}_{{i=1,\dots,d}}$ be Hilbert spaces. For a closed subset $M \subset S$, denote by $CB(M)$ the unital $C^{\ast}$-algebra of all uniformly bounded operator fields
  $(\psi(\gamma)\in \mathcal{B}(\mathcal{H}_i))_{\gamma\in S\cap\Gamma_i,i=1,\dots, d}$, which are operator norm continuous on the
subsets $\Gamma_i \cap M$ for every $i \in \{1, \dots ,d\}$ with $\Gamma_i \cap M \neq \emptyset$. We provide the algebra $CB(M)$ with the infinity-norm 
$$\lVert \phi \rVert_M := \sup\big\{\lVert\phi(\gamma)\rVert_{\mathcal{B}(\mathcal{H}_i)}\ |\ M\cap \Gamma_i\neq 0,\ \gamma\in M\cap\Gamma_i   \big\}.$$

\end{itemize}

\end{defn}

\begin{defn}
    We say that the $C^{\ast}$-algebra $A$ \textit{has norm-controlled dual limits} if the spectrum $\widehat{A}$ of $A$ is a locally compact space of step $\leq d$, and for every $a \in A$, we have 
    \begin{itemize}
        \item[(1)] The mappings $\gamma \to \mathcal{F}(a)(\gamma )$ are norm continuous on the difference sets $\Gamma_i=S_{i-1}\setminus S_{i}$;
        \item[(2)]  For any $i = 0, \dots ,d + 1$ and any converging sequence contained in $\Gamma_i$ with limit set outside $\Gamma_i$, thus in $S_i$, there exists a properly converging subsequence $\overline{\gamma} = (\gamma_k)$, $k\in \N$, a constant $C > 0$ and for every $k \in N$ an involutive linear mapping $\sigma_{\overline{\gamma} ,k} : CB(S_i) \to B(\mathcal{H}_i)$, bounded by $C \lVert \cdot\rVert_{S_i}$, such that for arbitrary $a\in A$,
$$\lim_{k\to\infty} \lVert  F(a)(\gamma_k) - \sigma_{\overline{\gamma} ,k}(\mathcal{F}(a)|_{S_i})\rVert_{B(\mathcal{H}_i)} = 0.$$
    \end{itemize}
\end{defn}
\begin{thm}[\cite{beltictua2017fourier}]\label{thm11}
    For a connected, simply connected nilpotent Lie group $\G$, the $C^{\ast}$-algebra $C^{\ast}(\G)$ has norm-controlled dual limits.
\end{thm}

\begin{defn}\label{defn3}
     Let $S$ be a locally compact topological space of step $\leq d$. Define \textit{$B^{\ast}(S)$} to be the set of all operator fields $\phi$ defined over $S$ such that
\begin{itemize}
    \item[(a)] $\phi(\gamma) \in K(\mathcal{H}_i)$ for every $\gamma\in  \Gamma_i = S_{i-1}\setminus S_i, i = 1, \dots, d$;
     \item[(b)] The field $\phi$ is uniformly bounded, i.e., 
     $$\lVert \phi\rVert=\max_i \sup_{\gamma\in \Gamma_i} \lVert \phi(\gamma)\rVert_{B(\mathcal{H}_{i})}<\infty ;$$
      \item[(c)] The mappings $\gamma \mapsto \phi(\gamma )$ are norm continuous on the difference sets $\Gamma_i$;
       \item[(d)] For any sequence $(\gamma_k)_{k\in \N} \subset S$ going to infinity, 
       $$\lim_{k\to \infty}\lVert \phi(\gamma_h)\rVert_{\text{op}}= 0;$$
        \item[(e)]  For any $i = 0, \dots ,d + 1$ and any converging sequence contained in $\Gamma_i$ with its limit set outside $\Gamma_i$, in $S_i$, there exists a properly converging subsequence $\overline{\gamma} = (\gamma_k)$, $k\in \N$, a constant $C > 0$ and for every $k \in N$ an involutive linear mapping $\sigma_{\overline{\gamma} ,k} : CB(S_i) \to B(\mathcal{H}_i)$, which is bounded by $C \lVert \cdot\rVert_{S_i}$, such that for arbitrary $a\in A$,
$$\lim_{k\to\infty} \lVert  \phi(\gamma_k) - \sigma_{\overline{\gamma} ,k}(\phi|_{S_i})\rVert_{B(\mathcal{H}_i)} = 0.$$
\end{itemize}
\end{defn}
   
\begin{thm}[\cite{beltictua2017fourier}]\label{thm1}
    Let $A$ be a $C^{\ast}$-algebra with norm-controlled dual limits. Then, the Fourier transform of $A$ can be derived as the $C^{\ast}$-algebra $B^{\ast}(\widehat{A})$.
\end{thm}
Consequently, combining Theorem \ref{thm11} and Theorem \ref{thm1}, we have
\begin{cor}\label{cor1}
    The Fourier transform of $C^{\ast}(\G)$ is the $C^{\ast}$-algebra $B^{\ast}(\widehat{\G})$.
\end{cor}

\begin{exmp}
    The Heisenberg group $C^{\ast}$-algebra $C^{\ast}(\mathbb{H}_n)$ is isomorphic to $D_{\nu}(\widehat{\mathbb{H}_n})$ by Fourier transform, where $D_{\nu}(\widehat{\mathbb{H}_n})$ is defined to be the algebra comprising all operator fields $(F = F(\lambda))_{\lambda\in\R}$ satisfying the following conditions:
    \begin{itemize}
        \item $F(\lambda)$ denotes a compact operator on $L^2(\R^2)$;
        \item $F(0)\in C^{\ast}(\R^2)$;
        \item the mapping $\R^{\ast} \to B(L^2(\R^n))): \lambda \to F(\lambda)$ is norm continuous;
        \item  $\lim_{\lambda\to \infty} \lVert F(\lambda)\rVert_{op} = 0$;
        \item $\lim_{\lambda\to 0} \lVert F(\lambda) - \nu_{\lambda}(F(0))\rVert_{op} = 0$, where
        $$\nu_{\lambda}(h):=\int_{\R^{2n}}\widehat{h}(u) P_{\eta(\lambda,u)}du/|\lambda|^n,$$
        where $\eta$ denotes a Schwartz-function in $S(\R^n)$ with $L^2$-norm equal to $1$, $P_{\eta}$ is the projection onto the one-dimensional subspace $\C \eta$.

    \end{itemize}

    In this case,  $C^{\ast}(\mathbb{H}_n)$ of $\mathbb{H}_n$ is an extension of an ideal $J$ isomorphic to $C_0(\R^{\ast}, \mathcal{K})$ with the quotient algebra isomorphic to $C^{\ast}(\R^{2n})$.
\end{exmp}

\section{Application to the multiplier algebra}
\label{sec6.2}
By the main idea at the beginning of Section \ref{sec6.1} and Corollary \ref{cor1}, it suffices to analyze the multiplier algebra of $B^{\ast}(\widehat{\G})$. Then we can give an answer for the Question \ref{que1}.

The multiplier algebra of $C_0$ functions on a topological space $M$ with values in a certain $C^{\ast}$-algebra can be calculated as follows.
\begin{defn}
    Let $A$ be a $C^{\ast}$-algebra. Let $\mathcal{M}(A)$ be its multiplier $C^{\ast}$-algebra. The \textit{strict topology} on $\mathcal{M}(A)$ is given by the convergence condition:
    $$x_{\lambda}\to x \Longleftrightarrow \forall a\in A: (\lVert xa_{\lambda}-xa\rVert + \lVert a_{\lambda}x- ax\rVert\to 0 ).$$

\end{defn}

\begin{lem}[\cite{akemann1973multipliers}  Corollary 3.4] \label{lem4} Let $A$ be a $C^{\ast}$-algebra, then
    $${M}(C_0(M,A))=C_b(M,\mathcal{M}(A)_{\beta}),$$
    where $\mathcal{M}(A)_{\beta}$ is the multiplier algebra of $A$ with the strict topology.
\end{lem}
\begin{remark}\label{rmk6}
    By Lemma \ref{lem4}, we see that the multiplier algebra for continuous functions valued in compact operator space consists of continuous elements valued in the bounded operator space with the strict topology.
\end{remark}
  Applying Remark \ref{rmk6} to the conditions in Definition \ref{defn3}, we get
\begin{cor}\label{cor6}
    A multiplier of $\mathcal{M}(B^{\ast}(S))$ is an element of the operator field $\psi$ over $S$ that satisfies:
    \begin{itemize}
    \item[$(a')$] $\psi(\gamma) \in B(\mathcal{H}_i)$ for every $\gamma\in  \Gamma_i = S_{i-1}\setminus S_i, i = 1, \dots, d$;
     \item[$(b')$] The field $\psi$ is uniformly bounded;
      \item[$(c')$] The mappings $\gamma \mapsto \psi(\gamma )$ are strictly continuous on the difference sets $\Gamma_i$;
        \item[$(d')$]  For any $i = 0, \dots ,d + 1$ and any converging sequence contained in $\Gamma_i$ with limit set outside $\Gamma_i$, thus in $S_i$, there exists a properly converging subsequence $\overline{\gamma} = (\gamma_k)$, $k\in \N$, a constant $C > 0$ and for every $k \in N$ an involutive linear mapping $\sigma_{\overline{\gamma} ,k} : CB(S_i) \to B(\mathcal{H}_i)$, which is bounded by $C \lVert \cdot\rVert_{S_i}$, such that for arbitrary $a\in A$,
$$\lim_{k\to\infty}  \psi(\gamma_k) - \sigma_{\overline{\gamma} ,k}(\psi|_{S_i}) = 0 $$ in the strict topology of $B(\mathcal{H}_i)$.
    \end{itemize}
    
\end{cor}
Take a multiplier $F\in \mathcal{M}(C^{\ast}(\G))$, $F:C^{\ast}(\G)\to C^{\ast}(\G)$. For each representation $\pi\in \widehat{\G}$, $\pi$ extends to a representation of $C^{\ast}(\G)$; it induces a map $\pi(F):=F\otimes_{\pi} 1$ from $C^{\ast}(\G)\otimes_{\pi} \mathcal{H}_{\pi}\cong \mathcal{H_\pi}$ to itself. Then by Definition \ref{defn2}, we find that $\{\pi(F)\}_{\pi\in \widehat{\G}}$ is the Fourier transform of an element $F$ in the multiplier algebra $\mathcal{M}(C^{\ast}(\G))$.

\begin{cor}\label{cor4}
          A multiplier $Q\in \mathcal{M}(E_x\otimes C^{\ast}(Gr(H)_x))$ is compact, i.e., $Q\in \text{End}(E_x)\otimes C^{\ast}(Gr(H))_x$ if and only if the operator field
    $$\phi=\{\pi_{\gamma}(Q)\}_{\gamma\in \widehat{Gr(H)}_x} \subseteq \text{End}(E_x)\otimes \mathit{l}^{\infty}(\widehat{Gr(H)}_x)$$
    lies in the subspace $\text{End}(E_x)\otimes B^{\ast}(\widehat{Gr(H)}_x)$,
    i.e., each term $\phi_{ij}$ of the $\mathit{l}^{\infty}(\widehat{Gr(H)}_x)$-valued matrix $\phi$ in $\text{End}(E_x)$ satisfies the conditions $(a)-(e)$ in Definition \ref{defn3}. 
\end{cor}
\begin{proof}
      Take $S=C^{\ast}(\G)$ and where $G=Gr(H)_x$. The multiplier $Q$ lies in $\text{End}(E_x)\otimes C^{\ast}(Gr(H))_x$ if and only if its Fourier transform $\phi=\{\pi_{\gamma}(Q)\}_{\gamma\in \widehat{Gr(H)}_x}$ lies in $\text{End}(E_x)\otimes \mathcal{F}(C^{\ast}(Gr(H)_x))$ which is equal to $\text{End}(E_x)\otimes B^{\ast}(\widehat{Gr(H)}_x)$ by Corollary \ref{cor1}. By Definition \ref{defn3}, it is equivalent to the conditions $(a)-(e)$ are satisfied for each term $\phi_{ij}$ of the $\mathit{l}^{\infty}(\widehat{Gr(H)}_x)$-valued matrix $\phi$ in $\text{End}(E_x)$.
\end{proof}

\begin{exmp}
    For instance, consider the $\R^2$ action on $\mathbb{H}_3\cong \R^5$ on the first two $x_1, y_1$ directions. In addition, consider the operators of the form $P_{\gamma}=-X_2^2-Y_2^2+i\gamma T$, which is induced from an operator on $\mathbb{H}_2\cong \R^3$ to $\mathbb{H}_3\cong \R^5$. $P$ is normalized to be self-adjoint, and $\phi$ is calculated as the diagonal matrix with 
    $$\phi_{11}=(1+\sigma_H(P^{\ast},x)\sigma_H(P,x))^{-1}(1+d/d x_1^2+d/d y_1^2)^{-1},$$
    $$\phi_{22}=(1+\sigma_H(P,x)\sigma_H(P^{\ast},x))^{-1}(1+d/d x_1^2+d/dy_1^2)^{-1}.
    $$
    Suppose $P_{\gamma}$ is an H-elliptic operator on $\R^3$, i.e., $\gamma$ is not an odd integer, then $(b)(c)(e)$ will be evident by the condition that each representation of $\mathbb{H}_3$ can be pulled back to the representation of $\R^2$ and $\mathbb{H}_2\cong \R^3$---both of which satisfies these conditions. We verify the conditions $(a)$ and $(d)$. Based on Example 2.2, the representations of $\mathbb{H}_3$ contain two classes. For $(a)$, using the coadjoint orbit, we can derive:
    \begin{itemize}
        \item   for a Schr\"{o}dinger representation $\pi_{\lambda}$ of $\mathbb{H}_3$ which can be represented by $\lambda T$, the restriction to $\mathbb{H}_2$ will be Schr\"{o}dinger representations $\pi'_{\lambda}$ of $\mathbb{H}_2$. Based on the Fourier transform of the sub-Lie group $\mathbb{H}_2\subset \mathbb{H}_3$, $\pi'_{\lambda}(1+\sigma_H(P^{\ast},x)\sigma_H(P,x))^{-1})$ and $\pi'_{\lambda}(1+\sigma_H(P,x)\sigma_H(P^{\ast},x))^{-1})$ will be compact on $L^2(\R^1)$ regarding as a representation of an element of the liminary $C^{\ast}$-algebra $C^{\ast}(\mathbb{H}_2)$, generated by $x_2,y_2,t$ coefficient. $\pi_{\lambda}((1+d/d x_1^2+d/dy_1^2)^{-1})$ will be compact on the first variable $L^2(\R^1)$, as it can be regarded as a representation of an element of the liminary $C^{\ast}$-algebra $C^{\ast}(\mathbb{H}_2)$, generated by $x_1,y_1,t$ coefficient. Upon combining these two parts, we observe that $\pi_{\lambda}(\phi)$ is compact for all Schr\"{o}dinger representations $\pi_{\lambda}$. 
    \item  for one-dimensional representations, if $\pi=\pi_{x_1,x_2,y_1,y_2}$ is presumed to be nonzero for the $x_1,x_2$ components, then $\pi((1+d/d x_1^2+d/d y_1^2)^{-1})$ will be compact on $\R$. Otherwise, if it is nonzero for the $y_1,y_2$ part, then $\pi(1+\sigma_H(P^{\ast},x)\sigma_H(P,x))^{-1})$ and $\pi(1+\sigma_H(P,x)\sigma_H(P^{\ast},x))^{-1})$ will be compact on $\R$. Consequently, the product will be a compact operator in all cases.
    \end{itemize}
       $(d)$ can also be shown by the same argument. As a result, $P_{\gamma}$ is a transversally H-elliptic operator on $\mathbb{H}_3\cong \R^5$ if $\gamma$ doesn't take values in odd integers.
\end{exmp}
\begin{remark}
\begin{itemize}
    \item[(i)]  In general cases, the direct decomposition of representations (aligned with coadjoint orbits) into transverse and leaf-wise components, similar to the partitioning of cotangent vectors, presents significant challenges. This complexity leads to difficulties in addressing transversal H-ellipticity. Nevertheless, the application of Fourier transform methodologies facilitates the transformation of operators to actions on Hilbert spaces, thereby simplifying the task to prove the compactness of these operators within each Hilbert space. This approach allows for the partitioning of representing Hilbert spaces into distinct segments, facilitating a separate examination of the transverse and leaf-wise directional elements.

\item[(ii)] If we focus on the symbol terms of differential operators and the inverses of Rockland terms within the cosymbol space, it is possible that the continuity conditions $(c)$ and $(e)$ in Definition \ref{defn3} are redundant. However, the absence of a definitive proof leaves this aspect open to further research. 
\end{itemize}
     
\end{remark}

%% file: chapters/appendix.tex
\section{Other ideas for the definition of transversal H-ellipticity}

Alternative definitions of transversal H-ellipticity can also be considered. All the definitions are motivated by some characteristics of H-ellipticity or transversal ellipticity. 
However, these definitions are not sufficient to calculate in $KK$-theory and it's hard to show the equivalent or inclusion relations between any two of these different conditions, because the $C^{\ast}$-algebra $C^{\ast}(Gr(H))$ is non-commutative and we cannot regularly decompose $gr(H)_x$ elements into transversal and leaf-wise directions. These definitions only provide other possible directions to analyze.

We will briefly show the possible directions for analyzing operators in these different definitions.

\subsection{The main definition}
Recall that our main definition is 
\begin{defx}\label{defA}
     Let $M$ be a complete filtered Riemannian manifold and $G$ be a Lie group acting on $M$ isometrically and preserving the filtered structure. We call a pseudodifferential operator of $0$-order $P$ is transversally H-elliptic if at each point $x\in M$, the multiplier term 
    $$(1-\sigma_H^2(P,x))(1+\sigma_H(\Delta_{G},x))^{-1}\in \mathcal{M}(E_x\otimes C^{\ast}(Gr(H)_x))$$
    is compact as operators on $C^{\ast}(Gr(H)_x)$-modules.
\end{defx}
The main definition (\textit{Definition} \ref{defA}) is motivated by the Rockland condition characteristic of H-elliptic operators and the tangent Clifford symbol for transversally elliptic operators. It facilitates the definition of index and symbol classes in $KK$-theory for operator $P$, enabling the application of Kasparov's method for index theorem analysis.

\subsection{Definition from Atiyah's approach}
The second possible definition is stated as follows: 
\begin{defx}\label{defB}
    $P$ is called transversally H-elliptic if for each $x\in M$, for all representation $\pi$ in $Gr(H)_x$, both $\pi(\sigma_H((P,\Delta_G),x))$ and $\pi(\sigma_H((P^{\ast},\Delta_G),x))$ are injective.
\end{defx}
Inspired by Atiyah's approach \cite{atiyah2006elliptic}, the second definition (\textit{Definition} \ref{defB}) focuses on the injectivity of symbols for the operator $P$ and its adjoint $P^{\ast}$ when combined with the leaf-wise elliptic Laplacian operator. This perspective allows for the analytical index to be defined as a distribution over the group $G$, employing Atiyah's method. 

More explicitly, let $$H^s(E)_{\lambda}=\ker(\Delta_G-\lambda):H^s(E)\to H^{s-2}(E),$$
where $H^s(E)$ are ordinary Sobolev spaces and $\Delta_G$ is the orbital Laplacian operator we defined in Definition \ref{defn4}. Since $P$ commutes with the $G$-action, it commutes with $\Delta_G$, and especially, it commutes the $\lambda$-eigenspaces $\mathcal{D}(E)_{\lambda} \subset\mathcal{D}(E)$ and $\mathcal{D}(F)_{\lambda} \subset\mathcal{D}(F)$, where $\mathcal{D}(E)=C^{\infty}_c(E)$. We obtain a linear map $P_{\lambda}:\mathcal{D}(E)_{\lambda}\to\mathcal{D}(F)_{\lambda}$ and it extends to $$P_{\lambda}:H^s(E)_{\lambda}\to H^{s-k}(F)_{\lambda}.$$
    Similarly, we can construct the operator $P_{\lambda}^{\ast}$ from $P^{\ast}$. Based on the condition in (\textit{Definition} \ref{defB}), such $P_{\lambda}$ and $P^{\ast}_{\lambda}$ are Fredholm for each $\lambda$. Its index is defined as 
    $${Ind}(P_{\lambda}) = \chi_{\ker(P_{\lambda})}-\chi_{\ker(P^{\ast}_{\lambda})},$$
    where $\chi_{\ker(P_{\lambda})}$ and $\chi_{\ker(P^{\ast}_{\lambda})}$ are defined to be distributions over $G$. The analytical index of $P$ is then defined as
    $${Ind}(P)=\sum_{\lambda}{Ind}(P_{\lambda})$$
    as an equality of distributions on $G$. It will be possible that for operators in Definition \ref{defB}, we can define the cohomological index and prove the index theorem that the cohomological index is equal to the analytical index.
    
    However, the $\Delta_G$ operator cannot be killed by certain special $\pi_{\xi}$. So for an operator $P$ satisfying the condition in Definition \ref{defB}, it is difficult to separate the $(P,\Delta_G)$ or $(P^{\ast},\Delta_G)$ in the condition to obtain the operator $P$ itself. Therefore, it is hard to define the symbol class and index class without $\Delta_G$ in $KK$-homology. The analysis of operators in this definition in $KK$-theory seems extremely difficult.
\subsection{An intuitive definition}
In Chapter \ref{chap4}, we have defined the bundle isomorphism $\psi: TM\to gr(H)$, which maps a vector $v\in T_x M$ decomposed as $v=\sum_i v^i$  to $\psi(v)=\sum_i [v^i]\in gr(H)_x$, with each $v^i\in H^i_x$ and all $v^i$'s are orthogonal to each other. We obtain the adjoint map $\psi^{\ast}:gr(H)^{\ast}\to T^{\ast}M$. The image of transversal part $T^{\ast}_{G,x} M$ by $(\psi^{\ast})^{-1}$ is denoted as $U^{\ast}_{G,x}\subseteq gr(H)_x^{\ast}$. Then, we define $$O^{\ast}_{G,x}=\{\xi\in gr(H)^{\ast}_x\ |\text{ The coadjoint orbit containing } \xi \text{ passing through } U^{\ast}_{G,x}\}.$$
It can be understood as the combination of all the orbits containing a point in $U_{G,x}^{\ast}$.

The third possible definition is stated as follows:
\begin{defx}\label{defC}
     $P$ is called transversally H-elliptic if for all $x\in M$, $\sigma(P,x,\xi)$ is bijective for $\pi\in O^{\ast}_{G,x}$, i.e., $\sigma(P,x,\pi_{\xi})$ is bijective if the coadjoint orbit corresponding to $\pi_{\xi}$ contains at least one point of $T^{\ast}_G M$.
\end{defx}

 The third definition (\textit{Definition} \ref{defC}) is the most straightforward one, directly aligning with the foundational concept of transversal ellipticity. However, its practical application within $KK$-theory is restricted by the lack of information on the representations out of $O^{\ast}_{G,x}$. There is also a problem that we cannot decompose coadjoint orbits of the nilpotent Lie group $Gr(H)_x$, similar to the decomposition of $T_x M$ to tangent part and transversal part, because $C^{\ast}(Gr(H))$ is non-commutative.

%% file: references.bib
@article{akemann1973multipliers,
  title={Multipliers of $C^{\ast}$-algebras},
  author={Akemann, Charles A and Pedersen, Gert K and Tomiyama, Jun},
  journal={Journal of Functional Analysis},
  volume={13},
  number={3},
  pages={277--301},
  year={1973},
  publisher={Elsevier}
}

@article{atiyah1968index1,
  title={The index of elliptic operators: I},
  author={Atiyah, Michael Francis and Singer, Isadore Manuel},
  journal={Annals of mathematics},
  pages={484--530},
  year={1968},
  publisher={JSTOR}
}

@article{atiyah1968index2,
  title={The index of elliptic operators: II},
  author={Atiyah, Michael F and Segal, Graeme B},
  journal={Annals of Mathematics},
  pages={531--545},
  year={1968},
  publisher={JSTOR}
}

@book{atiyah2006elliptic,
  title={Elliptic operators and compact groups},
  author={Atiyah, Michael Francis},
  volume={401},
  year={2006},
  publisher={Springer}
}

@article{androulidakis2022pseudodifferential,
  title={A pseudodifferential calculus for maximally hypoelliptic operators and the Helffer-Nourrigat conjecture},
  author={Androulidakis, Iakovos and Mohsen, Omar and Yuncken, Robert},
  journal={arXiv preprint arXiv:2201.12060},
  year={2022}
}

@article{berline1996chern,
  title={The Chern character of a transversally elliptic symbol and the equivariant index},
  author={Berline, Nicole and Vergne, Mich{\`e}le},
  journal={Inventiones mathematicae},
  volume={124},
  pages={11--49},
  year={1996},
  publisher={Springer}
}

@article{beltictua2017fourier,
  title={Fourier transforms of $C^{\ast}$-algebras of nilpotent Lie groups},
  author={Belti{\c{t}}{\u{a}}, Ingrid and Belti{\c{t}}{\u{a}}, Daniel and Ludwig, Jean},
  journal={International Mathematics Research Notices},
  volume={2017},
  number={3},
  pages={677--714},
  year={2017},
  publisher={Oxford University Press}
}

@book{blackadar1998k,
  title={$K$-theory for operator algebras},
  author={Blackadar, Bruce},
  volume={5},
  year={1998},
  publisher={Cambridge University Press}
}

@article{baum2014k,
  title={$K$-homology and index theory on contact manifolds},
  author={Baum, Paul F and Erp, Erik},
  year={2014}
}

@book{bourbaki1989lie,
  title={Lie groups and Lie algebras: Chapters 1-3},
  author={Bourbaki, Nicolas},
  volume={1},
  year={1989},
  publisher={Springer Science \& Business Media}
}

@book{connes1994noncommutative,
  title={Noncommutative geometry},
  author={Connes, Alain},
  year={1994},
  publisher={Springer}
}

@article{de1978index,
  title={On the index of Toeplitz operators of several complex variables},
  author={de Monvel, Louis Boutet},
  journal={Inventiones mathematicae},
  volume={50},
  number={3},
  pages={249--272},
  year={1978},
  publisher={Springer}
}

@misc{epstein2004heisenberg,
  title={The Heisenberg algebra, index theory and homology},
  author={Epstein, Charles L and Melrose, Richard B and Mendoza, Gerardo},
  year={2004},
  publisher={preparation}
}

@book{fischer2016quantization,
  title={Quantization on nilpotent Lie groups},
  author={Fischer, Veronique and Ruzhansky, Michael},
  year={2016},
  publisher={Springer Nature}
}

@article{fack1981connes,
  title={Connes' analogue of the Thom isomorphism for the Kasparov groups},
  author={Fack, Thierry and Skandalis, Georges},
  journal={Inventiones mathematicae},
  volume={64},
  pages={7--14},
  year={1981},
  publisher={Springer-Verlag}
}

@incollection{joachim2003unbounded,
  title={Unbounded Fredholm operators and $K$-theory},
  author={Joachim, Michael},
  booktitle={High-dimensional manifold topology},
  pages={177--199},
  year={2003},
  publisher={World Scientific}
}

@article{kasparov1988equivariant,
  title={Equivariant KK-theory and the Novikov conjecture},
  author={Kasparov, Gennadi G},
  journal={Inventiones mathematicae},
  volume={91},
  number={1},
  pages={147--201},
  year={1988},
  publisher={Springer}
}

@article{kasparov2017elliptic,
  title={Elliptic and transversally elliptic index theory from the viewpoint of $KK$-theory},
  author={Kasparov, Gennadi},
  journal={Journal of Noncommutative Geometry},
  volume={10},
  number={4},
  pages={1303--1378},
  year={2017}
}

@article{kasparov2022k,
  title={$K$-theory and index theory on manifolds with a proper Lie group action},
  author={Kasparov, Gennadi},
  journal={arXiv preprint arXiv:2210.02332},
  year={2022}
}

@misc{kasparov2024coarsepseudodifferentialcalculusindex,
      title={Coarse pseudo-differential calculus and index theory on manifolds with a tangent Lie structure}, 
      author={Gennadi Kasparov},
      year={2024},
      eprint={2409.19002},
      archivePrefix={arXiv},
      primaryClass={math.DG},
      url={https://arxiv.org/abs/2409.19002}, 
}

@book{kirillov2004lectures,
  title={Lectures on the orbit method},
  author={Kirillov, Aleksandr Aleksandrovich},
  volume={64},
  year={2004},
  publisher={American Mathematical Soc.}
}

@article{lescure2017convolution,
  title={About the convolution of distributions on groupoids},
  author={Lescure, Jean-Marie and Manchon, Dominique and Vassout, St{\'e}phane},
  journal={Journal of Noncommutative Geometry},
  volume={11},
  number={2},
  pages={757--789},
  year={2017}
}

@article{melin1983parametrix,
  title={Parametrix constructions for right invariant differential operators on nilpotent groups},
  author={Melin, Anders},
  journal={Annals of Global Analysis and Geometry},
  volume={1},
  pages={79--130},
  year={1983},
  publisher={Springer}
}

@article{mohsen2020index,
  title={Index theorem for inhomogeneous hypoelliptic differential operators},
  author={Mohsen, Omar},
  journal={arXiv preprint arXiv:2001.00488},
  year={2020}
}

@article{mohsen2022index,
  title={On the index of maximally hypoelliptic differential operators},
  author={Mohsen, Omar},
  journal={arXiv preprint arXiv:2201.13049},
  year={2022}
}

@article{nistor2003index,
  title={An index theorem for gauge-invariant families: The case of solvable groups},
  author={Nistor, Victor},
  journal={Acta Mathematica Hungarica},
  volume={99},
  number={1-2},
  pages={155--183},
  year={2003},
  publisher={Springer}
}

@book{renault2006groupoid,
  title={A groupoid approach to $C^{\ast}$-algebras},
  author={Renault, Jean},
  volume={793},
  year={2006},
  publisher={Springer}
}

@article{skandalis1984some,
  title={Some remarks on Kasparov theory},
  author={Skandalis, Georges},
  journal={Journal of functional analysis},
  volume={56},
  number={3},
  pages={337--347},
  year={1984},
  publisher={Academic Press}
}

@article{van2010atiyah1,
  title={The Atiyah-Singer index formula for subelliptic operators on contact manifolds. Part I},
  author={van Erp, Erik},
 journal={Annals of mathematics},
  pages={1647--1691},
  year={2010},
  publisher={JSTOR}
}

@article{van2010atiyah2,
  title={The Atiyah-Singer index formula for subelliptic operators on contact manifolds. Part II},
  author={van Erp, Erik},
  journal={Annals of mathematics},
  pages={1683--1706},
  year={2010},
  publisher={JSTOR}
}

@article{van2017tangent,
  title={On the tangent groupoid of a filtered manifold},
  author={van Erp, Erik and Yuncken, Robert},
  journal={Bulletin of the London Mathematical Society},
  volume={49},
  number={6},
  pages={1000--1012},
  year={2017},
  publisher={Wiley Online Library}
}

@article{van2019groupoid,
  title={A groupoid approach to pseudodifferential calculi},
  author={van Erp, Erik and Yuncken, Robert},
  journal={Journal f{\"u}r die reine und angewandte Mathematik (Crelles Journal)},
  volume={2019},
  number={756},
  pages={151--182},
  year={2019},
  publisher={De Gruyter}
}
